\documentclass[12pt]{amsart}
\pdfoutput=1
\usepackage{graphicx}
\usepackage{caption}
\usepackage{subcaption}
\usepackage{url}
\usepackage{hyperref}
\usepackage{amsfonts,amsthm,latexsym,amsmath,amssymb,amscd,amsmath,epsf} 
\usepackage{tikz}
\usepackage{verbatim}

\usetikzlibrary{snakes,arrows,shapes}

\theoremstyle{plain}
\newtheorem{theorem}{\bf Theorem}[section]
\newtheorem{lemma}[theorem]{\bf Lemma}
\newtheorem{proposition}[theorem]{\bf Proposition}
\newtheorem{corollary}[theorem]{\bf Corollary}
\newtheorem{conjecture}[theorem]{\bf Conjecture}
\newtheorem{question}[theorem]{\bf Question}

\theoremstyle{definition}
\newtheorem{definition}[theorem]{Definition}
\newtheorem{example}[theorem]{Example}

\theoremstyle{remark}
\newtheorem{remark}[theorem]{Remark}

\numberwithin{equation}{section}

\makeindex

\newcommand{\ZZ}{\mathbb{Z}}
\newcommand{\NN}{\mathbb{N}}
\newcommand{\PP}{\mathbb{P}}
\newcommand{\QQ}{\mathbb{Q}}

\newcommand{\kk}{{\bf k}}

\newcommand{\E}{\mathcal{E}}

\newcommand{\Q}{\mathcal{Q}}

\newcommand{\T}{\mathcal{T}}
\newcommand{\BT}{\mathcal{BT}}

\newcommand{\redu}{\mathfrak{red}}
\newcommand{\vr}{v^{\mathfrak{r}}}
\newcommand{\init}{\mathfrak{init}}

\newcommand{\MM}{\mathcal{M}}

\newcommand{\sym}{\mathfrak{S}}
\newcommand{\lie}{\mathcal{L}ie}
\newcommand{\com}{\mathcal{C}om}
\newcommand{\clr}{\mathbf{color}}
\newcommand{\xx}{\mathbf{x}}
\newcommand{\yy}{\mathbf{y}}

\newcommand{\comb}{\textsf{Comb}}
\newcommand{\lyn}{\textsf{Lyn}}

\newcommand{\nor}{\textsf{Nor}}
\newcommand{\Par}{\textsf{Par}}
\newcommand{\tlambda}{\tilde \lambda}
\newcommand{\lyndonlambda}{\lambda^{\textsf{Lyn}}}
\newcommand{\comblambda}{\lambda^{\textsf{Comb}}}
\newcommand{\aalambda}{\lambda^{\textsf{AA}}}
\newcommand{\tnlambda}{\lambda^{\textsf{TN}}}

\def\newop#1{\expandafter\def\csname #1\endcsname{\mathop{\rm #1}\nolimits}}

\newop{diag}
\newop{Tor}
\newop{supp}
\newop{per}
\newop{rank}
\newop{ker}
\newop{cl}
\newop{step}
\newop{col}
\newop{red}
\newop{blue}
\newop{length}
\newop{sgn}
\newop{val}
\newop{c}
\newop{root}
\newop{inv}
\newop{im}
\newop{couple}
\newop{free}
\newop{asc}
\newop{des}
\newop{red}
\newop{ch}
\newop{gr}
\newop{p}
\newop{final}
\newop{rgp}
\newop{row}
\newop{clm}
\newop{wcomp}
\newop{valinv}
\newop{colinv}
\newop{re}

\title[Multibracketed Lie algebras]{On the free Lie algebra
with multiple brackets}

\author[R. S. Gonz\'alez D'Le\'on]{Rafael S. Gonz\'alez D'Le\'on}
\address{Department of Mathematics, University of Kentucky, Lexington, KY 40506}
\email{rafaeldleon@uky.edu}
\thanks{Supported by NSF Grant  DMS 1202755}

\begin{document}
\allowdisplaybreaks

\begin{abstract}
It is a classical result that the multilinear component of the 
free Lie algebra is isomorphic (as a representation of 
the symmetric group) to the top (co)homology of the proper part of the poset of partitions $\Pi_n$ 
tensored with the sign representation. We generalize this result in order to study the multilinear 
component of the free Lie algebra with multiple compatible Lie brackets.  
We introduce a new poset of weighted 
partitions $\Pi_n^k$ that allows us to generalize the result. The new poset is a generalization 
of $\Pi_n$ and of the poset of weighted partitions $\Pi_n^w$ introduced by Dotsenko and Khoroshkin 
and studied by the author and Wachs for the case of two compatible brackets. We prove that the 
poset $\Pi_n^k$ with a top element added is EL-shellable and hence Cohen-Macaulay. 
This and other properties of $\Pi_n^k$ enable us to answer 
questions posed by Liu on free multibracketed Lie algebras. In 
particular, we obtain various dimension formulas and multicolored generalizations of 
the classical Lyndon and comb bases for the multilinear component of the free Lie 
algebra. 
We also obtain a plethystic formula for the 
Frobenius characteristic of the representation of the symmetric group on the multilinear component 
of the free multibracketed Lie algebra.
\end{abstract}

\maketitle

\tableofcontents
\newpage
\section{Introduction}\label{section:introduction}
Recall that a \emph{Lie bracket} on a  vector
space $V$ is a bilinear binary product  
$[\cdot,\cdot]:V\times V \rightarrow V$ such that for all $x,y,z \in V$,
\begin{align}
 [x,y] + [y,x]&=0\quad\quad& \text{(Antisymmetry),}\label{relation:lb1}\\
[x,[y,z]]+[z,[x,y]]+[y,[z,x]]&=0 \quad\quad&\text{(Jacobi Identity).}\label{relation:lb2}
\end{align}
Throughout this paper let $\kk$ denote an arbitrary field. 
The \emph{free Lie algebra on $[n]:=\{1,2,\dots,n\}$} (over the field ${\bf k}$) is the  ${\bf 
k}$-vector space
generated by the elements
of $[n]$ and all the possible bracketings involving these elements subject only to the relations
(\ref{relation:lb1}) and (\ref{relation:lb2}).
Let $\lie(n)$ denote the \emph{multilinear} component of the free Lie algebra on $[n]$, i.e., 
the subspace generated by bracketings that contain each
element of $[n]$ exactly once.  We call these bracketings \emph{bracketed permutations}. For
example
$[[2,3],1]$ is a bracketed permutation in $\lie(3)$, while $[[2,3],2]$ is not. For any set $S$, 
the symmetric group $\sym_S$ is the group of permutations of $S$.  In particular we denote by
$\sym_n:=\sym_{[n]}$ the group of permutations of the set $[n]$. The symmetric
group
$\sym_n$ acts naturally on $\lie(n)$ making it into an $\sym_n$-module.  A permutation 
$\tau \in \sym_n$ acts on the bracketed permutations by replacing each letter $i$ by
$\tau(i)$. For example $(1,2)\,\,[[[ 3,5 ], [2,4] ],1]= [[[ 3,5 ], [1,4] ],2]$.  Since this action 
respects the relations (\ref{relation:lb1}) and (\ref{relation:lb2}), it induces a 
representation of $\sym_n$ on $\lie(n)$. It is a classical result that
\begin{align*}
 \dim \lie(n) = (n-1)!.
\end{align*}

Although the $\sym_n$-module $\lie(n)$ is an algebraic object it turns out that the information 
needed to 
completely describe 
this object is of combinatorial nature. Let $P$ denote a \emph{partially ordered set} (or 
\emph{poset} for short). To every poset $P$ one can associate a simplicial complex 
$\Delta(P)$ (called the \emph{order complex}) whose faces are the chains (totally ordered subsets) 
of $P$.
Consider now the poset $\Pi_n$ of set partitions of $[n]$ ordered by refinement. The
symmetric group $\sym_n$ acts naturally on $\Pi_n$ and this action
induces isomorphic representations of $\sym_n$ on the unique nonvanishing reduced simplicial
homology $\widetilde H_{n-3}(\overline{\Pi}_n)$ and cohomology 
$\widetilde H^{n-3}(\overline{\Pi}_n)$  of  the order complex $\Delta(\overline{\Pi}_n)$ of
the
proper part $\overline{\Pi}_n:=\Pi_n \setminus \{\hat{0},\hat{1}\}$ of $\Pi_n$.
It is a classical result that

\begin{align} 
\label{equation:introlie}  \lie(n)\simeq_{\sym_n} \widetilde H_{n-3}(\overline{\Pi}_n)  \otimes 
\sgn_n,
\end{align}  
where  $\sgn_n$ is the sign representation of $\sym_n$.

Equation (\ref{equation:introlie}) was observed by Joyal \cite{Joyal1986}
by comparing a computation of the character of $\widetilde H_{n-3}(\overline{\Pi}_n)$ by 
Hanlon and Stanley (see \cite{Stanley1982}), to an earlier formula of  Brandt \cite{Brandt1944} 
for the character of  $\lie(n)$.
 Joyal \cite{Joyal1986} gave a proof of the isomorphism using his theory of species.
The first purely combinatorial proof was obtained by Barcelo \cite{Barcelo1990} who provided  a
bijection between known bases for the two $\sym_n$-modules  (Bj\"orner's NBC basis for  $\widetilde
H_{n-3}(\overline{\Pi}_n)$ and the Lyndon basis for $\lie(n)$).  Later Wachs \cite{Wachs1998} gave a 
more
general combinatorial proof by providing a natural bijection between generating sets of $\widetilde
H^{n-3}(\overline{\Pi}_n)$ and $\lie(n)$, which revealed the strong connection between the two
$\sym_n$-modules. Connections between Lie type structures and  various types of partition posets 
have  been studied 
in 
other places in the literature, see for example \cite{BarceloBergeron1990}, \cite{Bergeron1991}, 
 \cite{HanlonWachs1995}, \cite{GottliebWachs2000}, \cite{Fresse2004}, \cite{Vallette2007}, 
 \cite{ChapotonVallette2006},  \cite{LodayVallette2012}.

The moral of equation (\ref{equation:introlie}) is that we can describe $\lie(n)$ and understand 
its 
algebraic properties by studying and applying poset theoretic techniques to the 
combinatorial object $\Pi_n$. This observation will play a central role throughout this paper.

\subsection{Doubly bracketed Lie algebra}

Two Lie brackets  ${\color{blue}[}\bullet,\bullet{\color{blue}]_{1}}$ and
${\color{red}[}\bullet,\bullet{\color{red}]_{2}}$ on a vector space $V$ are said to be 
\emph{compatible} if any linear
combination of the brackets is also a Lie bracket on $V$, that is, satisfies relations 
(\ref{relation:lb1}) 
and (\ref{relation:lb2}). 
As pointed out in \cite{DotsenkoKhoroshkin2007,Liu2010}, this kind of compatibility is
equivalent to the \emph{mixed Jacobi} condition: for all $x,y,z \in V,$
\begin{align}
{\color{blue}[}x,{\color{red}[} y, z
{\color{red}]_2}{\color{blue}]_1}+{\color{blue}[}z,{\color{red}[} x,y
{\color{red}]_2}{\color{blue}]_1}+{\color{blue}[}y,{\color{red}[} z,
x{\color{red}]_2}{\color{blue}]_1}&+\quad\quad&\text{(Mixed Jacobi)}\label{relation:lb3}\\
{\color{red}[} x,{\color{blue}[}y,z{\color{blue}]_1}
{\color{red}]_2} + 
{\color{red}[} z,{\color{blue}[}x,y{\color{blue}]_1} {\color{red}]_2}+{\color{red}[}
y,{\color{blue}[}z,x{\color{blue}]_1}{\color{red}]_2}&=0.&\nonumber
\end{align}
Let $\lie_2(n)$ be the multilinear component of the free Lie algebra on $[n]$ with two compatible
brackets,
that is, the multilinear component of the  ${\bf k}$-vector space generated by (mixed) bracketings
of elements of $[n]$  
subject only to the five  relations given by (\ref{relation:lb1}) and (\ref{relation:lb2}), for 
each bracket,
and 
(\ref{relation:lb3}). For each $i$, let $\lie_2(n,i)$ be the subspace of $\lie_2(n)$ generated by 
bracketed
permutations with exactly $i$ brackets of the first type and $n-1-i$ brackets of the second type.
The symmetric group $\sym_n$ acts naturally on $\lie_2(n)$ and since this
action preserves the number of brackets of each type, we have the following decomposition into
$\sym_n$-submodules: 
\begin{align*}
 \lie_2(n)=\bigoplus_{i=0}^{n-1} \lie_2(n,i).
\end{align*}

Note that interchanging the roles of the two brackets makes evident the $\sym_n$-module isomorphism

\begin{align*}
 \lie_2(n,i) \simeq_{\sym_n}\lie_2(n,n-1-i)
\end{align*}
for every $i$. Also note that in particular $\lie(n)$ is isomorphic to the submodules $\lie_2(n,i)$ 
when $i=0$ or $i=n-1$.

It was conjectured by Feigin and proved independently by Dotsenko-Khoroshkin
\cite{DotsenkoKhoroshkin2007} and Liu \cite{Liu2010} that
\begin{align}\label{equation:dimensionlie2}
 \dim \lie_2(n) = n^{n-1}.
\end{align}

In \cite{Liu2010}  Liu proves the conjecture by constructing a combinatorial basis for $\lie_2(n)$ 
indexed by rooted trees giving as a byproduct the refinement
\begin{align} \label{equation:dimensionlie2i}
 \dim \lie_2(n,i) = |\T_{n,i}|,
\end{align} 
where $\T_{n,i}$ is the set of rooted trees on vertex set $[n]$ with $i$ descending edges 
(a parent with a
greater label than its child).

The Dotsenko-Khoroshkin proof \cite{DotsenkoKhoroshkin2007,DotsenkoKhoroshkin2010} of Feigin's 
conjecture  
 was operad-theoretic; they used a pair of 
functional equations that apply to Koszul operads to compute the $\mathit{SL}_2  \times 
\sym_n$-character 
of $\lie_2(n)$.
They also proved that the dimension generating polynomial has a nice factorization:
\begin{align} \label{equation:introprod}
\sum_{i=0}^{n-1}\dim \lie_2(n,i)t^i=\prod_{j=1}^{n-1}((n-j)+jt).
\end{align}

Since, as was proved by Drake \cite{Drake2008},  the  right hand side of (\ref{equation:introprod})
is equal to the generating function for rooted trees on node set $[n]$ according to the number of 
descents of the tree, it follows that for each $i$,
the dimension of $\lie_2(n,i)$ equals  the number of rooted trees on node set
$[n]$ with $i$ descents.  (Drake's result is a refinement of the well-known result that the number
of trees on node set $[n]$ is $n^{n-1}$.) 

Although Dotsenko and Khoroshkin \cite{DotsenkoKhoroshkin2007} did not use poset theoretic
techniques in their ultimate proof of (\ref{equation:dimensionlie2}), they introduced the poset of 
weighted partitions $\Pi_n^w$ as a possible approach to establishing Koszulness of the operad
associated with $\lie_2(n)$, a key step in their proof. In \cite{DleonWachs2013a}
Wachs and the author applied poset theoretic techniques to the poset of weighted
partitions to give an alternative proof of (\ref{equation:dimensionlie2}) and 
(\ref{equation:dimensionlie2i}) and to obtain
further results on $\lie_2(n)$. 

The poset $\Pi_n^w$ has a minimum element $\hat 0:= \{\{1\}^0,\{2\}^0,\dots, \{n\}^0\}$ and  $n$
maximal elements $\{[n]^0\}, \, \{[n]^1\}, \dots,  \{[n]^{n-1}\}$. For each $i$, the maximal 
intervals $[\hat 0, [n]^i]$ and $[\hat 0, [n]^{n-1-i}]$ are isomorphic to each other, and the two
maximal 
intervals $[\hat 0, [n]^0]$ and $[\hat 0, [n]^{n-1}]$ are isomorphic to $\Pi_n$. 

In \cite{DleonWachs2013a} the authors found a nice EL-labeling of $\Pi_n^w \cup \{\hat{1}\}$ that 
generalized a 
classical 
EL-labeling of $\Pi_n$ due to Bj\"orner and Stanley (see \cite{Bjorner1980}).
An \emph{EL-labeling} of a poset (defined in Section~\ref{section:ellabeling}) is a labeling of the 
edges 
of the Hasse 
diagram of the poset that
satisfies certain requirements.  Such a  labeling has important topological and algebraic
consequences, such as the determination of the homotopy type of each open interval of the poset.  
The so
called ascent-free maximal chains give a basis for cohomology of the open intervals. A poset that 
admits an
EL-labeling is said to be \emph{EL-shellable}.  See \cite{Bjorner1980}, \cite{BjornerWachs1983} and 
\cite{Wachs2007} for 
further
information.

\begin{theorem}[Theorem 3.2, Corollary 3.5 and Theorem 3.6
\cite{DleonWachs2013a}]\label{theorem:el2}
 The poset $\widehat{\Pi_n^w}:=\Pi_n^w \cup \{\hat{1}\}$ is
EL-shellable and hence Cohen-Macaulay. Consequently, for each $i=0,\dots,n-1$,
the order complex $\Delta((\hat{0},[n]^i))$ has the homotopy type of a wedge of
$ |\T_{n,i}|$ spheres.
\end{theorem}

The 
 symmetric group acts naturally on each $\lie_2(n,i)$ and on each open interval $(\hat 0, [n]^i)$.  
In \cite{DleonWachs2013a} Gonz\'alez D'Le\'on and Wachs give an explicit $\sym_n$-module 
isomorphism which establishes
\begin{align} \label{equation:introlie2}  \lie_2(n,i) 
\simeq_{\sym_n} \widetilde H^{n-3}((\hat 0,
[n]^i)) \otimes \sgn_n .\end{align}
For $i=0$ or $i=n-1$, equation (\ref{equation:introlie2}) reduces to equation 
(\ref{equation:introlie}) and the isomorphism reduces to the one in \cite{Wachs1998}.
 They construct bases for $\widetilde 
H^{n-3}((\hat 
0,
[n]^i))$ and for $\lie_2(n,i)$ that generalize the classical Lyndon tree basis and 
the comb basis for
$\widetilde H^{n-3}(\overline{\Pi}_n)$ and $\lie(n)$. In particular, the general Lyndon basis is
obtained from the ascent-free maximal chains of the EL-labeling of Theorem \ref{theorem:el2}.
They also define a basis for $\widetilde H_{n-3}((\hat 0,[n]^i))$ in terms of labeled rooted 
trees 
that
generalizes the Bj\"orner NBC basis for homology of $\overline{\Pi}_n$ (see
\cite[Proposition~2.2]{Bjorner1982}).   
\subsection{Multibracketed Lie algebras}
Liu posed the following natural question.

\begin{question}[Liu \cite{Liu2010}, Question 11.7 ] \label{question:liu}
Is it possible to define $\lie_k(n)$ for any $k \ge 1$ so
that it has nice dimension formulas like those for $\lie(n)$ and $\lie_2(n)$? What are the right
combinatorial objects for $\lie_k(n)$, if it can be defined?
\end{question}
The results developed in this paper provide an  answer to this question.

Let $\NN$ denote the set of nonnegative integers and $\PP$ the set of positive
integers. We say that a set $B$ of Lie brackets on a vector space is \emph{compatible} if any 
linear combination of 
the brackets in $B$ is a Lie bracket. We now consider compatible Lie brackets $[\cdot,\cdot]_j$ 
indexed by positive integers $j 
\in \PP$ and define $\lie_{\PP}(n)$ to be the multilinear
component of the free multibracketed Lie algebra on $[n]$;
that is, the ${\bf k}$-vector space generated by (mixed) bracketed
permutations of $[n]$ subject only to the relations given by (\ref{relation:lb1}) and 
(\ref{relation:lb2}), for each bracket,
and the compatibility relations for any set of brackets. For example, 
$[[[2,5]_{2},3]_{1},[1,4]_{1}]_{3}$ is a generator of $\lie_{\PP}(5)$

A \emph{weak composition} $\mu$ of $n$  
is a sequence of nonnegative
integers $(\mu(1),\mu(2),\dots)$ such that $|\mu|:=\sum_{i\ge 1}\mu(i)=n$. 
Let $\wcomp$ be the set of weak compositions and $\wcomp_n$ the set of weak compositions of $n$. 
For $\mu \in \wcomp_{n-1}$, define $\lie(\mu)$ to be the subspace of $\lie_{\PP}(n)$
generated by bracketed permutations of $[n]$ with $\mu(j)$ 
brackets of type $j$ for each $j$. For example $\lie(0,1,2,0,1)$ is generated by bracketed
permutations of $[5]$ that
contain one bracket of type $2$, two brackets of type $3$, one bracket of type $5$ and no brackets
of any other type. 

As before, $\sym_n$ acts naturally on $\lie(\mu)$ by replacing the letters of a bracketed 
permutation. Interchanging the roles of the brackets reveals that for every $\nu, \mu \in 
\wcomp_{n-1}$, such 
that $\nu$ is a rearrangement of $\mu$, we have that $\lie(\nu)\simeq_{\sym_n}\lie(\mu)$.
In particular, if $\mu$ has a single nonzero component, $\lie(\mu)$ is isomorphic to $\lie(n)$. 
If $\mu$ has at most two nonzero components then $\lie(\mu)$ is isomorphic to $\lie(n,i)$ for some 
$0 \le i \le n-1$.

For $\mu \in \wcomp_n$ define its \emph{support} $\supp(\mu)=\{j\in \PP \,\mid\,
\mu(j) \ne 0\}$ and for a subset $S \subseteq \PP$ let
\begin{align*}
 \lie_S(n):=\bigoplus_{\substack{\mu \in \wcomp_{n-1}\\ \supp(\mu) \subseteq S}}
\lie(\mu).
\end{align*}
Note that $\lie_k(n):=\lie_{[k]}(n)$ generalizes $\lie(n)=\lie_1(n)$ and
$\lie_2(n)$.

The isomorphisms (\ref{equation:introlie}) and (\ref{equation:introlie2}) provide a way to study the
algebraic objects $\lie(n)$ and $\lie_2(n)$ by applying poset topology techniques to $\Pi_n$ and
$\Pi_n^w$. In particular the dimensions of the modules can be read from the structure of the
posets and the bases for the cohomology of the posets can be directly translated into bases of
$\lie(n)$ and $\lie_2(n)$.
It is then natural to look for a poset whose cohomology allows us to analyze 
$\lie_k(n)$.

\subsection{The poset of weighted partitions}
We introduce a more general poset of weighted partitions $\Pi_n^k$ where
the weights are given by weak compositions supported in $[k]$. A \emph{(composition)-weighted 
partition}
of
$[n]$ is a set $\{B_1^{\mu_1},B_2^{\mu_2},...,B_t^{\mu_t}\}$ where $\{B_1,B_2,...,B_t\}$ is a set
partition of $[n]$ and $\mu_i \in \wcomp_{|B_i|-1}$ with $\supp(\mu_i) \subseteq [k]$. 
For $\nu,\mu
\in \wcomp$, we say that $\mu\le \nu$ if $\mu(i)\le \nu(i)$
for every $i$. Since weak compositions are infinite vectors we can use component-wise addition 
and subtraction, for instance, we denote by $\nu + \mu$, the weak composition defined by 
$(\nu + \mu)(i):=\nu(i) + \mu(i)$.

 The \emph{poset of weighted partitions} $\Pi_n^{k}$  is the set of weighted partitions of
$[n]$ with  order relation given by 
$\{A_1^{\mu_1},A_2^{\mu_2},...,A_s^{\mu_s}\}\le\{B_1^{\nu_1}, B_2^{\nu_2},...,B_t^{\nu_t}\}$ if the 
following
conditions hold:

\begin{itemize}
 \item $\{A_1,A_2,...,A_s\} \le \{B_1,B_2,...,B_t\}$ in $\Pi_n$ and,
 \item If $B_j=A_{i_1}\cup A_{i_2}\cup ... \cup A_{i_l} $ then 
 $\nu_j \ge (\mu_{i_1} + \mu_{i_2} + ... + \mu_{i_l})$ and $|\nu_j-(\mu_{i_1} + \mu_{i_2} + ... +
\mu_{i_l})|=l-1$
\end{itemize}

Equivalently, we can define the covering relation  
$\{A_1^{\mu_1},A_2^{\mu_2},...,A_s^{\mu_s}\} \lessdot \{B_1^{\nu_1}, 
B_2^{\nu_2},...,B_t^{\nu_t}\}$ by:

\begin{itemize}
 \item $\{A_1,A_2,\dots,A_s\} \lessdot \{B_1,B_2,\dots,B_t\}$ in $\Pi_n$
 \item if $B_j=A_{i_1}\cup A_{i_2}$ then $\nu_j- (\mu_{i_1} + \mu_{i_2})= \bf{e}_r$ for some $r \in 
[k]$, where
$\bf{e}_r$ is
the weak composition
 with a $1$ in the $r$-th component and $0$ in all other entries.
 \item  if $B_k = A_i$ then $\nu_k = \mu_i$.

 \end{itemize}
 
In Figure~\ref{fign3k3}  below the set brackets and commas have been omitted.

\begin{figure}[ht]

\begin{center} 
\begin{tikzpicture}[line join=bevel,scale=0.8]

\tikzstyle{every node}=[inner sep=0pt, scale=0.6, minimum width=4pt]
 \node (n102030) at (0,0)  {$1^{(0,0,0)}| 2^{(0,0,0)}| 3^{(0,0,0)}$};

  \node (n12i30) at (-8,2) {$12^{(1,0,0)}| 3^{(0,0,0)}$};
  \node (n13i20) at (-6,2) {$13^{ (1,0,0)}| 2^{ (0,0,0)}$};
  \node (n1023i) at (-4,2)  {$1^{(0,0,0)}| 23^{(1,0,0)}$};

 \node (n12j30) at (-2,2)  {$12^{(0,1,0)}| 3^{(0,0,0)}$};
 \node (n13j20) at (0,2)  {$13^{(0,1,0)}| 2^{(0,0,0)}$};
  \node (n1023j) at (2,2) {$1^{(0,0,0)}| 23^{(0,1,0)}$};

\node (n12k30) at (4,2)  {$12^{ (0,0,1)}| 3^{(0,0,0)}$};
 \node (n13k20) at (6,2)  {$13^{(0,0,1)}| 2^{(0,0,0)}$};
  \node (n1023k) at (8,2) {$1^{(0,0,0)}| 23^{(0,0,1)}$};

 \node (n123ii) at (-6,4){$123^ {(2,0,0)}$};
 \node (n123ij) at (-3,4){$123^ {(1,1,0)}$};
 \node (n123ik) at (-1,4){$123^ {(1,0,1)}$};
 \node (n123jj) at (1,4){$123^ {(0,2,0)}$};
 \node (n123jk) at (3,4){$123^ {(0,1,1)}$};
 \node (n123kk) at (6,4){$123^ {(0,0,2)}$};

 \draw (n123ii) -- (n1023i) ;
 \draw (n123ii)-- (n13i20);
  \draw (n123ii) -- (n12i30);  

\draw (n123jj) -- (n1023j) ;
 \draw (n123jj)-- (n13j20);
  \draw (n123jj) -- (n12j30);  

\draw (n123kk) -- (n1023k) ;
 \draw (n123kk)-- (n13k20);
  \draw (n123kk) -- (n12k30); 

 \draw (n123ik) -- (n1023i) ;
 \draw (n123ik)-- (n13i20);
  \draw (n123ik) -- (n12i30);  
\draw (n123ik) -- (n1023k) ;
 \draw (n123ik)-- (n13k20);
  \draw (n123ik) -- (n12k30);

 \draw (n123ij) -- (n1023i) ;
 \draw (n123ij)-- (n13i20);
  \draw (n123ij) -- (n12i30);  
\draw (n123ij) -- (n1023j) ;
 \draw (n123ij)-- (n13j20);
  \draw (n123ij) -- (n12j30);
	
 \draw (n123jk) -- (n1023j) ;
 \draw (n123jk)-- (n13j20);
  \draw (n123jk) -- (n12j30);  
\draw (n123jk) -- (n1023k) ;
 \draw (n123jk)-- (n13k20);
  \draw (n123jk) -- (n12k30);

  \draw [] (n13i20) -- (n102030);
  \draw [] (n12i30)-- (n102030);
  \draw [] (n1023i)--(n102030);
 \draw [] (n1023j)  --  (n102030);
  \draw [] (n12j30) --  (n102030);
  \draw [] (n13j20) --(n102030);
 \draw [] (n1023k)  --  (n102030);
  \draw [] (n12k30) --  (n102030);
  \draw [] (n13k20) --(n102030);
\end{tikzpicture}
\end{center}
\caption[]{Weighted partition poset for $n=3$ and $k=3$}\label{fign3k3}
\end{figure}
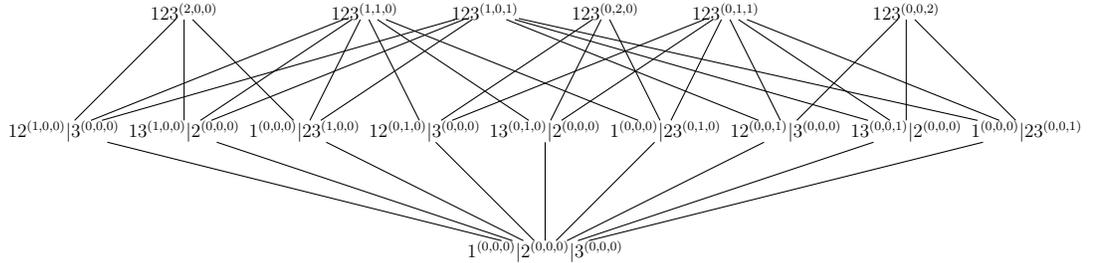

The poset $\Pi_n^k$ has a minimum element
\begin{align*}
\hat 0:=
\{\{1\}^{(0,\dots,0)},\{2\}^{(0,\dots,0)},\dots, \{n\}^{(0,\dots,0)}\} 
\end{align*}
and $\genfrac{(}{)}{0pt}{1}{k+n-2}{n-1}$
maximal elements 
\begin{align*}
\{[n]^\mu\}\text{ for } \mu \in \wcomp_{n-1} \text{ and } \supp(\mu)\subseteq [k].
\end{align*}
We write each maximal element $\{[n]^\mu\}$ as
$[n]^\mu$ for simplicity. Note that for every $\nu, 
\mu \in \wcomp_{n-1}$ with $\supp(\nu),\supp(\mu) \subseteq [k]$, such 
that $\nu$ is a rearrangement of $\mu$,
the maximal intervals 
$[\hat 0, [n]^{\nu}]$  and $[\hat 0, [n]^{\mu}]$ are isomorphic to each other.
In particular, if $\mu$ has a single nonzero component, these intervals are isomorphic to $\Pi_n$. 
If $\supp(\mu) 
\subseteq [2]$ then these intervals are isomorphic to maximal intervals of $\Pi_n^w$. 
Indeed, we can think of a composition $(i,n-1-i)$ as being the weight $i$ in the poset $\Pi_n^w$ in 
\cite{DleonWachs2013a}. Hence $\Pi_n^1 \simeq \Pi_n$ and $\Pi_n^2 \simeq 
\Pi_n^w$.

\subsection{Main results}
The symmetric group acts naturally on each open interval $(\hat 0, [n]^{\mu})$.
Using Wachs' technique in Section~\ref{section:isomorphism} we give an explicit isomorphism 
that proves the following theorem.
\begin{theorem}\label{theorem:liekiso}For $\mu \in \wcomp_{n-1}$,
 \begin{align*}  
 \lie(\mu) \simeq_{\sym_n}  \widetilde H^{n-3}((\hat 0, [n]^{\mu}))  \otimes \sgn_n . 
\end{align*}
\end{theorem}

Theorem \ref{theorem:liekiso} is a generalization of equations (\ref{equation:introlie}) and 
(\ref{equation:introlie2}). It reduces to equation (\ref{equation:introlie}) when 
$\supp(\mu)\subseteq [1]$ and to equation (\ref{equation:introlie2}) when $\supp(\mu)\subseteq [2]$.
We use Theorem \ref{theorem:liekiso} to give information about $\lie(\mu)$ by studying the 
algebraic and combinatorial properties of the poset $\Pi_n^k$.

In \cite{DotsenkoKhoroshkin2010} Dotsenko and Khoroshkin prove using operad-theoretic 
techniques that the operad related to $\lie_k(n)$ is Koszul. This implies using Vallette's theory  
of operadic partition posets \cite{Vallette2007} that the maximal intervals $[\hat 0, [n]^{\mu}]$ 
of $\Pi_n^k$ are Cohen-Macaulay. In Section~\ref{section:ellabeling} we prove a stronger 
property.

\begin{theorem}\label{theorem:elk}
 The poset $\widehat{\Pi_n^k}:= \Pi_n^k \cup \{\hat{1}\}$ is EL-shellable and hence Cohen-Macaulay. 
Consequently, for each
$\mu \in \wcomp_{n-1}$,
the order complex $\Delta((\hat{0},[n]^{\mu}))$ has the homotopy type of a wedge of
$(n-3)$-spheres.

\end{theorem}

Using Vallette's theory, Theorem \ref{theorem:elk} gives a 
new proof of the fact that the operads $\lie_k$ and ${}^k\com$ considered in 
\cite{DotsenkoKhoroshkin2010} are 
Koszul.

The set of ascent-free maximal chains of this EL-labeling provides a basis for $\widetilde 
H^{n-3}((\hat{0},[n]^{\mu}))$ and hence, by the isomorphism of Theorem \ref{theorem:liekiso}, 
also a basis for $\lie(\mu)$. This basis is a multicolored generalization of the classical Lyndon 
basis for $\lie(n)$. 
We also construct a multicolored generalization of the classical comb basis for $Lie(n)$ and use 
our multicolored Lyndon basis to show that our construction does indeed yield a basis for 
$\lie(\mu)$.

We consider the generating function
\begin{align}\label{definition:stirlingeuleriansymmetric}
L_{n}(\xx):=\sum_{\mu \in \wcomp_{n}}\dim \lie(\mu)\,\xx^{\mu},
\end{align}
where $\xx^{\mu}=x_1^{\mu(1)}x_2^{\mu(2)}\cdots$. Since for any rearrangement $\nu$ of $\mu$ it 
happens that $\lie(\nu)\simeq_{\sym_n}\lie(\mu)$ it follows that 
(\ref{definition:stirlingeuleriansymmetric}) belongs to the ring of 
symmetric functions $\Lambda_{\ZZ}$. The following theorem gives a 
characterization of this 
symmetric function.

\begin{theorem}\label{theorem:compositionalinverse}
We have
  \begin{align*}
   \sum_{n\ge1}\sum_{\mu \in \wcomp_{n-1}}\dim 
\lie(\mu)\,\xx^{\mu}\frac{y^n}{n!} =\left [ \sum_{n\ge1}(-1)^{n-1} 
h_{n-1}(\xx)\frac{y^n}{n!} \right ] ^{<-1>},
  \end{align*}
  where $h_n$  is the complete homogeneous symmetric function and $(\cdot)^{<-1>}$ denotes the 
compositional inverse of a formal power series.
 \end{theorem}

It follows from our construction of the multicolored Lyndon basis for $\lie(\mu)$ that the 
symmetric function $L_n(x)$ is $e$-positive; i.e., the coefficients of the expansion of $L_n(x)$ in 
the basis of elementary symmetric functions are all nonnegative.  We give various combinatorial 
interpretations of these coefficients in this paper. Two of the interpretations involve binary 
trees 
and two involve the Stirling permutations introduced by Gessel and Stanley in 
\cite{StanleyGessel1978}. We will now give one of the binary tree interpretations (Theorem 
\ref{theorem:dimensionscombtype}).  The others are given in Theorems 
\ref{theorem:dimensionslyndontype} and \ref{theorem:dimensionsofliekstirling}.

We say that a planar labeled binary 
tree with label set $[n]$ is 
\emph{normalized} 
if  the leftmost leaf of each subtree has the smallest label in the subtree. See Figure 
\ref{fig:normalized} for an example of a normalized tree and Section \ref{section:ascentfreechains} 
for the proper definitions. We denote the set of 
normalized binary trees with label set $[n]$ by $\nor_n$. 
\begin{figure}[ht]
  \begin{tikzpicture}[scale=0.5]
\tikzstyle{every node}=[draw,inner sep=1mm,scale=0.5]
\node[fill=blue!20,blue!20,draw,rectangle,rounded corners,rotate=-45, minimum width=280pt,minimum height=30pt,scale=0.8] at  (5.5,0.5) {};
\node[fill=blue!20,blue!20,draw,rectangle,rounded corners,rotate=-45, minimum width=130pt,minimum height=30pt,scale=0.8] at  (3,1) {};
\node[fill=blue!20,blue!20,draw,rectangle,rounded corners,rotate=-45, minimum width=30pt,minimum height=30pt,scale=0.8] at  (1,1) {};
\node[fill=blue!20,blue!20,draw,rectangle,rounded corners,rotate=-45, minimum width=30pt,minimum height=30pt,scale=0.8] at  (3,-1) {};

    \draw [circle] (1,1)  node (i1){$$};
    \draw [circle] (2,2)  node (i2){$$};
	   \draw [circle] (3,3)  node (i3){$$};
   \draw [circle] (7,-1)  node (i4){$$};
   \draw [circle] (8,-2)  node (i5){$$};
   \draw [circle] (4,0)  node (i6){$$};
   \draw [circle] (3,-1)  node (i7){$$};

\tikzstyle{every node}=[inner sep=1pt, minimum width=14pt,scale=0.7]

    \draw (0,0)  node (m){$1$};
    \draw (2,0)  node (l1){$3$};
    \draw (2,-2)  node (l2){$4$};
 	\draw (4,-2)  node (l3){$8$};
    \draw (5,-1)  node (l4){$5$};
  \draw (6,-2)  node (l5){$2$};
 	\draw (7,-3)  node (l6){$6$};
    \draw (9,-3)  node (l7){$7$};
  
    \draw (m) --  (i1) ;
    \draw (i1) --  (i2) ;
    \draw (i2) --  (i6) ;
    \draw (i6) --  (i7) ;
    \draw (i3) --  (i2) ;
	\draw (i3) --  (i4) ;
	\draw (i4) --  (i5) ;
    \draw (i1) --  (l1) ;
    \draw (i7) --  (l2) ;
	\draw (i7) --  (l3) ;
    \draw (i6) --  (l4) ;
	\draw (i4) --  (l5) ;
	\draw (i5) --  (l6) ;
	\draw (i5) --  (l7) ;

\end{tikzpicture}
 \caption{Example of a normalized tree}
 \label{fig:normalized}
\end{figure}
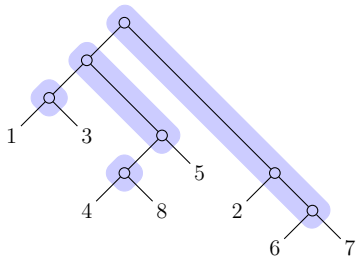

We associate a \emph{type} (or integer partition) to  each $\Upsilon \in \nor_n$ in the following 
way: Let $\pi^{\comb}(\Upsilon)$ be the finest (set) partition of the set of internal nodes of 
$\Upsilon$ satisfying
\begin{itemize}
\item for every pair of internal nodes $x$ and $y$ such that $y$ is a right child of $x$, 
$x$ 
and $y$ belong to the same block of $\pi^{\comb}(\Upsilon)$.
\end{itemize}
We define the \emph{comb type} $\comblambda(\Upsilon)$ of $\Upsilon$ to be the (integer) 
partition whose parts are the sizes of the blocks of $\pi^{\comb}(\Upsilon)$. In 
Figure \ref{fig:normalized} the associated partition is $\comblambda(\Upsilon)=(3,2,1,1)$.
The following theorem gives a direct method, alternative to Theorem 
\ref{theorem:compositionalinverse}, for computing the dimensions of $\lie(\mu)$.

\begin{theorem} \label{theorem:dimensionscombtype}
For all $n$,
\begin{align*}
  \sum_{\mu \in \wcomp_{n-1}}\dim \lie(\mu)\,\xx^{\mu}= \sum_{\Upsilon \in \nor_n} 
e_{\comblambda(\Upsilon)}(\xx),
 \end{align*} 
 where $e_{\lambda}$ is the elementary symmetric function associated with the partition 
$\lambda$.
 \end{theorem}

To prove Theorem \ref{theorem:dimensionscombtype} we use another normalized tree type 
$\lyndonlambda$, related to our colored Lyndon basis for $\lie(\mu)$, which come from the 
EL-labeling of $[\hat 0, [n]^\mu]$.  We use the colored Lyndon basis to show that Theorem 
\ref{theorem:dimensionscombtype} holds with $\comblambda$ replaced by $\lyndonlambda$.  We then 
construct a bijection on $\nor_n$ which takes $\lyndonlambda$ to $\comblambda$.   This bijection 
makes use of Stirling permutations and leads to two versions of Theorem 
\ref{theorem:dimensionscombtype}  involving Stirling permutations.

In terms of these combinatorial objects, the dimension of $\lie_k$ has a simple description as an 
evaluation of the symmetric 
function (\ref{definition:stirlingeuleriansymmetric}). 

\begin{corollary}
For all $n$ and $k$,
 \begin{align*}
  \dim \lie_k(n) = \sum_{\Upsilon \in \nor_{n}} e_{\comblambda(\Upsilon)}(\stackrel{k \text{ 
times}}{\overbrace{1,\dots,1}},0,0,\dots).
 \end{align*}
\end{corollary}

From equation (\ref{equation:introprod}), it follows that the polynomial $\sum_{i=0}^{n-1}\dim
\lie_2(n,i)\,t^i$ has only negative real roots and hence it has a property known as 
\emph{$\gamma$-positivity}, i.e, when written in the basis
$t^i(1+t)^{n-1-2i}$ it has positive coefficients. Note that 
this polynomial is actually $L_{n-1}(t,1,0,0,\dots)$.
The property of $\gamma$-positivity of this polynomial is a consequence of the $e$-positivity of 
$L_n(\xx)$.

A more general question is to understand the representation of $\sym_n$ on $\lie(\mu)$. The 
characters of the representation of $\sym_n$ on $\lie(n)$ and $\lie_2(n)$ were computed in
(\cite{Brandt1944, Stanley1982} and \cite{DotsenkoKhoroshkin2007}). Here we consider
\begin{align}\label{equation:generatingfunctionofcharacters}
 \sum_{\mu \in \wcomp_{n-1}}\ch \lie(\mu)\,\xx^{\mu},
\end{align}
where $\ch \lie(\mu)$ denotes the Frobenius characteristic in variables $\yy=(y_1,y_2,\dots)$ of
the representation $\lie(\mu)$.  The generating function of  
(\ref{equation:generatingfunctionofcharacters}) belongs to the ring  $\Lambda_{R}$ of 
symmetric functions in $\yy$ with coefficients in the ring of symmetric functions $R=\Lambda_{\QQ}$ 
in $\xx$. 
The following result generalizes Theorem \ref{theorem:compositionalinverse}.

\begin{theorem}\label{theorem:lierepresentation}
 We have that 
\begin{align*}
 \sum_{n\ge 1} \sum_{\mu \in \wcomp_{n-1}}\ch \lie(\mu)\,\xx^{\mu}=-\Bigl 
(-\sum_{n\ge
1}h_{n-1}(\xx)h_{n}(\yy)
\Bigr)^{[-1]},
\end{align*}
where $(\cdot)^{[-1]}$ denotes the plethystic inverse in 
 the ring of symmetric power series in $\yy$ 
with coefficients in the ring $\Lambda_{\QQ}$ of symmetric functions in $\xx$.
\end{theorem}

To prove Theorem \ref{theorem:lierepresentation} we use Theorem \ref{theorem:liekiso} and the  
Whitney (co)homology technique developed by Sundaram in \cite{Sundaram1994}, and further developed 
by Wachs in \cite{Wachs1999}.

The paper is organized as follows: 
In Section \ref{section:isomorphism} we describe generating sets of $\lie(\mu)$ and 
$\widetilde{H}^{n-3}((\hat{0},[n]^{\mu}))$  in terms of labeled binary 
trees with colored internal nodes. The description makes 
transparent the isomorphism of Theorem \ref{theorem:liekiso}, which we prove using Wachs' 
technique, as in \cite{Wachs1998} and \cite{DleonWachs2013a}.
In Section \ref{section:homotopytype} we use the 
recursive definition of the M\"obius invariant of $\Pi_n^k$ to prove an analogue of Theorem 
\ref{theorem:compositionalinverse} for the poset $\Pi_n^k$, that together with the results of 
Section \ref{section:isomorphism} imply Theorem \ref{theorem:compositionalinverse}. We also
prove Theorem \ref{theorem:elk} and we give a description of the 
ascent-free maximal chains of the EL-labeling.
Theorems \ref{theorem:compositionalinverse} and the 
version of Theorem 
\ref{theorem:dimensionscombtype} in which
$\comblambda$ is replaced by $\lyndonlambda$, are presented in Section \ref{section:binarystirling} 
as corollaries of results in the previous chapters.
We prove 
Theorem \ref{theorem:dimensionscombtype} and we use the language of Stirling permutations to give
two additional combinatorial descriptions of the dimension of $\lie(\mu)$.
In Section \ref{section:combinatorialbases} we summarize some of the results of 
Sections \ref{section:homotopytype} and \ref{section:binarystirling} on the colored Lyndon basis 
and we present the colored comb basis for $\lie(\mu)$ and 
$\widetilde{H}^{n-3}((\hat{0},[n]^{\mu}))$. We also discuss bases for 
$\widetilde{H}^{n-2}(\Pi_n^k\setminus\{\hat{0}\})$ in terms of the two families of colored 
binary trees.  We 
present in Section \ref{section:whitneynumbers} results 
on Whitney numbers of the first and second kind and on Whitney cohomology. In section 
\ref{section:frobeniuscharacteristic} we prove 
Theorem \ref{theorem:lierepresentation}.

Some of the results in this work are generalizations of results in
\cite{DleonWachs2013a} and we refer the reader to that article for the context and some 
of the proofs.

\section{The isomorphism $\lie(\mu) \simeq_{\sym_n}\widetilde H^{n-3}((\hat 0, [n]^{\mu}))\otimes 
\sgn_n$}\label{section:isomorphism}
In this section we establish the isomorphism of Theorem \ref{theorem:liekiso}. We will
use this isomorphism to study $\lie(\mu)$ by understanding the algebraic and combinatorial
properties of the maximal intervals $[\hat{0},[n]^{\mu}]$ of $\Pi_n^k$. In \cite{DleonWachs2013a} 
Wachs and the author  gave a proof of the
isomorphism for the case $k=2$, analogous to a proof for the case $k=1$ in \cite{Wachs1998}. For 
the sake of completeness in the discussion, we reproduce some of the steps but
omit the proofs that are similar to the ones in \cite{DleonWachs2013a} and
\cite{Wachs1998}.

\subsection{A combinatorial description of $\lie(\mu)$}\label{section:genliek}
We give a description of the generators and relations of $\lie(\mu)$. A \emph{tree} is a simple
connected graph that is free of cycles. A tree is said to be \emph{rooted} if it has a distinguished
node or \emph{root}. For an edge $\{x,y\}$ in a tree $T$ we say that $x$ is the \emph{parent} of 
$y$, or $y$ is the \emph{child} of $x$, if $x$ is in the unique path from $y$ to the root. A node
that has children is said to be \emph{internal}, otherwise we call a node without children a 
\emph{leaf}. A rooted tree is said to be \emph{planar} if for every internal node its set of 
children has been totally ordered. In the following we will be only considering trees that are 
rooted and planar and so when using the word tree we mean a planar rooted tree.

A \emph{binary tree} is a tree for which every internal node has a left and a
right child. A \emph{colored binary tree} is a binary tree
for which each internal node $x$ has been assigned an element $\clr(x) \in \PP$. 
For a colored binary tree $T$ with $n$ leaves and $\sigma \in \sym_n$, we define the 
\emph{labeled colored binary tree} $(T,\sigma)$ to be the colored tree $T$ whose $j$th leaf from 
left to right has been labeled $\sigma(j)$. For $\mu \in \wcomp_{n-1}$  we denote by $\BT_{\mu}$ 
the set of labeled colored binary trees with $n$ leaves and $\mu(j)$ internal nodes with color $j$ 
for each $j$. We call these trees \emph{$\mu$-colored binary trees}. We will often denote a 
labeled binary tree by $\Upsilon=(T,\sigma)$. If $\Upsilon$ is a colored labeled binary tree, we 
use $\widetilde \Upsilon$ to denote its underlying uncolored 
labeled binary tree.
It will also be convenient to consider trees whose label 
set is more general than $[n]$.  For a finite subset $A$ of positive integers with $|A|=|\mu|+1$, 
let $\BT_{A,\mu}$ be the 
set of
$\mu$-colored binary trees whose leaves are labeled by a permutation of $A$.
If $(S,\alpha) \in  \BT_{A,\mu}$ and $(T,\beta) \in 
\BT_{B,\nu}$, where $A$ and $B$ are disjoint finite sets, and $j \in \PP$ then
$(S,\alpha) \substack{j \\ \wedge} (T,\beta)$ denotes the tree in $\BT_{A\cup B,\mu+\nu+\bf{e_j}}$ 
whose \emph{left subtree} is $(S,\alpha)$, \emph{right subtree} is $(T,\beta)$, and the color of 
the root is $j$.

We can represent the bracketed permutations that generate $\lie(\mu)$ with labeled colored
binary trees.  More precisely, let $(T_1,\sigma_1)$ and $(T_2,\sigma_2)$ be the left and right
labeled subtrees of the root $r$ of $(T,\sigma)\in \BT_{\mu}$.  Then define recursively

\begin{align*}
[T,\sigma]= \left\{ 
  \begin{array}{l l}
   \biggl [[T_1,\sigma_1],[T_2,\sigma_2] \biggr ]_j & \quad \text{if $\clr(r)=j$ and $n>1$}\\
    \sigma & \quad \text{if $n=1$.}\\
  \end{array} \right .
\end{align*}
Clearly $[T,\sigma]$ is a bracketed permutation of $\lie(\mu)$. 
See Figure \ref{fig:coloredbinarytree}.

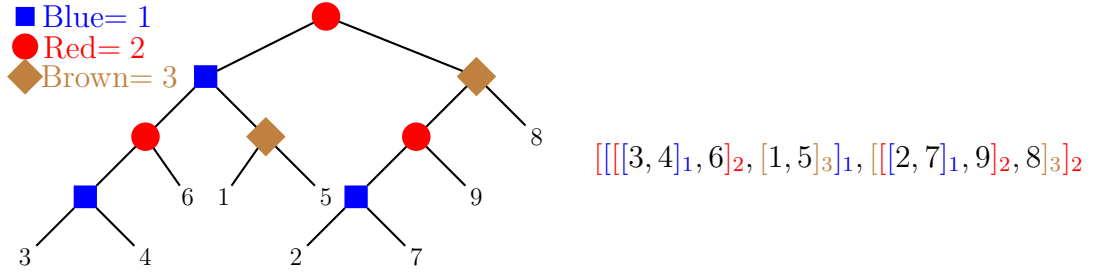
\begin{figure}
        \centering
        \begin{subfigure}[b]{0.6\textwidth}
                \centering
                \usetikzlibrary{shapes,snakes}

\begin{tikzpicture}[thick,scale=0.8]

 \draw [color=blue] (2.2,4)  node (blue){ Blue$=1$};
\draw [circle,color=red] (2.15,3.5)  node (red){ Red$=2$};
 \draw [color=brown] (2.4,3)  node (brown){ Brown$=3$};

\tikzstyle{every node}=[fill, draw,inner sep=4pt, minimum width=1pt,scale=0.8]
    
    \draw [color=blue] (1,4)  node (b){};
\draw [circle,color=red] (1,3.5)  node (r){};
    \draw [diamond,color=brown] (1,3)  node (g){};

\tikzstyle{every node}=[fill,draw,inner sep=1pt, minimum width=10pt,scale=0.8]

    \draw [circle,color=red] (6,4)  node (i1){8};
    \draw [diamond, color=brown] (8.5,3)  node (i2){7};
    \draw [circle,color=red] (7.5,2)  node (i3){6};
    \draw [color=blue] (6.5,1)  node (i4){5};
    \draw [color=blue] (4,3)  node (i5){4};
    \draw [diamond, color=brown] (5,2)  node (i6){3};
    \draw [circle,color=red] (3,2)  node (i7){2};
    \draw [color=blue] (2,1)  node (i8){1};
\tikzstyle{every node}=[inner sep=1pt, minimum width=14pt,scale=0.8]

    \draw (4.3,1)  node (l1){1};
    \draw (5.5,0)  node (l2){2};
    \draw (1,0)  node (l3){3};
    \draw (3,0)  node (l4){4};
    \draw (6,1)  node (l5){5};
    \draw (3.7,1)  node (l6){6};
    \draw (7.5,0)  node (l7){7};
    \draw (9.5,2)  node (l8){8};
    \draw (8.5,1)  node (l9){9};

    \draw (i1) --  (i2) ;
    \draw (i1) --  (i5) ;
    \draw (i2) --  (i3) ;
    \draw (i2) --  (l8) ;
    \draw (i3) --  (i4) ;
    \draw (i3) --  (l9) ;
    \draw (i4) --  (l7) ;
    \draw (i4) --  (l2) ;
    \draw (i5) --  (i6) ;
    
    \draw (i5) --  (i7) ;
    \draw (i6) --  (l5) ;
    \draw (i6) --  (l1) ;
    \draw (i7) --  (l6) ;
    \draw (i7) --  (i8) ;
    \draw (i8) --  (l3) ;
    \draw (i8) --  (l4) ;
\end{tikzpicture}
        \end{subfigure}~ 
        \begin{subfigure}[b]{0.55\textwidth}
                \centering
                 \begin{tikzpicture}[thick,scale=0.8,baseline=50]

\tikzstyle{every node}=[inner sep=1pt, minimum width=14pt,scale=1]

    \draw (0,4)  node
(v){${\color{red}[}{\color{blue}[}{\color{red}[}{\color{blue}[}3,4{\color{blue}]_1},6{
\color{red}]_2}, 
{\color{brown}[}1,5{\color{brown}]_3}{\color{blue} ]_1},{\color{brown}[}
{\color{red}[}{\color{blue}[}2,7{\color{blue}]_1},9{\color{red}]_2},8{\color{brown}]_3}{\color{red}
]_2}$};

\end{tikzpicture}
        \end{subfigure}
  \caption{Example of a labeled colored binary tree $(T,346152798) \in \BT_{(3,3,2)}$\,\, and \,\,
$[T,346152798]\in \lie(3,3,2)$}
  \label{fig:coloredbinarytree}

  \end{figure}

Recall that we call a set $B$ of Lie brackets on a vector space \emph{compatible} if any 
linear combination of 
the brackets in $B$ is a Lie bracket.
As it turns out the description of the relations in $\lie(\mu)$ are simplified by the following 
proposition.

\begin{proposition}\label{proposition:kcompatible2compatible}
 A set of Lie brackets is compatible if and only if the brackets in the set are pairwise compatible.
\end{proposition}

\begin{proof}
 Assume that the brackets $\{[\cdot,\cdot]_j \mid j \in S\}$ are pairwise compatible. 
Hence for  any $i,j \in S$ we have that the relation (\ref{relation:lb3}) holds.
Now for scalars $\alpha_j \in \kk$ and a finite subset $\{i_1,\dots,i_k\}\subseteq S$ define 
\[
 {\color{black}\langle} \cdot, \cdot {\color{black}\rangle}=\sum_{j=1}^k \alpha_j 
[\cdot,\cdot]_{i_j}.
\]

By relations (\ref{relation:lb2}) and (\ref{relation:lb3}) and bilinearity of the brackets , we have

\begin{align*}
0&= \sum_{j=1}^k \alpha_{j}^2 ([x,[ y, z ]_{j}]_{j}+[z,[ x,y ]_{j}]_{j}+[y,[ z, 
x]_{j}]_{j})\\
&+\sum_{l<j} \alpha_{l}\alpha_{j} ([x,[ y, z ]_{i_j}]_{i_l}+[z,[ x,y ]_{i_j}]_{i_l}+[y,[ z, 
x]_{i_j}]_{i_l}\\
&\qquad \qquad+[ x,[y,z]_{i_l} ]_{i_j} + [ z,[x,y]_{i_l} ]_{i_j} +[y,[z,x]_{i_l}]_{i_j})\\
&= \sum_{l,j=1}^k \alpha_{l}\alpha_{j} [ x, [ y,z ]_{i_l} ]_{i_j} +  \alpha_{l}\alpha_{j}[ 
z,[ x,y]_{i_l}
]_{i_j}+ \alpha_{l}\alpha_{j} [ y,[z,x]_{i_l} ]_{i_j}\\
&=\langle x, \langle y,z \rangle \rangle + \langle z,\langle x,y\rangle \rangle+\langle
y,\langle z,x\rangle \rangle .
\end{align*}
This implies that ${\color{black}\langle} \cdot, \cdot {\color{black}\rangle}$ satisfies relation
(\ref{relation:lb2}). It follows from the definition that ${\color{black}\langle} \cdot,
\cdot {\color{black}\rangle}$ also satisfies the relation (\ref{relation:lb1}) and hence it is a Lie
bracket. 

For the converse note, from the definition of compatibility, that all the brackets in a 
compatible set of Lie brackets are pairwise compatible.\end{proof}

Thus we see that $\lie(\mu)$ is subject only to the relations (\ref{relation:lb1}) and
(\ref{relation:lb2}), for each bracket $j$, and (\ref{relation:lb3}) for any pair of brackets 
$i \neq j \in [k]$. If the characteristic of $\kk$ is not $2$  we can even say 
that $\lie(\mu)$ is subject only to relations (\ref{relation:lb1}) and (\ref{relation:lb3}) 
for 
any pair of brackets $i,j \in [k]$ (including $i=j$). 

We denote by $\Upsilon_1\substack{j \\ \wedge}\Upsilon_2$, the labeled colored binary
tree whose left subtree is $\Upsilon_1$, right subtree is $\Upsilon_2$ and root color  is  $j$,
with $j\in \PP$. If $\Upsilon$ is a labeled colored binary tree then  
$\alpha(\Upsilon)\beta$ denotes a labeled colored binary tree with $\Upsilon$ as a subtree. 
The following result is an easy consequence of relations (\ref{relation:lb1}) and  
(\ref{relation:lb2}) for each $j$, and (\ref{relation:lb3}) for each pair $i\ne j$.

\begin{proposition}\label{proposition:binarybasislie}
The set $\{ [T,\sigma] \mid (T,\sigma) \in \BT_{\mu}\}$ is a generating set for $\lie(\mu)$,
subject only to the relations for $i \ne j \in \supp(\mu)$ 

  \begin{align}[\alpha(\Upsilon_1\substack{j \\ \wedge \\ \,}\Upsilon_2)\beta]+
[\alpha(\Upsilon_2\substack{j \\ \wedge \\ \,}\Upsilon_1)\beta]=0 
\label{relation:1}\end{align}

\begin{align}\label{relation:3}  [\alpha(\Upsilon_1\substack{j \\ \wedge \\
\,}(\Upsilon_2\substack{j \\ \wedge \\ \,\\ }\Upsilon_3))\beta]
&- [\alpha((\Upsilon_1\substack{j \\ \wedge \\ \,}\Upsilon_2)\substack{j \\ \wedge \\
\,}\Upsilon_3)\beta]\\ 
&-  [\alpha(\Upsilon_2\substack{j \\ \wedge \\ \,}(\Upsilon_1\substack{j \\ \wedge \\
\,}\Upsilon_3))\beta]
\nonumber  \\ &= 0\nonumber
\end{align}

\begin{align} \label{relation:5} 
[\alpha(\Upsilon_1\substack{j \\ \wedge}(\Upsilon_2\substack{i \\ \wedge}\Upsilon_3))\beta]
 &+ [\alpha(\Upsilon_1\substack{i \\ \wedge}(\Upsilon_2\substack{j \\
\wedge}\Upsilon_3))\beta]\\ \nonumber
- \,\, \,\, [\alpha((\Upsilon_1\substack{j \\ \wedge}\Upsilon_2)\substack{i \\
\wedge}\Upsilon_3)\beta]
&-  [\alpha((\Upsilon_1\substack{i \\ \wedge}\Upsilon_2)\substack{j \\
\wedge}\Upsilon_3)\beta] \\ \nonumber
-\,\,\,\, [\alpha(\Upsilon_2\substack{j \\ \wedge}(\Upsilon_1\substack{i \\
\wedge}\Upsilon_3))\beta]
&- [\alpha(\Upsilon_2\substack{i \\ \wedge}(\Upsilon_1\substack{j \\
\wedge}\Upsilon_3))\beta]\\ \nonumber
&= 0.
\end{align}

\end{proposition}

\subsection{A generating set for $\widetilde H^{n-3}((\hat 0, [n]^{\mu}))$}\label{section:genhom}

The top dimensional cohomology of a pure poset $P$, say of length $\ell$, has a particularly simple
description.  Let $\MM(P)$ denote the set of maximal chains of $P$ and let $\MM^\prime(P)$ denote
the set of chains of length $\ell-1$.  We view the coboundary map $\delta$ as a map from the chain
space of  $P$ to itself, which takes chains of length $d$ to chains of length $d+1$ for all $d$. 
Since the image of $\delta$ on the top chain space (i.e. the space spanned by $\MM(P)$) is $0$, the
kernel is the entire top chain space. Hence top cohomology is the quotient of the space spanned by
$\MM(P)$ by the image of the space
spanned by $\MM^\prime(P)$.  The image of $\MM^\prime(P)$ is what we call the coboundary relations. 
We thus have the following presentation of the top cohomology  $$\widetilde H^{\ell}(P) = \langle 
\MM(P)
| \mbox{ coboundary relations} \rangle.$$ 

Recall that the \emph{postorder listing} of the internal nodes of a binary tree $T$ is defined
recursively as follows:  first list the internal nodes of the left subtree  in postorder, then list
the internal nodes of the right subtree in postorder, and finally list the root. The postorder
listing of the internal nodes of the binary tree of Figure~\ref{fig:coloredbinarytree}
is illustrated in Figure \ref{fig:postorder}.

Given $s$ blocks $A_1^{\mu_1},A_2^{\mu_2}, \dots, A_s^{\mu_s}$ in a weighted partition $\alpha$ and 
$\nu \in \wcomp_{s-1}$, by $\nu$-\emph{merge} these blocks we mean remove them from 
$\alpha$ and
replace them by the block $(\bigcup A_i )^{\sum \mu_i + \nu}$.  For  $(T,\sigma) \in \BT_{A,\mu}$, 
let $\pi(T,\sigma)=A^{\mu}$.   

\begin{definition}\label{definition:treechain} 
For $(T,\sigma) \in \BT_{\mu}$ and $t \in [n-1]$, let $T_{t}=L_{t}\substack{j_t \\ \wedge} R_{t}$
be the subtree of $(T,\sigma)$ rooted at 
the $t$th node listed in postorder.  The chain $\c(T,\sigma)\in \MM([\hat{0},[n]^{\mu}])$  is the 
one
whose rank $t$ weighted partition is obtained from the rank $t-1$ weighted partition by
$\bf{e_{j_t}}$-merging the blocks $\pi(L_{t})$ and
$\pi(R_{t})$. 
See Figure~\ref{fig:binarytreechain}.
\end{definition}

\begin{figure}[ht]
        \centering
        \begin{subfigure}[b]{0.5\textwidth}
                \centering
                   \usetikzlibrary{shapes,snakes}

\begin{tikzpicture}[thick,scale=0.8]

\tikzstyle{every node}=[draw,inner sep=1pt, minimum width=10pt,scale=0.8]

    \draw [circle,color=red] (6,4)  node (i1){8};
    \draw [diamond, color=brown] (8.5,3)  node (i2){7};
    \draw [circle,color=red] (7.5,2)  node (i3){6};
    \draw [color=blue] (6.5,1)  node (i4){5};
    \draw [color=blue] (4,3)  node (i5){4};
    \draw [diamond, color=brown] (5,2)  node (i6){3};
    \draw [circle,color=red] (3,2)  node (i7){2};
    \draw [color=blue] (2,1)  node (i8){1};
\tikzstyle{every node}=[inner sep=1pt, minimum width=14pt,scale=0.8]

    \draw (4.3,1)  node (l1){1};
    \draw (5.5,0)  node (l2){2};
    \draw (1,0)  node (l3){3};
    \draw (3,0)  node (l4){4};
    \draw (6,1)  node (l5){5};
    \draw (3.7,1)  node (l6){6};
    \draw (7.5,0)  node (l7){7};
    \draw (9.5,2)  node (l8){8};
    \draw (8.5,1)  node (l9){9};

    \draw (i1) --  (i2) ;
    \draw (i1) --  (i5) ;
    \draw (i2) --  (i3) ;
    \draw (i2) --  (l8) ;
    \draw (i3) --  (i4) ;
    \draw (i3) --  (l9) ;
    \draw (i4) --  (l7) ;
    \draw (i4) --  (l2) ;
    \draw (i5) --  (i6) ;
    
    \draw (i5) --  (i7) ;
    \draw (i6) --  (l5) ;
    \draw (i6) --  (l1) ;
    \draw (i7) --  (l6) ;
    \draw (i7) --  (i8) ;
    \draw (i8) --  (l3) ;
    \draw (i8) --  (l4) ;
\end{tikzpicture}
                \caption{$(T,\sigma) \in \BT_{(3,3,2)}$}
                \label{fig:postorder}
        \end{subfigure}%
        ~ 
        \begin{subfigure}[b]{0.5\textwidth}
                \centering
                 \begin{tikzpicture}[thick,scale=0.6]

 \tikzstyle{every node}=[inner sep=1pt, minimum width=10pt,scale=0.60]

    \draw (0,0)  node (c0){$1^{(0,0,0)}|2^{(0,0,0)}|3^{(0,0,0)}|4^{(0,0,0)}|5^{(0,0,0)}|6^{(0,0,0)}|7^{(0,0,0)}|8^{(0,0,0)}|9^{(0,0,0)}$};

    \draw (0,1)  node (c1){$1^{(0,0,0)}|2^{(0,0,0)}|34^{(1,0,0)}|5^{(0,0,0)}|6^{(0,0,0)}|7^{(0,0,0)}|8^{(0,0,0)}|9^{(0,0,0)}$};
    \draw [color=blue] (c0) -- (c1);

    \draw (0,2)  node (c2){$1^{(0,0,0)}|2^{(0,0,0)}|346^{(1,1,0)}|5^{(0,0,0)}|7^{(0,0,0)}|8^{(0,0,0)}|9^{(0,0,0)}$};
    \draw [color=red](c1) -- (c2);

    \draw (0,3)  node 
(c3){$15^{(0,0,1)}|2^{(0,0,0)}|346^{(1,1,0)}|7^{(0,0,0)}|8^{(0,0,0)}|9^{(0,0,0)}$};
    \draw [color=brown] (c2) -- (c3) ;

    \draw (0,4)  node (c4){$13456^{(2,1,1)}|2^{(0,0,0)}|7^{(0,0,0)}|8^{(0,0,0)}|9^{(0,0,0)}$};
    \draw [color=blue] (c3) -- (c4);

    \draw (0,5)  node (c5){$13456^{(2,1,1)}|27^{(1,0,0)}|8^{(0,0,0)}|9^{(0,0,0)}$};
    \draw [color=blue](c4) -- (c5);

    \draw (0,6)  node (c6){$13456^{(2,1,1)}|279^{(1,1,0)}|8^{(0,0,0)}$};
    \draw [color=red](c5) -- (c6);

    \draw (0,7)  node (c7){$13456^{(2,1,1)}|2789^{(1,1,1)}$};
    \draw [color=brown] (c6) -- (c7);

    \draw (0,8)  node (c8){$123456789^{(3,3,2)}$};
    \draw [color=red](c7) -- (c8);

\end{tikzpicture}
 
                \caption{$\c(T,\sigma)$}
                 \label{fig:binarytreechain}
        \end{subfigure}
 \caption{Example of postorder (internal nodes) of the binary tree of Figure
\ref{fig:coloredbinarytree} and the chain $\c(T,\sigma)$}
    \label{fig:binarytreepostorderandchain}

  \end{figure}
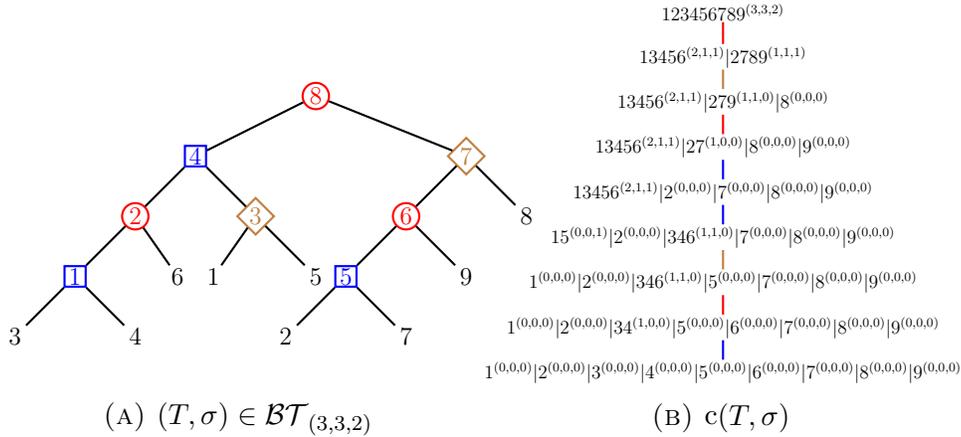

Not all maximal chains in $\MM([\hat{0},[n]^{\mu}])$ can be described as $\c(T,\sigma)$. For some 
maximal chains postordering of the internal nodes is not enough to describe the process of merging 
the blocks.  We need a more flexible construction in terms of linear
extensions (cf. \cite{Wachs1998}).  Let $v_1,\dots,v_n$ be the postorder listing of the internal
nodes of $T$.  A listing $v_{\tau(1)},v_{\tau(2)},...,v_{\tau(n-1)}$ of the internal nodes  such
that each node precedes its parent is said to be a \emph{linear extension} of $T$.  We will say that
the permutation $\tau$ induces the linear extension.   In particular, the identity permutation
$\varepsilon$ induces postorder which is a linear extension. Denote by $e(T)$  the set of 
permutations   that induce linear extensions of the internal nodes of  $T$.
We extend the construction of $\c(T,\sigma)$ by letting $\c(T,\sigma,\tau)$  be the chain in
$\MM([\hat{0},[n]^{\mu}])$ 
whose rank $t$ weighted partition is obtained from the rank $t-1$ weighted partition by
$e_{j_{\tau(t)}}$-merging the blocks
$\pi(L_{\tau(t)})$ and
$\pi(R_{\tau(t)})$, where $L_{i}\substack{j_i \\ \wedge}
R_{i}$ 
is the subtree rooted at $v_{i}$. 
In particular,
$\c(T,\sigma)=\c(T,\sigma,\varepsilon)$. From each maximal chain we can easily construct a binary
tree and a linear extension that encodes the merging
instructions along the chain. Thus, any maximal chain can be obtained in this form.
  
For any colored labeled binary tree $(T,\sigma)$, the chains obtained with any two different linear
extensions are cohomologous in the sense of Lemma \ref{lemma:52} below.

The number of inversions of a permutation  $\tau \in \sym_n$ is defined by $\inv(\tau) := |\{ (i,j) 
\mid
1 \le i < j \le n, \,\, \tau(i) > \tau(j) \}|$ and the sign of $\tau$ is defined by
$\sgn(\tau) := (-1)^{\inv(\tau)}$.   For $T \in \BT_{n,\mu}$, $\sigma \in \sym_n$, and $\tau \in 
e(T)$, 
write $\bar c(T,\sigma,\tau)$ for  $\overline{c(T,\sigma,\tau)}:= c(T,\sigma,\tau) \setminus \{\hat 
0, [n]^{\mu}\}$ and  $\bar c(T,\sigma)$ for  $\overline{c(T,\sigma)}:= c(T,\sigma) \setminus \{\hat 
0, 
[n]^{\mu}\}$.

\begin{lemma}[{cf. \cite[Lemma 5.2]{Wachs1998} }]\label{lemma:52}
 Let $(T,\sigma) \in \BT_{\mu}$, $\tau \in e(T)$. Then in
$\widetilde H^{n-3}((\hat{0},[n]^{\mu}))$
\[
\bar{\c}(T,\sigma,\tau)=\sgn(\tau)\bar{\c}(T,\sigma).
\]
\end{lemma}

The proof of Lemma \ref{lemma:52} is essentially the same as \cite[Lemma 5.2]{Wachs1998}.

For $\Upsilon=(T,\sigma)$, let $I(\Upsilon)$ denote the set of internal nodes of $\Upsilon$.
The following result generalizes  \cite[Theorem 4.4]{DleonWachs2013a}.

\begin{theorem}\label{theorem:binarybasishomology}
 The set $\{ \bar{\c}(T,\sigma) \mid  (T,\sigma) \in \BT_{\mu}\}$ is a generating set for 
$\widetilde
H^{n-3}((\hat 0, [n]^{\mu}))$,
subject only to the relations for $i \ne j \in \supp(\mu)$

  \begin{align}\bar{\c}(\alpha(\Upsilon_1\substack{j \\ \wedge \\
\,}\Upsilon_2)\beta)-(-1)^{|I(\Upsilon_1)||I(\Upsilon_2)|}\bar{\c}(\alpha(\Upsilon_2\substack{j
\\ \wedge \\ \,}\Upsilon_1)\beta)=0\label{relation:1h}\end{align}

\begin{align}\label{relation:3h}\qquad \bar{\c}(\alpha(\Upsilon_1\substack{j \\ \wedge
\\ \,}(\Upsilon_2\substack{j \\ \wedge \\ \,\\ }\Upsilon_3))\beta)
&+ (-1)^{|I(\Upsilon_3)|}\bar{\c}(\alpha((\Upsilon_1\substack{j \\ \wedge \\
\,}\Upsilon_2)\substack{j \\ \wedge \\ \,}\Upsilon_3)\beta)\\ 
  &+ (-1)^{|I(\Upsilon_1)||I(\Upsilon_2)|}\bar{\c}(\alpha(\Upsilon_2\substack{j \\ \wedge \\
\,}(\Upsilon_1\substack{j \\ \wedge \\ \,}\Upsilon_3))\beta)
\nonumber  \\ &= 0\nonumber
\end{align}

\begin{align} \label{relation:5h} 
 \bar{\c}(\alpha(\Upsilon_1\substack{j \\ \wedge}(\Upsilon_2\substack{i \\
\wedge}\Upsilon_3))\beta)
 &+ \bar{\c}(\alpha(\Upsilon_1\substack{i \\ \wedge}(\Upsilon_2\substack{j \\
\wedge}\Upsilon_3))\beta)\\ \nonumber
+ \,\, \,\, (-1)^{|I(\Upsilon_3)|} {\Big {(}} \bar{\c}(\alpha((\Upsilon_1\substack{j \\
\wedge}\Upsilon_2)\substack{i \\ \wedge}\Upsilon_3)\beta)
&+  \bar{\c}(\alpha((\Upsilon_1\substack{i \\ \wedge}\Upsilon_2)\substack{j \\
\wedge}\Upsilon_3)\beta) \Big) \\ \nonumber
+\,\,\,\, (-1)^{|I(\Upsilon_1)||I(\Upsilon_2)|}\Big(\bar{\c}(\alpha(\Upsilon_2\substack{j \\
\wedge}(\Upsilon_1\substack{i \\ \wedge}\Upsilon_3))\beta)
&+ \bar{\c}(\alpha(\Upsilon_2\substack{i \\ \wedge}(\Upsilon_1\substack{j \\
\wedge}\Upsilon_3))\beta)\Big) \\ \nonumber
&= 0.
\end{align}

\end{theorem}

\begin{proof}
The proof is analogous to \cite[Theorem 4.3]{DleonWachs2013a}. The ``only" part 
follows from Proposition \ref{proposition:lyndononlyrelations}.
\end{proof}

 \subsection{The isomorphism}

 The symmetric group $\sym_n$ acts  naturally  on $\Pi_n^k$.  Indeed,  let 
$\sigma\in \sym_n$ act on
the weighted blocks of $\Pi_n^k$ by replacing each element $x$ of each weighted block of
$\pi$ with $\sigma(x)$.  Since the maximal elements  of $\Pi_n^k$ are fixed by each $\sigma\in
\sym_n$ and the order is preserved, each open interval $(\hat 0, [n]^{\mu})$ is an $\sym_n$-poset. 
Hence (see \cite{Wachs2007}) we have that $\tilde H^{n-3}((\hat 
0, [n]^{\mu}))$ is an $\sym_n$-module.
The symmetric group $\sym_n$ also acts naturally on $\lie(\mu)$.  Indeed, let $\sigma \in \sym_n$
act by replacing letter $x$ of a bracketed permutation with $\sigma(x)$.  Since this action
preserves the number of brackets of each type, $\lie(\mu)$ is an $\sym_n$-module for each $\mu \in 
\wcomp_{n-1}$.
In this section we obtain an explicit sign-twisted isomorphism between the $\sym_n$-modules 
$\tilde
H^{n-3}((\hat 0, [n]^{\mu}))$ and $\lie(\mu)$.

Define the \emph{sign} of a binary tree $T$ recursively by 
$$\sgn(T) = \begin{cases} 1 &\mbox{ if } I(T) = \emptyset \\ (-1)^{|I(T_2)|}\sgn(T_1)\sgn(T_2) &
\mbox{ if } T = T_1 \land T_2, \end{cases}
$$
where $I(T)$ is the set of internal nodes of the binary tree $T$.
The sign of a colored (labeled or unlabeled) binary tree is defined to be the sign of the binary 
tree obtained by
removing the colors and leaf labels.

\begin{theorem}\label{theorem:liehomisomorphism}
 For each $\mu \in \wcomp_{n-1}$,  there is an $\sym_n$-module isomorphism
$\varphi:\lie(\mu)\rightarrow \widetilde H^{n-3}((\hat 0, [n]^{\mu}))\otimes \sgn_{n}$ determined by
\[
\varphi([T,\sigma])=\sgn(\sigma)\sgn(T)\bar{\c}(T,\sigma),\]
for all $ (T,\sigma) \in \BT_{\mu}$.
\end{theorem}
 
\begin{proof}
The proof of this result is almost identical to the one in \cite{DleonWachs2013a}, so we omit
the details. The map $\varphi$ maps the generators and relations of Proposition
\ref{proposition:binarybasislie} onto the generators and relations of Proposition
\ref{theorem:binarybasishomology} and clearly respects the $\sym_n$ action.
\end{proof}

\section{Homotopy type of the intervals of $\Pi_n^k$}\label{section:homotopytype}
We assume familiarity  with basic terminology and results in poset topology. The reader is
referred to \cite[Section $3$ and Appendix]{DleonWachs2013a} for a  
review of poset
(co)homology.
For further poset topology terminology  not defined here the reader could also visit 
\cite{Stanley2012} and 
\cite{Wachs2007}. 

For $u  \le
v$ in
a poset $P$,  the open interval $\{w \in P \mid u < w < v\}$ is denoted  by $(u,v)$ and the closed
interval $\{w \in P \mid u \le w \le v\}$ by $[u,v]$. A poset is said to be \emph{bounded} if it 
has a
minimum element $\hat 0$ and a maximum element $\hat 1$.  For a bounded poset $P$, we define the
\emph{proper part} of $P$ as $\overline {P}:= P\setminus\{\hat 0,\hat 1\}$.  A poset is said to 
be \emph{pure} (or ranked)  if all its maximal chains have the same length, where the length of a 
chain $s_0<s_1 < \dots < s_{\ell}$ is $\ell$.  If $P$ is pure and has a minimal element 
$\hat{0}$, we can define a rank function $\rho$ by requiring that $\rho(\hat{0})=0$ and 
$\rho(\beta)=\rho(\alpha)+1$ whenever $\alpha \lessdot \beta$ in $P$. For example $\Pi_n^k$ is a 
pure poset with rank function given by $\rho(\alpha)=n-|\alpha|$ for every $\alpha \in \Pi_n^k$. 
The 
\emph{length} $\ell(P)$ of a poset $P$ is the 
length of 
its longest chain. For a bounded poset $P$, let  $\overline \mu_P$ denote its M\"obius function. 
The reason for the nonstardard notation $\overline \mu_P$ is that we have 
been using the symbol $\mu$ to denote a weak composition.

\subsection{M\"obius invariant}\label{subsection:mobius}
For $\alpha = \{A_1^{\mu_1},\dots,A_s^{\mu_s}\}  \in \Pi_n^k$, let $\mu(\alpha)=\sum_{i=1}^s \mu_i$.
 The following proposition about the structure of $\Pi_n^k$ will be used in the computations below.
\begin{proposition}\label{proposition:upperlowerideals}
 For all  $\alpha = \{A_1^{\mu_1},\dots,A_s^{\mu_s}\} \in \widehat{\Pi_n^{k}}$, $\alpha \neq 
\hat{1}$, and
 $\nu \in \wcomp_{n-1}$ such that $\nu-\mu(\alpha) \in \wcomp_{|\alpha|-1}$,
 \begin{enumerate}
 \item $[\alpha,\hat 1]$ and $\widehat{\Pi_s^{k}}$ are isomorphic posets,
 \item $[\alpha, [n]^\nu]$ and $[\hat 0, [|\alpha|]^{\nu-\mu(\alpha)}]$ are isomorphic posets,
 \item $[\hat{0},\alpha]$ and  $[\hat{0},[|A_1|]^{\mu_1}] \times \cdots \times [\hat 0,
[|A_s|]^{\mu_s}]$ 
are isomorphic posets.
\end{enumerate}
\end{proposition}

Proposition \ref{proposition:upperlowerideals} is a general statement that is satisfied by any
partition poset associated to a basic set operad (see \cite{Vallette2007}) replacing the 
composition 
$\mu$ by an element of the given operad (see also \cite{MendezYang1991}). 

Recall that $\mathbf{x}^{\mu}=x_1^{\mu(1)}\cdots x_k^{\mu(k)}$ and $(\cdot)^{<-1>}$ denotes 
compositional inverse. We use the recursive definition of the M\"obius function $\bar \mu_P$ and 
the compositional 
formula to derive the following theorem.

\begin{theorem}\label{theorem:compositionalinversemu}
  We have that
  \begin{align*}
   \sum_{n\ge1}\sum_{\substack{\mu \in \wcomp_{n-1}\\\supp(\mu)\subseteq [k]}}
		\bar \mu_{\Pi_{n}^{k}}(\hat{0},[n]^{\mu})\mathbf{x}^{\mu}\frac{y^n}{n!} =\left [ 
\sum_{n\ge1}h_{n-1}(x_1,\dots,x_k)\frac{y^n}{n!} \right ] ^{<-1>},
  \end{align*}
  where $h_n$  is the complete homogeneous symmetric polynomial.

 \end{theorem}
\begin{proof}
 By the recursive definition of the 
M\"obius function we have that
\begin{align*}
 \delta_{n,1}&=\sum_{\substack{\mu \in \wcomp_{n-1}\\\supp(\mu)\subseteq [k]}}\mathbf{x}^{\mu}
\sum_{\hat{0}\le\alpha\le[n]^{\mu}}\bar \mu_{\Pi_{n}^{k}}(\alpha,[n]^{\mu})\\
	    &= 	\sum_{\alpha\in \Pi_{n}^{k}}\mathbf{x}^{\mu(\alpha)}
		\sum_{\substack{\mu \in \wcomp_{n-1}\\ \mu \ge \mu(\alpha)\\\supp(\mu)\subseteq 
[k]}}
		\bar\mu_{\Pi_{n}^{k}}(\alpha,[n]^{\mu})\mathbf{x}^{\mu-\mu(\alpha)}.
\end{align*}
Now using Proposition \ref{proposition:upperlowerideals}
\begin{align*}
	   \delta_{n,1} &= 	\sum_{\alpha \in \Pi_{n}^{k}}
		\mathbf{x}^{\mu(\alpha)}
		\sum_{\substack{\nu \in \wcomp_{|\alpha|-1}\\\supp(\nu)\subseteq [k]}}
		\bar \mu_{\Pi_{|\alpha|}^{k}}(\hat{0},[|\alpha|]^{\nu})\mathbf{x}^{\nu}\\
	    &= 	\sum_{\alpha \in \Pi_{n}}
		\prod_{i=1}^{|\alpha|}h_{|\alpha_i|-1}(x_1,...,x_k)
		\sum_{\substack{\nu \in \wcomp_{|\alpha|-1}\\\supp(\nu)\subseteq [k]}}
		\bar \mu_{\Pi_{|\alpha|}^{k}}(\hat{0},[|\alpha|]^{\nu})\mathbf{x}^{\nu}.
 \end{align*}
The last statement implies using the compositional formula see (\cite[Theorem~5.1.4]{Stanley1999}) 
that
the two power series are compositional inverses.
\end{proof}

\subsection{EL-labeling}\label{section:ellabeling}

Let $P$ be a bounded poset. An \emph{edge labeling} is a map $\bar \lambda: \mathcal E(P) \to 
\Lambda$,  where $\mathcal E(P)$ 
is the set of edges of the Hasse diagram of a poset $P$ and $\Lambda$ is a fixed poset. We denote by

\begin{align*}
\bar \lambda(c) = \bar \lambda(x_0, x_1) \bar \lambda(x_1, x_2) \cdots \bar \lambda(x_{t-1}, x_{t}),
\end{align*} 
the word of labels corresponding to a maximal chain $c = (\hat 0 = x_0
\lessdot x_1 \lessdot \cdots \lessdot x_{t-1} \lessdot x_t= \hat 1)$. 
We say that  $c$ is  \emph{increasing} if its word of labels $\bar \lambda(c)$ is
\emph{strictly}  increasing, that is, $c$ is  increasing if 
\begin{align*}
 \bar \lambda(x_0, x_1) < \bar \lambda(x_1, x_2)<  \cdots < \bar \lambda(x_{t-1}, 
x_t). 
\end{align*}
  We say 
that  $c
$ is  \emph{ascent-free} (or decreasing, or falling) if its word of labels $\bar \lambda(c)$ has no 
ascents,
i.e. $\bar \lambda(x_i, x_{i+1}) \not<  \bar \lambda(x_{i+1}, x_{i+2}) $, for all $i=0,\dots,t-2$.
\emph{ An edge-lexicographical
labeling} (EL-labeling, for short)  of
$P$ is an edge labeling such that in each closed
interval $[x,y]$ of $P$, there is a unique  increasing maximal chain, and this chain
lexicographically precedes all other maximal chains of $[x,y]$.

A classical EL-labeling for the partition lattice $\Pi_n$ is obtained as follows.  Let $\Lambda =
\{(i,j)\in [n-1] \times [n] \mid i <j\}$ with lexicographic order as the order relation on 
$\Lambda$. 
If $x\lessdot y $ in $\Pi_n$  then $y$ is obtained from $x$ by merging two blocks $A$ and $B$, 
where $\min A < \min B$.
 Let $\bar \lambda(x,y) = (\min A, \min B)$.  This defines a map $\bar \lambda:\mathcal E(\Pi_n)
\to \Lambda$ (Note that $\bar \lambda$ in this section is an edge labeling and not an integer 
partition).
By viewing $\Lambda$ as the set of atoms of $\Pi_n$, one sees that this labeling is a special case
of an  edge labeling for geometric lattices, which first appeared in Stanley \cite{Stanley1974} and 
was
one of  Bj\"orner's \cite{Bjorner1980} initial examples of an EL-labeling. A generalization of the
Bj\"orner-Stanley EL-labeling was given in \cite{DleonWachs2013a} for the poset $\Pi_n^w$. We
generalize further this labeling by providing one for $\Pi_n^k$ that reduces to the one in
\cite{DleonWachs2013a} for $k=2$ and to the Bj\"orner-Stanley EL-labeling when $k=1$.

\begin{definition}[Poset of labels]\label{definition:posetoflabels}
For each $a \in [n]$, let $\Gamma_a:= \{(a,b)^u \mid   a<b \le n+1, \,\, u \in [k] \}$. 
We
partially order $\Gamma_a$  by letting $(a,b)^u \le  (a,c)^v$ if $b\le  c$ and $u \le v$.   
Note that $\Gamma_a$ is isomorphic to the direct product of the chain $a+1< a+2 <\cdots < n+1 $ and
the chain $1 < 2 < \cdots < k$.  Now define $\Lambda^k_n$ to be the 
ordinal sum
$\Lambda^k_n := \Gamma_1 \oplus  \Gamma_2  \oplus \cdots \oplus \Gamma_{n}$ (see Figure
\ref{fig:lambdaposet}).
\end{definition}

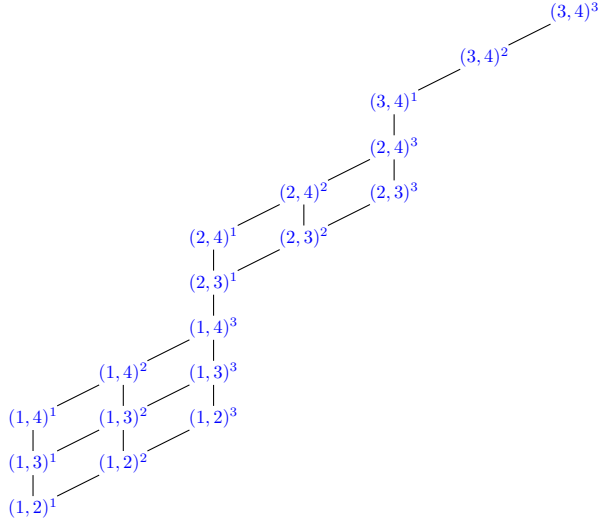
\begin{figure}[ht]
  \centering
  \begin{tikzpicture}[scale=0.6]
 \tikzstyle{every node}=[inner sep=1pt, minimum width=14pt,scale=0.7, font=\footnotesize]
\draw (0,0) node (n121) {\color{blue}$(1,2)^1$};
\draw (0,1) node (n131) {\color{blue}$(1,3)^1$};
\draw (0,2) node (n141) {\color{blue}$(1,4)^1$};
\draw (2,1) node (n122) {\color{blue}$(1,2)^2$};
\draw (2,2) node (n132) {\color{blue}$(1,3)^2$};
\draw (2,3) node (n142) {\color{blue}$(1,4)^2$};
\draw (4,2) node (n123) {\color{blue}$(1,2)^3$};
\draw (4,3) node (n133) {\color{blue}$(1,3)^3$};
\draw (4,4) node (n143) {\color{blue}$(1,4)^3$};

\draw (n143) -- (n142) -- (n141) ;
\draw (n133) -- (n132) -- (n131) ;
\draw (n123) -- (n122) -- (n121) ;
\draw (n141)-- (n131) -- (n121) ;
\draw (n142)-- (n132) -- (n122) ;

\draw (4,5) node (n231) {\color{blue}$(2,3)^1$};
\draw (4,6) node (n241) {\color{blue}$(2,4)^1$};
\draw (6,6) node (n232) {\color{blue}$(2,3)^2$};
\draw (6,7) node (n242) {\color{blue}$(2,4)^2$};
\draw (8,7) node (n233) {\color{blue}$(2,3)^3$};
\draw (8,8) node (n243) {\color{blue}$(2,4)^3$};

\draw (8,9) node (n341) {\color{blue}$(3,4)^1$};
\draw (10,10) node (n342) {\color{blue}$(3,4)^2$};
\draw (12,11) node (n343) {\color{blue}$(3,4)^3$};
\draw (n243) -- (n242) -- (n241) ;
\draw (n233) -- (n232) -- (n231) ;

\draw (n241)-- (n231)  -- (n143) -- (n133) -- (n123);

\draw (n343) -- (n342) -- (n341) ;

\draw (n341) -- (n243) -- (n233);
\draw (n242) -- (n232);

\end{tikzpicture}
  \caption{$\Lambda^3_3$}
  \label{fig:lambdaposet}
\end{figure}

\begin{definition}[EL-labeling]
  If $x\lessdot y $ in $\Pi^k_n$   then $y$ is obtained from $x$ by $\bf{e}_r$-merging two blocks 
$A$ 
and $B$ for some $r \in [k]$,  where $\min A < \min B$.

Let $$\bar \lambda(x \lessdot y) = (\min A, \min B)^r.$$  This defines a map
$\bar \lambda:\mathcal E(\Pi^k_n) \to \Lambda^k_n$.  We extend this map to $\bar \lambda:\mathcal
E(\widehat{\Pi^k_n})\to \Lambda_n$ by letting ${\bar \lambda}([n]^{\mu} \lessdot 
\hat{1})=(1,n+1)^1$, for all
$\mu \in \wcomp_{n-1}$ with $\supp(\mu) \subseteq [k]$ (See Figure~\ref{fig:ellabelingposet}).

\begin{remark}
 Recall that when $\mu$ has a single nonzero entry (equal to $n-1$), the interval $[\hat 0, 
[n]^{\mu}]$ is isomorphic to $\Pi_n$. Note that  $\bar \lambda$ reduces to the 
Bj\"orner-Stanley EL-labeling on those intervals.
\end{remark}

\end{definition}

\begin{figure}[ht]
                \centering
               \begin{tikzpicture}[line join=bevel,scale=0.8]

\tikzstyle{every node}=[inner sep=0pt, scale=0.6, minimum width=4pt]
\node (hat1) at (0,6) {$\hat{1}$}; 

\node (n102030) at (0,0)  {$1^{(0,0,0)}| 2^{(0,0,0)}| 3^{(0,0,0)}$};

  \node (n12i30) at (-8,2) {$12^{(1,0,0)}| 3^{(0,0,0)}$};
  \node (n13i20) at (-6,2) {$13^{ (1,0,0)}| 2^{ (0,0,0)}$};
  \node (n1023i) at (-4,2)  {$1^{(0,0,0)}| 23^{(1,0,0)}$};

 \node (n12j30) at (-2,2)  {$12^{(0,1,0)}| 3^{(0,0,0)}$};
 \node (n13j20) at (0,2)  {$13^{(0,1,0)}| 2^{(0,0,0)}$};
  \node (n1023j) at (2,2) {$1^{(0,0,0)}| 23^{(0,1,0)}$};

\node (n12k30) at (4,2)  {$12^{ (0,0,1)}| 3^{(0,0,0)}$};
 \node (n13k20) at (6,2)  {$13^{(0,0,1)}| 2^{(0,0,0)}$};
  \node (n1023k) at (8,2) {$1^{(0,0,0)}| 23^{(0,0,1)}$};

 \node (n123ii) at (-6,4){$123^ {(2,0,0)}$};
 \node (n123ij) at (-3,4){$123^ {(1,1,0)}$};
 \node (n123ik) at (-1,4){$123^ {(1,0,1)}$};
 \node (n123jj) at (1,4){$123^ {(0,2,0)}$};
 \node (n123jk) at (3,4){$123^ {(0,1,1)}$};
 \node (n123kk) at (6,4){$123^ {(0,0,2)}$};

 \draw (n123ii) -- (n1023i) ;
 \draw (n123ii)-- (n13i20);
  \draw (n123ii) -- (n12i30);  

\draw (n123jj) -- (n1023j) ;
 \draw (n123jj)-- (n13j20);
  \draw (n123jj) -- (n12j30);  

\draw (n123kk) -- (n1023k) ;
 \draw (n123kk)-- (n13k20);
  \draw (n123kk) -- (n12k30); 

 \draw (n123ik) -- (n1023i) ;
 \draw (n123ik)-- (n13i20);
  \draw (n123ik) -- (n12i30);  
\draw (n123ik) -- (n1023k) ;
 \draw (n123ik)-- (n13k20);
  \draw (n123ik) -- (n12k30);

 \draw (n123ij) -- (n1023i) ;
 \draw (n123ij)-- (n13i20);
  \draw (n123ij) -- (n12i30);  
\draw (n123ij) -- (n1023j) ;
 \draw (n123ij)-- (n13j20);
  \draw (n123ij) -- (n12j30);
	
 \draw (n123jk) -- (n1023j) ;
 \draw (n123jk)-- (n13j20);
  \draw (n123jk) -- (n12j30);  
\draw (n123jk) -- (n1023k) ;
 \draw (n123jk)-- (n13k20);
  \draw (n123jk) -- (n12k30);

  \draw [] (n13i20) -- (n102030);
  \draw [] (n12i30)-- (n102030);
  \draw [] (n1023i)--(n102030);
 \draw [] (n1023j)  --  (n102030);
  \draw [] (n12j30) --  (n102030);
  \draw [] (n13j20) --(n102030);
 \draw [] (n1023k)  --  (n102030);
  \draw [] (n12k30) --  (n102030);
  \draw [] (n13k20) --(n102030);

\draw (hat1) -- (n123ii);
\draw (hat1) -- (n123jj);
\draw (hat1) -- (n123kk);
\draw (hat1) -- (n123ij);
\draw (hat1) -- (n123jk);
\draw (hat1) -- (n123ik);

\tikzstyle{every node}= [scale=0.45]

\node  at (-6.5,1.4) {\color{blue}$(1,2)^1$};
\node  at (-4.9,1.4) {\color{blue}$(1,3)^1$};
\node  at (-3.4,1.4) {\color{blue}$(2,3)^1$};
\node  at (-1.9,1.4) {\color{blue}$(1,2)^2$};
\node  at (-0.4,1.4) {\color{blue}$(1,3)^2$};
\node  at (1,1.4) {\color{blue}$(2,3)^2$};
\node  at (2.4,1.4) {\color{blue}$(1,2)^3$};
\node  at (3.5,1.4) {\color{blue}$(1,3)^3$};
\node  at (5,1.4) {\color{blue}$(2,3)^3$};

\node  at (-4,5) {\color{blue}$(1,4)^1$};
\node  at (-2,5) {\color{blue}$(1,4)^1$};
\node  at (-0.9,5) {\color{blue}$(1,4)^1$};
\node  at (0.1,5) {\color{blue}$(1,4)^1$};
\node  at (1,5) {\color{blue}$(1,4)^1$};
\node  at (2.2,5) {\color{blue}$(1,4)^1$};

\node  at (-7,3.5) {\color{blue}$(1,3)^1$};
\node  at (-6.3,3.3) {\color{blue}$(1,2)^1$};
\node  at (-5.7,3.3) {\color{blue}$(1,2)^1$};

\node  at (-4.2,3.7) {\color{blue}$(1,3)^2$};
\node  at (-4.15,3.4) {\color{blue}$(1,2)^2$};
\node  at (-3.8,3.1) {\color{blue}$(1,2)^2$};
\node  at (-2.9,3) {\color{blue}$(1,3)^1$};
\node  at (-2,3) {\color{blue}$(1,2)^1$};
\node  at (-1.3,3.1) {\color{blue}$(1,2)^1$};

\node  at (-5.4,2.9) {\color{blue}$(1,3)^3$};
\node  at (-4.3,2.8) {\color{blue}$(1,2)^3$};
\node  at (-3.2,2.7) {\color{blue}$(1,2)^3$};
\node  at (2.9,2.6) {\color{blue}$(1,3)^1$};
\node  at (4.1,2.65) {\color{blue}$(1,2)^1$};
\node  at (5.5,2.7) {\color{blue}$(1,2)^1$};

\node  at (-0.3,3.3) {\color{blue}$(1,3)^2$};
\node  at (0.3,3.2) {\color{blue}$(1,2)^2$};
\node  at (1.5,3.7) {\color{blue}$(1,2)^2$};

\node  at (2.1,3.8) {\color{blue}$(1,3)^3$};
\node  at (1,2.9) {\color{blue}$(1,2)^3$};
\node  at (2.5,3.5) {\color{blue}$(1,2)^3$};
\node  at (3,3.3) {\color{blue}$(1,3)^2$};
\node  at (3.8,3.2) {\color{blue}$(1,2)^2$};
\node  at (4,3.8) {\color{blue}$(1,2)^2$};

\node  at (5.3,3.6) {\color{blue}$(1,3)^3$};
\node  at (5.7,3.3) {\color{blue}$(1,2)^3$};
\node  at (6.8,3.5) {\color{blue}$(1,2)^3$};

\end{tikzpicture}
                \caption{Labeling $\bar \lambda$ on $\widehat{\Pi_3^3}$}
                \label{fig:ellabelingposet}
\end{figure}

The proof of the following theorem follows the same ideas of \cite[Theorem 3.2]{DleonWachs2013a}.
 \begin{theorem}\label{theorem:ellabelingposet}
 The labeling $\bar \lambda:\E(\widehat{\Pi_{n}^k})\rightarrow \Lambda _n$ defined above is an
EL-labeling of $\widehat{\Pi_{n}^k}$.  
\end{theorem}

Theorem \ref{theorem:elk} is then a corollary of Theorem \ref{theorem:ellabelingposet} and the 
following theorem linking lexicographic shellability and topology.

\begin{theorem}[Bj\"orner and Wachs \cite{BjornerWachs1996}] \label{theorem:elth}
Let $\bar \lambda$ be an EL-labeling  of a bounded   poset $P$. 
Then for all $x<y$ in $P$, 
\begin{enumerate}
\item the open interval $(x,y)$ is homotopy equivalent to a wedge of  spheres, where for each $r \in
\NN$ the number of spheres  of dimension $r$ is the number of ascent-free maximal chains of the
closed interval $[x,y]$ of length $r+2$. 
\item the set
$$\{\bar c \mid c \mbox{ is an ascent-free maximal chain of $[x,y]$ of length } r+2 \}$$
forms a basis for cohomology $\widetilde H^{r}((x,y))$, for all $r$.
\end{enumerate}
 \end{theorem}

Since the M\"obius invariant of a  bounded poset $P$ equals the reduced Euler characteristic of the
order complex of $\overline{P}$ (see \cite{Stanley2012}), Theorem \ref{theorem:elth} and the 
Euler-Poincar\'e formula imply the following corollary.
\begin{corollary}\label{corollary:elth}Let $P$ be a pure  EL-shellable poset of length $n$.  Then
\begin{enumerate}
\item $\overline{P}$ has the homotopy type of a wedge of spheres all of dimension $n-2$, where the 
number
of spheres is  $|\mu_P(\hat 0, \hat 1)|$. 
\item $P$ is Cohen-Macaulay, which means that $\widetilde H_i((x,y)) = 0$ for all $x <y$ in $P$ and 
$i <
l([x,y]) -2$. \end{enumerate}
\end{corollary}

\begin{theorem}[Theorem \ref{theorem:elk}]
 The poset $\widehat{\Pi_n^k}:= \Pi_n^k \cup \{\hat{1}\}$ is EL-shellable and hence Cohen-Macaulay. 
Consequently, for each
$\mu \in \wcomp_{n-1}$,
the order complex $\Delta((\hat{0},[n]^{\mu}))$ has the homotopy type of a wedge of
$(n-3)$-spheres.
\end{theorem}

In \cite{DotsenkoKhoroshkin2010} Dotsenko and Khoroshkin use operad theory to prove that the 
operad related to $\lie_k(n)$ is Koszul, which implies using Vallette's theory (\cite{Vallette2007}) 
that all
intervals of $\Pi_n^k$ are Cohen-Macaulay.  
Theorem \ref{theorem:elk} is an extension of their result.

\subsection{The ascent-free maximal chains from the EL-labeling}\label{section:ascentfreechains}

We will describe the ascent-free maximal chains of the maximal intervals $[\hat 
0,[n]^{\mu}]$ given by the EL-labeling
of Theorem~\ref{theorem:ellabelingposet}. A \emph{Lyndon tree}
is a labeled binary tree $(T,\sigma)$ such that for each internal node $x$ of $T$, the smallest leaf
label of the subtree $T_x$ rooted at $x$ is in the left subtree of $T_x$ and the second smallest
label is in the right subtree of $T_x$.  An alternative characterization of a   Lyndon tree is 
given in Proposition \ref{prop:valen} below.

For each internal node $x$ of a labeled binary tree, let $L(x)$ denote the left child of $x$ and 
$R(x)$
denote its right child.  For each node $x$ of a labeled binary tree $(T,\sigma)$ define
its  \emph{valency} $v(x)$ to be the smallest leaf label of the subtree rooted at $x$.  A Lyndon
tree is depicted in Figure \ref{fig:lyndonandvalency} illustrating the valencies of the
internal nodes. 

We say that a labeled  binary tree  is \emph{normalized} if  the leftmost leaf of each subtree has
the smallest label in the subtree. This is equivalent to requiring that for every internal node $x$,
\begin{align*}
 v(x)=v(L(x)).
\end{align*}

Note that a normalized tree can be thought of simply as a labeled nonplanar binary tree (or a 
phylogenetic 
tree)
that has been drawn in the plane following the convention above. We denote the set of normalized
labeled binary trees on label set $[n]$ by $\nor_n$ and the set of normalized binary trees on 
some arbitrary finite subset $A$ of $\PP$ by $\nor_A$. It is well-known that there are  
$(2n-3)!!:=1\cdot 3 \cdots (2n-3)$  phylogenetic trees 
on $[n]$ and so 
$|\nor_n|=(2n-3)!!$.

\begin{proposition}[{\cite[Proposition 5.6]{DleonWachs2013a}}] \label{prop:valen} Let $(T,\sigma)$ 
be a  labeled binary tree.  Then
$(T,\sigma)$ is a Lyndon tree if  and only if it is normalized and for every internal node $x$ of
$T$ we have 
\begin{align} \label{equation:lynnode} v(R(L(x))> v(R(x)).\end{align}
\end{proposition}

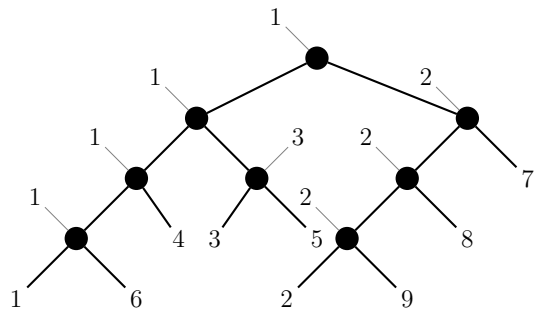
\begin{figure}[ht]
        \centering
        \begin{tikzpicture}[thick,scale=0.8]

\tikzstyle{every node}=[circle,fill,draw,inner sep=1pt, minimum width=10pt,scale=0.8]

    \draw [circle,color=black] (6,4)  node (i1)[pin=above left:\color{black}1]{};
    \draw [circle,color=black] (8.5,3)  node (i2)[pin=above left:\color{black}2]{};
    \draw [circle,color=black] (7.5,2)  node (i3)[pin=above left:\color{black}2]{};
    \draw [color=black] (6.5,1)  node (i4)[pin=above left:\color{black}2]{};
    \draw [color=black] (4,3)  node (i5)[pin=above left:\color{black}1]{};
    \draw [color=black] (5,2)  node (i6)[pin=above right:\color{black}3]{};
    \draw [circle,color=black] (3,2)  node (i7)[pin=above left:\color{black}1]{};
    \draw [color=black] (2,1)  node (i8)[pin=above left:\color{black}1]{};
\tikzstyle{every node}=[inner sep=1pt, minimum width=14pt,scale=0.8]

    \draw (4.3,1)  node (l1){3};
    \draw (5.5,0)  node (l2){2};
    \draw (1,0)  node (l3){1};
    \draw (3,0)  node (l4){6};
    \draw (6,1)  node (l5){5};
    \draw (3.7,1)  node (l6){4};
    \draw (7.5,0)  node (l7){9};
    \draw (9.5,2)  node (l8){7};
    \draw (8.5,1)  node (l9){8};

    \draw (i1) --  (i2) ;
    \draw (i1) --  (i5) ;
    \draw (i2) --  (i3) ;
    \draw (i2) --  (l8) ;
    \draw (i3) --  (i4) ;
    \draw (i3) --  (l9) ;
    \draw (i4) --  (l7) ;
    \draw (i4) --  (l2) ;
    \draw (i5) --  (i6) ;
    
    \draw (i5) --  (i7) ;
    \draw (i6) --  (l5) ;
    \draw (i6) --  (l1) ;
    \draw (i7) --  (l6) ;
    \draw (i7) --  (i8) ;
    \draw (i8) --  (l3) ;
    \draw (i8) --  (l4) ;
\end{tikzpicture}
 \caption{Example of a Lyndon tree. The numbers above the  lines correspond to the valencies of the
internal nodes}
\label{fig:lyndonandvalency}
  \end{figure}

 We will say that an internal node $x$ of  a labeled binary tree $(T,\sigma)$ is a \emph{Lyndon
node} if  (\ref{equation:lynnode}) holds.  
Hence Proposition~\ref{prop:valen} says that $(T,\sigma)$ is a Lyndon tree if and only if it
is normalized and all its internal nodes are Lyndon nodes.

A \emph{colored Lyndon tree} is a normalized binary tree such that for any node $x$ that is not
a Lyndon node it must happen that
\begin{align}\label{equation:lyndoncondition}
 \clr(L(x))>\clr(x).
\end{align}
For $\mu \in \wcomp_{n-1}$, let $\lyn_{\mu}$ be the set of colored Lyndon trees in $\BT_{\mu}$ and
$\lyn_{n}=\cup_{\mu \in \wcomp_{n-1}} \lyn_{\mu}$. 
Note that equation (\ref{equation:lyndoncondition}) implies that the monochromatic
Lyndon trees are just the classical Lyndon trees.

The set of bicolored Lyndon trees for $n=3$ is depicted in Figure \ref{fig:bicoloredlyndons}.

 \begin{figure}[ht]
        \centering
         \begin{tikzpicture}[thick,scale=0.6]
 \draw [circle,color=red] (0.5,2.5)  node (red){\small 2};
    \draw [color=blue] (0.5,1.7)  node (blue){ \small 1};
\tikzstyle{every node}=[fill, draw,inner sep=4pt, minimum width=1pt,scale=0.5]
    \draw [circle,color=red] (0,2.5)  node (r){};
    \draw [color=blue] (0,1.7)  node (b){};
\begin{scope}[xshift=-1cm,yshift=-1cm]

\tikzstyle{every node}=[fill, draw,inner sep=2pt,scale=0.8]
    \draw [color=blue] (3,1)  node (i1){};
    \draw [color=blue] (2,2)  node (i2){};

\tikzstyle{every node}=[inner sep=1pt, minimum width=14pt,scale=0.7]

    \draw (2,0)  node (m){$2$};
    \draw (4,0)  node (l1){$3$};
    \draw (1,1)  node (l2){$1$};
    
    \draw (m) --  (i1) ;
    \draw (i1) --  (l1) ;
    \draw (i1) --  (i2) ;
    \draw (i2) --  (l2) ;
\end{scope}

\begin{scope}[xshift=2.5cm,yshift=0]

\tikzstyle{every node}=[fill, draw,inner sep=2pt,scale=0.8]
    \draw [circle,color=red] (3,1)  node (i1){};
    \draw [color=blue] (2,2)  node (i2){};

\tikzstyle{every node}=[inner sep=1pt, minimum width=14pt,scale=0.7]

    \draw (2,0)  node (m){$2$};
    \draw (4,0)  node (l1){$3$};
    \draw (1,1)  node (l2){$1$};
    
    \draw (m) --  (i1) ;
    \draw (i1) --  (l1) ;
    \draw (i1) --  (i2) ;
    \draw (i2) --  (l2) ;
\end{scope}

\begin{scope}[xshift=6cm,yshift=0]

\tikzstyle{every node}=[fill, draw,inner sep=2pt,scale=0.8]
    \draw [color=blue] (3,1)  node (i1){};
    \draw [circle,color=red] (2,2)  node (i2){};

\tikzstyle{every node}=[inner sep=1pt, minimum width=14pt,scale=0.7]

    \draw (2,0)  node (m){$2$};
    \draw (4,0)  node (l1){$3$};
    \draw (1,1)  node (l2){$1$};
    
    \draw (m) --  (i1) ;
    \draw (i1) --  (l1) ;
    \draw (i1) --  (i2) ;
    \draw (i2) --  (l2) ;
\end{scope}
\begin{scope}[xshift=9.5cm,yshift=-1cm]

\tikzstyle{every node}=[fill, draw,inner sep=2pt,scale=0.8]
    \draw [circle,color=red] (3,1)  node (i1){};
    \draw [circle,color=red] (2,2)  node (i2){};

\tikzstyle{every node}=[inner sep=1pt, minimum width=14pt,scale=0.7]

    \draw (2,0)  node (m){$2$};
    \draw (4,0)  node (l1){$3$};
    \draw (1,1)  node (l2){$1$};
    
    \draw (m) --  (i1) ;
    \draw (i1) --  (l1) ;
    \draw (i1) --  (i2) ;
    \draw (i2) --  (l2) ;
\end{scope}

\begin{scope}[xshift=0,yshift=-5cm]

\tikzstyle{every node}=[fill, draw,inner sep=2pt,scale=0.8]
    \draw [color=blue] (1,1)  node (i1){};
    \draw [color=blue] (2,2)  node (i2){};

\tikzstyle{every node}=[inner sep=1pt, minimum width=14pt,scale=0.7]

    \draw (0,0)  node (m){$1$};
    \draw (2,0)  node (l1){$3$};
    \draw (3,1)  node (l2){$2$};

    \draw (m) --  (i1) ;
    \draw (i1) --  (l1) ;
    \draw (i1) --  (i2) ;
    \draw (i2) --  (l2) ;
\end{scope}

\begin{scope}[xshift=3.5cm,yshift=-6cm]
\tikzstyle{every node}=[fill, draw,inner sep=2pt,scale=0.8]
    \draw [circle,color=red] (1,1)  node (i1){};
    \draw [color=blue] (2,2)  node (i2){};

\tikzstyle{every node}=[inner sep=1pt, minimum width=14pt,scale=0.7]

    \draw (0,0)  node (m){$1$};
    \draw (2,0)  node (l1){$3$};
    \draw (3,1)  node (l2){$2$};

    \draw (m) --  (i1) ;
    \draw (i1) --  (l1) ;
    \draw (i1) --  (i2) ;
    \draw (i2) --  (l2) ;
\end{scope}

\begin{scope}[xshift=7cm,yshift=-6cm]
\tikzstyle{every node}=[fill, draw,inner sep=2pt,scale=0.8]
    \draw [color=blue] (1,1)  node (i1){};
    \draw [circle,color=red] (2,2)  node (i2){};

\tikzstyle{every node}=[inner sep=1pt, minimum width=14pt,scale=0.7]

    \draw (0,0)  node (m){$1$};
    \draw (2,0)  node (l1){$3$};
    \draw (3,1)  node (l2){$2$};

    \draw (m) --  (i1) ;
    \draw (i1) --  (l1) ;
    \draw (i1) --  (i2) ;
    \draw (i2) --  (l2) ;
\end{scope}
\begin{scope}[xshift=10.5cm,yshift=-5cm]
\tikzstyle{every node}=[fill, draw,inner sep=2pt,scale=0.8]
    \draw [circle,color=red] (1,1)  node (i1){};
    \draw [circle,color=red] (2,2)  node (i2){};

\tikzstyle{every node}=[inner sep=1pt, minimum width=14pt,scale=0.7]

    \draw (0,0)  node (m){$1$};
    \draw (2,0)  node (l1){$3$};
    \draw (3,1)  node (l2){$2$};
    
    \draw (m) --  (i1) ;
    \draw (i1) --  (l1) ;
    \draw (i1) --  (i2) ;
    \draw (i2) --  (l2) ;
\end{scope}

\begin{scope}[xshift=5cm,yshift=-3cm]
\tikzstyle{every node}=[fill, draw,inner sep=2pt,scale=0.8]
    \draw [circle,color=red] (1,1)  node (i1){};
    \draw [color=blue] (2,2)  node (i2){};

\tikzstyle{every node}=[inner sep=1pt, minimum width=14pt,scale=0.7]

    \draw (0,0)  node (m){$1$};
    \draw (2,0)  node (l1){$2$};
    \draw (3,1)  node (l2){$3$};

    \draw (m) --  (i1) ;
    \draw (i1) --  (l1) ;
    \draw (i1) --  (i2) ;
    \draw (i2) --  (l2) ;

\end{scope}
\end{tikzpicture}
 \caption{Set of bicolored Lyndon trees for $n=3$}
\label{fig:bicoloredlyndons}
  \end{figure}
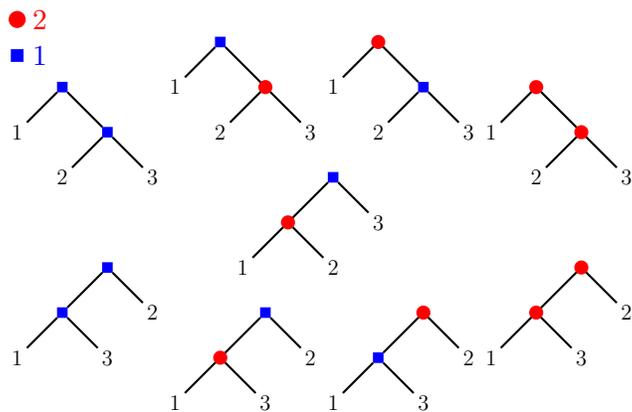

We will show that the ascent-free maximal chains of the EL-labeling of $[\hat 0,[n]^{\mu}]$ 
given in Theorem~\ref{theorem:ellabelingposet} are of the form $c(T,\sigma,\tau)$, where $(T,\sigma)
\in \lyn_{\mu}$ 
and $\tau$ is the linear extension of the internal nodes of $T$, which we now describe:
It is easy to see that there is a unique linear extension of the internal notes of $(T,\sigma) \in 
\BT_{\mu}$ in which the valencies of the nodes weakly decrease.   Let $\tau_{T,\sigma}$ denote the
permutation that  induces this linear extension.

\begin{theorem}\label{thm:ascfreeEL}
The set $\{c(T,\sigma,\tau_{T,\sigma}) \mid \, (T,\sigma) \in \lyn_{\mu}\}$ is the set of 
ascent-free
maximal chains of the EL-labeling of $[\hat 0, [n]^\mu]$ given in Theorem
\ref{theorem:ellabelingposet}.
\end{theorem}

\begin{proof}

We begin by showing that $c:=c(T,\sigma,\tau)$ is ascent-free whenever $(T,\sigma) \in \lyn_{\mu}$
and $\tau= \tau_{T,\sigma}$ .  Let  $x_i$ be the $i$th internal node of $T$ in postorder.  Then by
the definition of $\tau_{T,\sigma}$,
\begin{align} \label{eq:val} v(x_{\tau(1)}) \ge v(x_{\tau(2)}) \ge \dots \ge 
v(x_{\tau(n-1)}),\end{align} 
where $v$ is the valency.  For each $i$, the $i$th letter of  the label word $\bar \lambda(c)$ is 
given
by 
$$\bar \lambda_i(c) = (v(L(x_{\tau(i)})), v(R(x_{\tau(i)})))^{u_i}= (v(x_{\tau(i)}),
v(R(x_{\tau(i)})))^{u_i},$$ where  $u_i = \clr(x_{\tau(i)})$. 
Note that since $(T,\sigma)$ is normalized, $v(R(x_{\tau(i)})) \ne v(R(x_{\tau(i+1)}))$ for all $i 
\in [n-1]$.
Now suppose the word $\bar 
\lambda(c)$ has
an ascent at $i$.  Then it follows from (\ref{eq:val})
that 
\begin{align} \label{eq:ascent} v(x_{\tau(i)}) = v(x_{\tau(i+1)}),\,\,\, v(R(x_{\tau(i)}))<
v(R(x_{\tau(i+1)})),\,\mbox{ and } \, u_i \le u_{i+1} .\end{align}  
The equality of valencies implies that $x_{\tau(i)}= L(x_{\tau(i+1)})$ since $(T,\sigma)$ is
normalized and $\tau$ is a linear extension.  
Hence by (\ref{eq:ascent}),  $$ v(R(L(x_{\tau(i+1)})))<v(R(x_{\tau(i+1)})).$$
It follows that $x_{\tau(i+1)}$ is not a Lyndon node.  So by the coloring restriction on colored
Lyndon trees
\begin{align*}
 u_i=\clr(x_{\tau(i)})=\clr(L(x_{\tau(i+1)}))>\clr(x_{\tau(i+1)})=u_{i+1},
\end{align*}
which contradicts (\ref{eq:ascent}).  Hence the chain $c$ is ascent-free.

Conversely, assume $c$ is an ascent-free maximal chain of $[\hat 0, [n]^{\mu}]$.  Then $c =
c(T,\sigma,\tau)$ for some  bicolored labeled tree $(T,\sigma)$ and some permutation $\tau \in
\sym_{n-1}$.  We can assume without loss of generality that $(T,\sigma)$ is normalized.  Since $c$
is ascent-free, (\ref{eq:val}) holds.  This implies that $\tau$ is the unique permutation that
induces the valency-decreasing linear extension, namely $ \tau_{T,\sigma}$.  

If all internal nodes of $(T,\sigma)$ are Lyndon nodes we are done.  So  let $i\in [n-1]$ be such
that  $x_{\tau(i)}$ is not a Lyndon node.   That is $$  v(R(L(x_{\tau(i)})))< v(R(x_{\tau(i)})) .$$
  Since $(T,\sigma)$ is normalized and (\ref{eq:val}) holds,  $ L(x_{\tau(i)})= x_{\tau(i-1)}$. 
Hence,
$v(R(x_{\tau(i-1)}))<v(R(x_{\tau(i)})) $.  Since $(T,\sigma)$ is normalized we also have
$v(L(x_{\tau(i-1)})) = v(L(x_{\tau(i)}))$.  Since $c$ is ascent-free  we must have
that
\begin{align*}
\clr(x_{\tau(i-1)})>\clr(x_{\tau(i)}),
\end{align*}
which is precisely what we need to conclude that
$(T,\sigma)$ is a colored Lyndon tree.
\end{proof}

From Theorem \ref{theorem:elth}, Theorem \ref{thm:ascfreeEL}  and Corollary  \ref{corollary:elth}, 
we have the following corollary.
\begin{corollary}\label{corollary:muintervalslyn} For all $n \ge 
1$ and for all $\mu \in 
\wcomp_{n-1}$ with 
$\supp(\mu)\subseteq [k]$,  the order complex $\Delta((\hat 0, [n]^{\mu}))$ has the homotopy type 
of a wedge of 
$|\lyn_{\mu}|$ spheres of dimension $n-3$. Consequently,
\[
 \dim \widetilde H^{n-3}((\hat 0, [n]^{\mu}))=|\lyn_{\mu}|
\]
and
\[ \bar \mu_{\Pi_{n}^{k}}(\hat{0},[n]^{\mu})=(-1)^{n-1}|\lyn_{\mu}|.\]
\end{corollary}

\section{The dimension of $\lie(\mu)$}\label{section:binarystirling}

In this section we present various formulas for the dimension of $\lie(\mu)$.  We begin by using 
the isomorphism between $\lie(\mu)$ and $\widetilde H^{n-3}((\hat{0},[n]^{\mu}))$ of Theorem 
\ref{theorem:liehomisomorphism} to transfer information on $\widetilde 
H^{n-3}((\hat{0},[n]^{\mu}))$ obtained in the previous section to $\lie(\mu)$.

\begin{theorem}[Theorem \ref{theorem:compositionalinverse}]\label{theorem:compositionalinverse2}
We have
  \begin{align*}
   \sum_{n\ge1}\sum_{\mu \in \wcomp_{n-1}}\dim 
\lie(\mu)\,\xx^{\mu}\frac{y^n}{n!} =\left [ \sum_{n\ge1}(-1)^{n-1} 
h_{n-1}(\xx)\frac{y^n}{n!} \right ] ^{<-1>},
  \end{align*}
 \end{theorem}
 
\begin{proof}
 From Corollary \ref{corollary:muintervalslyn}  we have that 
\begin{align*}
 \bar \mu_{\Pi_{n}^{k}}(\hat{0},[n]^{\mu})=(-1)^{n-1} \dim \widetilde H^{n-3}((\hat{0},[n]^{\mu})).
\end{align*}
The theorem now follows from Theorems \ref{theorem:liehomisomorphism} and  
\ref{theorem:compositionalinversemu} when we let $k$ get large.
\end{proof}

We have a combinatorial description for the 
dimension of $\lie(\mu)$.

\begin{theorem}\label{theorem:dimliemu} For all $n \ge 1$ and $\mu \in \wcomp_{n-1}$,
\[
 \dim \lie(\mu)=|\lyn_{\mu}|.
\]
\end{theorem}
\begin{proof}
 We know from Corollary \ref{corollary:muintervalslyn} that $\dim \widetilde 
H^{n-3}((\hat{0},[n]^{\mu}))=|\lyn_{\mu}|$. Hence, the isomorphism of 
Theorem \ref{theorem:liehomisomorphism} proves the theorem.
\end{proof}

\subsection{Lyndon type of a normalized tree}\label{section:lyndontype}

With a normalized tree $\Upsilon \in \nor_n$ we can associate a (set) partition 
$\pi^{\lyn}(\Upsilon)$ of the set of 
internal nodes of $\Upsilon$, defined to be the finest partition satisfying the condition:
\begin{itemize}
 \item for every internal node $x$ that is not Lyndon, $x$ and $L(x)$ belong to the same block 
of $\pi^{\lyn}(\Upsilon)$.
\end{itemize}
For the tree in Figure \ref{fig:coloredlyndontype}, the shaded rectangles indicate the 
blocks of $\pi^{\lyn}(\Upsilon)$. 

Note that the coloring condition (\ref{equation:lyndoncondition}) implies that in 
a 
colored Lyndon tree $\Upsilon$ there are no repeated colors in each block $B$ 
of the partition 
$\pi^{\lyn}(\Upsilon)$ 
associated with $\Upsilon$. Hence after choosing a 
set of $|B|$ colors for the internal nodes in $B$ there is 
a unique way 
to assign the different colors such
that the colored tree $\Upsilon$ is a colored Lyndon tree (the colors must decrease towards 
the 
root in each block of $\pi^{\lyn}(\Upsilon)$). 

Define the \emph{Lyndon type}  $\lyndonlambda(\Upsilon)$ of a normalized tree 
(colored or 
uncolored) $\Upsilon$ to be the (integer) partition whose parts are the block sizes of the 
partition $\pi^{\lyn}(\Upsilon)$. For the tree $\Upsilon$ in Figure 
\ref{fig:coloredlyndontype}, we have $\lyndonlambda(\Upsilon)=(3,2,2,1)$. 

\begin{figure}[ht]
        \centering
        \usetikzlibrary{shapes,snakes}

\begin{tikzpicture}[thick,scale=0.8]

\begin{scope}[xshift=0cm,yshift=1cm]
 \draw [color=blue] (1.5,4)  node (blue){ $1$};
\draw [circle,color=red] (1.5,3.5)  node (red){ $2$};
 \draw [color=brown] (1.5,2.9)  node (brown){$3$};
\tikzstyle{every node}=[fill, draw,inner sep=4pt, minimum width=1pt,scale=0.8]
    
    \draw [color=blue] (1,4)  node (b){};
\draw [circle,color=red] (1,3.5)  node (r){};
    \draw [diamond,color=brown] (1,2.9)  node (g){};
\end{scope}

\node[fill=blue!20,blue!20,draw,rectangle,rounded corners,rotate=45, minimum width=70pt,minimum 
height=30pt,scale=0.8] at  (2.5,1.5) {};
\node[fill=blue!20,blue!20,draw,rectangle,rounded corners,rotate=26.565, minimum width=90pt,minimum 
height=30pt,scale=0.8] at  (5,3.5) {};
\node[fill=blue!20,blue!20,draw,rectangle,rounded corners,rotate=45, minimum width=120pt,minimum 
height=30pt,scale=0.8] at  (7.5,2) {};
\node[fill=blue!20,blue!20,draw,rectangle,rounded corners,rotate=45, minimum width=30pt,minimum 
height=30pt,scale=0.8] at  (5,2) {};
\tikzstyle{every node}=[fill,draw,inner sep=1pt, minimum width=10pt,scale=0.8]

    \draw [color=blue] (6,4)  node (i1)[pin=above left:\color{black}1]{N};
    \draw [color=blue] (8.5,3)  node (i2)[pin=above left:\color{black}3]{N};
    \draw [circle,color=red] (7.5,2)  node (i3)[pin=above left:\color{black}3]{n};
    \draw [diamond, color=brown] (6.5,1)  node (i4)[pin=above left:\color{black}3]{n};
    \draw [diamond,color=brown] (4,3)  node (i5)[pin=above left:\color{black}1]{n};
    \draw [diamond,color=brown] (5,2)  node (i6)[pin=above right:\color{black}2]{n};
    \draw [color=blue] (3,2)  node (i7)[pin=above left:\color{black}1]{N};
    \draw [circle,color=red] (2,1)  node (i8)[pin=above left:\color{black}1]{n};
\tikzstyle{every node}=[inner sep=1pt, minimum width=14pt,scale=0.8]

    \draw (4.3,1)  node (l1){2};
    \draw (5.5,0)  node (l2){3};
    \draw (1,0)  node (l3){1};
    \draw (3,0)  node (l4){4};
    \draw (6,1)  node (l5){5};
    \draw (3.7,1)  node (l6){6};
    \draw (7.5,0)  node (l7){7};
    \draw (9.5,2)  node (l8){9};
    \draw (8.5,1)  node (l9){8};

    \draw (i1) --  (i2) ;
    \draw (i1) --  (i5) ;
    \draw (i2) --  (i3) ;
    \draw (i2) --  (l8) ;
    \draw (i3) --  (i4) ;
    \draw (i3) --  (l9) ;
    \draw (i4) --  (l7) ;
    \draw (i4) --  (l2) ;
    \draw (i5) --  (i6) ;
    
    \draw (i5) --  (i7) ;
    \draw (i6) --  (l5) ;
    \draw (i6) --  (l1) ;
    \draw (i7) --  (l6) ;
    \draw (i7) --  (i8) ;
    \draw (i8) --  (l3) ;
    \draw (i8) --  (l4) ;
\end{tikzpicture}
 \caption{Example of a colored Lyndon tree of type (3,2,2,1). The numbers above the  lines 
correspond to the valencies of the internal nodes}
\label{fig:coloredlyndontype}
  \end{figure}
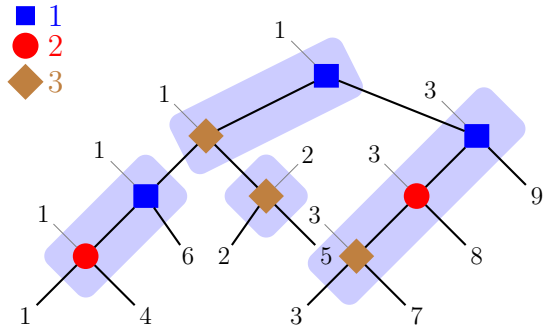
Let $e_{\lambda}(\xx)$ 
be the elementary symmetric function associated with the partition $\lambda$. 

\begin{theorem}\label{theorem:dimensionslyndontype}For all $n$,
 \begin{align*}
  \sum_{\mu \in \wcomp_{n-1}}\dim \lie(\mu)\,\xx^{\mu}= \sum_{\Upsilon \in \nor_n} 
e_{\lyndonlambda(\Upsilon)}(\xx).
 \end{align*}
 \end{theorem}

\begin{proof}
For a colored labeled binary tree $\Psi$ we define the \emph{content} $\mu(\Psi)$ of $G$ as the 
weak 
composition $\mu$ where $\mu(i)$ is the number of internal nodes of $\Psi$ that have color $i$.
Recall that  $\widetilde \Psi$ denotes the underlying uncolored labeled binary tree of $\Psi$.
Note that the comments above imply that for $\Upsilon \in \nor_n$, the generating 
function of
colored Lyndon trees associated with $\Upsilon$ is
\begin{align}\label{equation:contribution}
\sum_{\substack{\Psi \in \lyn_n\\ \widetilde \Psi = \Upsilon}}\xx^{\mu(\Psi)}= 
e_{\lyndonlambda(\Upsilon)}(\xx).
\end{align}
Indeed the internal 
nodes in a block of size $i$ in the partition  $\pi^{\lyn}(\Upsilon)$ can be colored 
uniquely with any set of $i$ different colors and so the contribution from this block of   
$\pi^{\lyn}(\Upsilon)$  to 
the generating function in (\ref{equation:contribution}) is $e_i(\xx)$.

By Theorem \ref{theorem:dimliemu},
\begin{align*}
\sum_{\mu \in \wcomp_{n-1}}\dim \lie(\mu)\,\xx^{\mu}&=\sum_{\mu \in 
\wcomp_{n-1}}|\lyn_{\mu}|\,\xx^{\mu}\\
          &=\sum_{\Psi \in \lyn_{n}}\xx^{\mu(\Psi)}\\
          &=\sum_{\Upsilon \in \nor_{n}}\sum_{\substack{\Psi \in \lyn_n\\ \widetilde \Psi = 
\Upsilon}}\xx^{\mu(\Psi)}\\
         &=\sum_{\Upsilon \in \nor_{n}}e_{\lyndonlambda(\Upsilon)}(\xx),
\end{align*}
with the last equation following from (\ref{equation:contribution}).
\end{proof}

\subsection{Stirling permutations}

A \emph{Stirling permutation} on the 
set $[n]$ is a permutation of the multiset $\{1,1,2,2,\cdots,n,n\}$ such that for all $m \in [n]$,
all numbers between the two occurrences of $m$ are larger than $m$.
The set of Stirling permutations on $[n]$ will be 
denoted by $\Q_n$. For example, the permutation
$12332144$ is in $\Q_n$ but $43341122$ is not since $3$ is between the two ocurrences of $4$. 
Stirling permutations were introduced by Stanley and Gessel in 
\cite{StanleyGessel1978} and have been also
studied by B\'ona, Park, Janson, Kuba, Panholzer and others (see
\cite{Bona2009,Park1994-1,JansonKubaPanholzer2011,HaglundVisontai2012}).  For an 
arbitrary subset $A:=\{a_1,a_2,\dots,a_n\}$ of positive integers, we denote by $\Q_A$, 
the set of Stirling permutations of $A$; that is,  
permutations of the multiset $\{a_1,a_1,a_2,a_2,\dots,a_n,a_n\}$, satisfying the 
condition 
above. 

It is known
that $|\Q_{n-1}|=(2n-3)!!$. So this set of Stirling permutations is equinumerous with the set 
$\nor_n$ of 
normalized binary trees with label set $[n]$. We will present an explicit bijection between these 
two sets. 
Moreover, this bijection has some nice properties that allow us to translate the previous results 
in this section to the language of Stirling permutations and to ultimately prove Theorem 
\ref{theorem:dimensionscombtype}.

\subsection{Type of a Stirling permutation}\label{section:stirlingtype}

A \emph{segment} $u$ of a Stirling permutation $\theta=\theta_1\theta_2\cdots\theta_{2n}$ is a 
subword  of $\theta$ of the form $u=\theta_{i}\theta_{i+1}\cdots\theta_{i+\ell}$, i.e., 
all the letters of $u$ are adjacent in $\theta$. A \emph{block} in a 
Stirling permutation $\theta$ is a segment of $\theta$ that starts and ends with the same letter. 
For example, $455774$ is a block of $12245577413366$. We define $B_{\theta}(a)$ to be the block of 
$\theta$ that starts and ends with the letter $a$, and define $\mathring{B_{\theta}}(a)$  to be the 
segment 
obtained from $B_{\theta}(a)$ after removing the two occurrences of the letter $a$. For example, 
$B_{\theta}(1)=1224557741$ in $\theta=12245577413366$ and $\mathring{B_{\theta}}(1)=22455774$. 

We call $(a,b)$ an \emph{ascending adjacent pair} if $a<b$ and the blocks $B_{\theta}(a)$ and 
$B_{\theta}(b)$ are adjacent in $\theta$, i.e., $\theta=\theta^{\prime} 
B_{\theta}(a)B_{\theta}(b)\theta^{\prime\prime}$. An 
\emph{ascending adjacent sequence} of $\theta$ of length $k$ is a 
subsequence  $a_1<a_2<\cdots<a_k$ such that $(a_j,a_{j+1})$  is an ascending adjacent pair for 
$j=1,\dots,k-1$. Similarly, for a Stirling permutation $\theta \in \Q_n$ we call $(a,b)$ a
\emph{terminally nested pair} if $a<b$ and the block $B_{\theta}(b)$ is the last 
block in $\mathring{B_{\theta}}(a)$, i.e., 
$\mathring{B_{\theta}}(a)=\theta^{\prime} 
B_{\theta}(b)$ for some Stirling permutation $\theta^{\prime}$. A 
\emph{terminally nested sequence} of 
$\theta$ of length $k$ is a 
subsequence  $a_1<a_2<\cdots<a_k$ such that $(a_j,a_{j+1})$  is a terminally nested pair for 
$j=1,\dots,k-1$.

We can associate a \emph{type} to a Stirling permutation $\theta \in \Q_n$ in two ways. We define
the \emph{ascending adjacent type} $\aalambda(\theta)$, to be the partition whose parts are the 
lengths of maximal ascending adjacent sequences; and  we define the \emph{terminally nested 
type} $\tnlambda(\theta)$, to be the partition whose parts are the lengths of maximal terminally 
nested sequences.
We will show that these two types are equinumerous in $\Q_n$.

\begin{example}
If $\theta=158851244667723399$, then the maximal ascending adjacent sequences
are $1239$, $467$, $5$ and $8$; then   
$\aalambda(\theta)=(4,3,1,1)$, which is a partition of $n=9$. Also the maximal terminally nested 
sequences are $158$, $27$, $3$, $4$, $6$ and $9$;  then 
$\tnlambda(\theta)=(3,2,1,1,1,1)$, which is also a partition of $n=9$.
\end{example}

It is 
easy to see that every Stirling permutation has a unique factorization 
$\theta=B_{\theta}(a_1)B_{\theta}(a_2)\cdots B_{\theta}(a_{\ell})$ into adjacent blocks. 
We call this factorization the \emph{block factorization} of $\theta$.
For example, $12245577413366$ has a block factorization $1224557741-33-66$. 
A \emph{Stirling factorization} of a Stirling 
permutation $\theta$ is a decomposition 
$\theta=\theta^{1}\theta^{2}\cdots \theta^{\ell}$, such that $\theta^{i}$ is a  
Stirling permutation for all $i$. Note that the block factorization of $\theta$ is the finest 
Stirling factorization.

Denote by $\kappa(a):=a_k$, the largest letter of the maximal terminally nested sequence 
$a=a_1<a_2<\cdots<a_k$ of $B_{\theta}(a)$ that contains $a$. In $\theta=158851244667723399$, we 
have for example that $\kappa(1)=8$, $\kappa(2)=7$ and $\kappa(7)=7$.
 We define the following two types of restricted Stirling factorizations:
 \begin{itemize}
 \item The \emph{ascending adjacent factorization} of $\theta$ is the Stirling factorization 
$\theta=\theta^{1}\theta^{2}$ in which $\theta^1$ is the shortest nonempty prefix of $\theta$ such 
that if $\theta^1=\alpha B_{\theta}(a)$ and $\theta^2=B_{\theta}(b)\beta$ then $a>b$.
For example if $\theta=133155442662$, then the ascending adjacent factorization of 
$\theta$ is $133155-442662$.
 
  \item The \emph{terminally nested factorization} of $\theta$ is the Stirling factorization 
$\theta=\theta^{1}\theta^{2}$ in which $\theta^1$ is the shortest nonempty prefix of $\theta$ 
such that if $\theta^1=B_{\theta}(a)\alpha$ and $\theta^2=B_{\theta}(b)\beta$ then $\kappa(a)>b$. 
In 
the case of 
$\theta=133155442662$, the terminally nested factorization of  $\theta$ is 
$13315544-2662$.
 \end{itemize}

 An \emph{irreducible AA-word} is a Stirling permutation that has no nontrivial ascending adjacent 
factorization.  It is not difficult to see that an irreducible AA-word is a Stirling permutation of 
the form 
\begin{align*}
 B_{\theta}(a_1)B_{\theta}(a_2)\cdots 
B_{\theta}(a_k)=
 a_1\tau_1a_1\,a_2\tau_2a_2\cdots a_{k-1}\tau_{k-1}a_{k-1}\,a_k\tau_ka_k,
\end{align*}
where $a_1<a_2< \dots <a_k$ and $\tau_i$ are Stirling permutations for each $i$. 

An \emph{irreducible TN-word}  is a Stirling permutation that has no nontrivial terminally nested 
factorization. It is not difficult to see that an irreducible TN-word is a Stirling permutation of 
the 
form 
\begin{align*}
B_{\theta}(a)\alpha
\end{align*}
where $\kappa(a)<a^{\prime}$ for any letter $a^{\prime}$ in $\alpha$. Equivalently, an irreducible 
TN-word is a Stirling permutation of 
the 
form 
\begin{align*}
a_1\tau_1a_2\tau_2a_{k-1}\tau_{k-1}a_ka_ka_{k-1}\cdots a_2a_1\tau_k
\end{align*}
where $a_1<a_2< \dots <a_k$ and $\tau_i$ are Stirling permutations for each $i$ with 
$a_k<a^{\prime}$ for 
any letter $a^{\prime}$ in $\tau_k$.

The \emph{complete ascending adjacent (terminally nested) factorization} of $\theta$ is the 
factorization $\theta=\theta^1\theta^2\cdots \theta^l$ that we obtain by factoring 
$\theta$ into $\theta^1\theta^2$ by the ascending adjacent (resp., terminally nested) 
factorization and then recursively applying the same procedure to $\theta^2$.

Let $A$ be a subset of the positive integers. We define a map  $\xi:\Q_A \rightarrow \Q_A$ 
recursively as follows:

\begin{enumerate}
 \item If $\theta=mm$ then $\xi(\theta)=mm$.
 \item If $\theta$ is an irreducible AA-word $a_1\tau_1a_1a_2\tau_2a_2\cdots 
a_{k-1}\tau_{k-1}a_{k-1}a_k\tau_ka_k$ then
\begin{align*}
 \xi(\theta)=a_1\xi(\tau_1)a_2\xi(\tau_2)\cdots a_{k-1}\xi(\tau_{k-1})a_ka_ka_{k-1}\cdots a_2a_1
\xi(\tau_k).
 \end{align*}
\item If $\theta=\theta^1\theta^2\cdots \theta^l$ is the complete ascending adjacent 
factorization of
$\theta$ then 
\[
 \xi(\theta)=\xi(\theta^1)\xi(\theta^2)\cdots \xi(\theta^l).
\]
\end{enumerate}

Step (2) guarantees that $\xi$ is well-defined. Indeed, in an 
irreducible AA-word of the form given in (2), we have $a_s<a_{s+1}<\cdots<a_k$ for any 
$s$. Hence, we are inserting only letters that are greater than $a_s$ between the two 
occurrences of $a_s$. 

The map $\xi$ is in fact a bijection and it is not difficult to check that
its inverse $\xi^{-1}:\Q_A \rightarrow \Q_A$ is defined by:

\begin{enumerate}
 \item If $\theta=mm$ then $\xi^{-1}(\theta)=mm$.
 \item If $\theta$ is an irreducible TN-word 
$a_1\tau_1a_2\tau_2 \cdots a_{k-1}\tau_{k-1}a_ka_ka_{k-1}\cdots a_2a_1\tau_k$ then
\begin{align*}
 \xi^{-1}(\theta)=a_1\xi^{-1}(\tau_1)a_1a_2\xi^{-1}(\tau_2)a_2\cdots
a_{k-1}\xi^{-1}(\tau_{k-1})a_{k-1}a_k\xi^{-1}(\tau_k)a_k.
\end{align*}
\item If $\theta=\theta^1\theta^2\cdots \theta^l$ is the complete terminally nested factorization 
of
$\theta$ then 
\begin{align*}
 \xi^{-1}(\theta)=\xi^{-1}(\theta^1)\xi^{-1}(\theta^2)\cdots \xi^{-1}(\theta^l).
\end{align*}
\end{enumerate}
Step (2) guarantees that $\xi^{-1}$ is well-defined since in an 
irreducible TN-word, $a_k<b$ for any letter $b$ in $\tau_k$.

\begin{example}
Consider $\theta=233772499468861551$. Its complete ascending factorization is
 $23377249946886-1551$; then
 \begin{align*}
  \xi(\theta)	&=\xi(23377249946886-1551)\\
		&=\xi(23377249946886)-\xi(1551)\\
		&=2\xi(3377)4\xi(99)6642\xi(88)-11\xi(55)\\
		&=237734996642881155.
 \end{align*}
Note that the maximal ascending  adjacent sequences of $\theta$ are $(246,37,1,5,8,9)$ which
are also the maximal terminally nested sequences of $\xi(\theta)$. These observations hold in 
general.
\end{example}

\begin{proposition}\label{proposition:propxi}The map $\xi:\Q_A \rightarrow \Q_A$ is a well-defined 
bijection that satisfies:	
 \begin{enumerate}
  \item $(i,j)$ is an ascending adjacent pair in $\theta$ if and only if $(i,j)$ is a terminally 
nested pair in $\xi(\theta)$,
\item $\tnlambda(\xi(\theta))=\aalambda(\theta)$.
  
 \end{enumerate}

\end{proposition}
\begin{proof}
 Note that if $\theta=\theta^1\theta^2\cdots \theta^l$ is the complete ascending adjacent 
factorization of a Stirling permutation $\theta$, then an ascending adjacent pair can only occur 
within one of the factors $\theta^i$. Similarly, if $\theta=\theta^1\theta^2\cdots \theta^l$ is 
the complete terminally nested factorization of a Stirling permutation $\theta$, then a 
terminally nested pair can only occur within one of the factors $\theta^i$. Hence, without loss of 
generality, as a consequence of step (3) in the definitions of $\xi$ and $\xi^{-1}$, we can assume 
that the word $\theta$ is an irreducible AA-word or an irreducible TN-word. Then the  
first assertion follows directly from step (2) in the definitions of $\xi$ and $\xi^{-1}$ and 
induction on the length of $\theta$. The 
second assertion is an immediate consequence of the first.
\end{proof}

From Proposition \ref{proposition:propxi} we see that $\aalambda$ and $\tnlambda$ are 
equidistributed on $\Q_n$.

\subsection{A bijection between normalized trees and Stirling permutations}

Let $\hat{\Q}_n$ be the set of permutations  $\theta \in \Q_n$ where
$\theta_1=\theta_{2n}=1$. There is a natural bijection $\redu:\hat{\Q}_n \rightarrow \Q_{n-1}$ 
obtained by 
removing the leading and trailing $1$ from $\theta=1\theta'1$ and then reducing the word 
$\theta'$ by decreasing every letter in $\{2,\dots,n\}$ by one. For example, 
$\redu(12332441)=122133$. In greater generality, for $A$ a subset of the positive integers, let 
$\hat{\Q}_A$ be the set of Stirling permutations of $A$ such that both the 
first
and last letter of the permutation is $\min A$.  Define  the map\footnote{The same map has 
appeared before in \cite{Dotsenko2012}.} $\tilde \gamma: \nor_A \rightarrow
\hat{\Q}_A$ recursively by:
\begin{enumerate}
 \item If $\Upsilon=(\bullet, m)$ then $\tilde \gamma(\Upsilon)=mm$.
 \item If $\Upsilon$ is of the form 
 
 
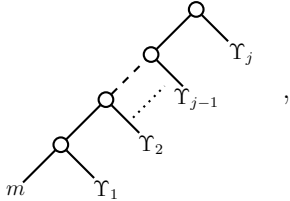
\begin{center}
 \begin{tikzpicture}[thick,scale=0.6]

\tikzstyle{every node}=[draw,scale=0.5]

    \draw [circle,radius=20pt,color=black] (1,1)  node (i1){};
    \draw [circle,color=black] (2,2)  node (i2){};
    \draw [circle,color=black] (4,4)  node (i4){};
    \draw [circle,color=black] (3,3)  node (i3){};

\tikzstyle{every node}=[inner sep=1pt, minimum width=14pt,scale=0.7]

    \draw (0,0)  node (m){$m$};
    \draw (2,0)  node (l1){$\Upsilon_1$};
    \draw (3,1)  node (l2){$\Upsilon_2$};
    \draw (4,2)  node (l3){$\Upsilon_{j-1}$};
    \draw (5,3)  node (l4){$\Upsilon_j$};
    \draw (6,2) node (comma){\Large ,};

    \draw (m) --  (i1) ;
    \draw (i1) --  (l1) ;
    \draw (i2) --  (l2) ;
    \draw (i3) --  (l3) ;
    \draw (i4) --  (l4) ;
    \draw (i1) --  (i2) ;
    \draw [dashed, thick] (i2) --  (i3) ;
    \draw [dotted, thick] (2.6,1.6) --  (3.3,2.3) ;

    \draw (i3) --  (i4) ;

\end{tikzpicture}
\end{center}

then $\tilde \gamma(\Upsilon)=m\tilde \gamma(\Upsilon_1)\tilde \gamma(\Upsilon_2)\cdots
\tilde \gamma(\Upsilon_{j-1})\tilde \gamma(\Upsilon_j)m.$ 
\end{enumerate}

The function $\tilde \gamma$ is well-defined since the tree is normalized. Indeed, $m$ is the 
minimal letter and we always obtain a word with values greater than $m$ between the two occurrences 
of $m$. Proceeding by induction on the number of internal nodes of $\Upsilon$, we have that the 
words $\gamma(\Upsilon_i)$ are Stirling permutations for each $i$ and so it is $\gamma(\Upsilon)$.

It is not difficult to check that the inverse $\tilde \gamma^{-1}:\hat{\Q}_A \rightarrow \nor_A$ 
can also be defined recursively by

\begin{enumerate}
 \item If $\theta=mm$ then $\tilde \gamma^{-1}(mm)=(\bullet,m)$.
 \item If $\theta=B_{\theta}(m)$ and 
$\mathring{B_{\theta}}(m)=B_{\theta}(a_1)B_{\theta}(a_2)\cdots 
B_{\theta}(a_{j-1})B_{\theta}(a_j)$, then  

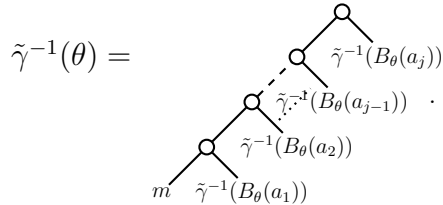
\begin{center}
 
 \begin{tikzpicture}[thick,scale=0.6]
\draw (-2,3) node {$\tilde\gamma^{-1}(\theta)=$};
\tikzstyle{every node}=[draw,scale=0.5]

    \draw [circle,radius=20pt,color=black] (1,1)  node (i1){};
    \draw [circle,color=black] (2,2)  node (i2){};
    \draw [circle,color=black] (4,4)  node (i4){};
    \draw [circle,color=black] (3,3)  node (i3){};

\tikzstyle{every node}=[inner sep=1pt, minimum width=14pt,scale=0.7]
    \draw (0,0)  node (m){$m$};
    \draw (2,0)  node (l1){$\tilde\gamma^{-1}(B_{\theta}(a_1))$};
    \draw (3,1)  node (l2){$\tilde\gamma^{-1}(B_{\theta}(a_2))$};
    \draw (4,2)  node (l3){$\tilde\gamma^{-1}(B_{\theta}(a_{j-1}))$};
    \draw (5,3)  node (l4){$\tilde\gamma^{-1}(B_{\theta}(a_j))$};
   \draw (6,2)  node (dot){\Large .};

    \draw (m) --  (i1) ;
    \draw (i1) --  (l1) ;
    \draw (i2) --  (l2) ;
    \draw (i3) --  (l3) ;
    \draw (i4) --  (l4) ;
    \draw (i1) --  (i2) ;
    \draw [dashed, thick] (i2) --  (i3) ;
    \draw [dotted, thick] (2.6,1.6) --  (3.3,2.3) ;

    \draw (i3) --  (i4) ;

\end{tikzpicture}
\end{center}

\end{enumerate}

The tree defined in the step above is clearly normalized. So we can encode any normalized
binary tree with a permutation in $\hat{\Q}_n$. See Figure \ref{fig:examplebijectionphi} for an 
example of the bijection.

\begin{figure}[ht]
\centering
   \begin{tikzpicture}[thick,scale=0.6]

\begin{scope}[xshift=0cm,yshift=0cm]

\tikzstyle{every node}=[inner sep=2pt,scale=0.8]

    \draw [circle] (0,0)  node (l1){$1$};
    \draw [circle] (0,-2)  node (l2){$2$};
    \draw [circle] (1,-3)  node (l3){$3$};
    \draw [circle] (3,-3)  node (l5){$5$};

 \draw [circle] (3,-1)  node (l6){$6$};
 \draw [circle] (3,1)  node (l4){$4$};
 \draw [circle] (4,2)  node (l7){$7$};

\tikzstyle{every node}=[draw,inner sep=2pt,scale=0.7]

    \draw [circle] (1,1)  node (n1)[pin={[color=red,pin distance=1pt]135:$2$}]{};
    \draw [circle] (2,0)  node (n3)[pin={[color=red,pin distance=1pt]45:$6$}]{};
    \draw [circle] (1,-1)  node (n2)[pin={[color=red,pin distance=1pt]135:$3$}]{};

     \draw [circle] (2,2)  node (n6)[pin={[color=red,pin distance=1pt]135:$4$}]{};
     \draw [circle] (3,3)  node (n4)[pin={[color=red,pin distance=1pt]135:$7$}]{};
   \draw [circle] (2,-2)  node (n5)[pin={[color=red,pin distance=1pt]45:$5$}]{};

     \draw (l1) --  (n1) ;
     \draw (n2) --  (l2) ;
\draw (n2) --  (n3) ;
\draw (n2) --  (n5) ;

\draw (n1) --  (n3) ;
     \draw (n3) --  (l6) ;
     \draw (n4) --  (l7) ;
\draw (n6) --  (l4) ;

     \draw (n6) --  (n1) ;
     \draw (n4) --  (n6) ;
     \draw (n5) --  (l3) ;
     \draw (n5) --  (l5) ;

\end{scope}

\begin{scope}[xshift=12cm,yshift=0cm]
\tikzstyle{every node}=[inner sep=2pt,scale=1]

   \draw [circle] (2,0)  node (i0){$12355366244771$};
\end{scope}

  \draw[color=black,->] (5,0.2) -- node[above] {$\tilde\gamma$}(10,0.2);
  \draw[color=black,->] (10,-0.2) -- node[below] {$\tilde\gamma^{-1}$}(5,-0.2);

\end{tikzpicture}
 
  \caption{Example of the bijection $\tilde \gamma$}
   \label{fig:examplebijectionphi}
   \end{figure}
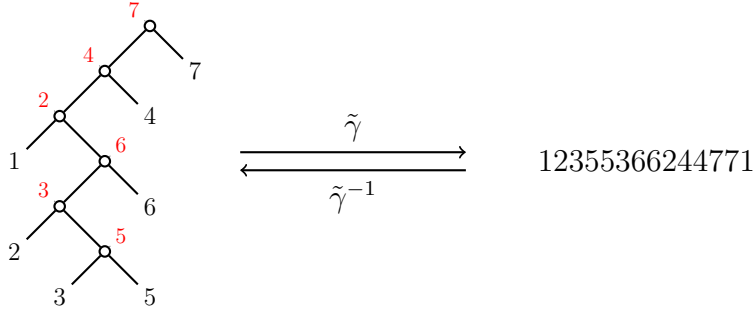

We give an alternative description of $\tilde \gamma$. First we extend the leaf labeling of 
$\Upsilon \in \nor_n$ to a labeling $\theta$ that includes the internal nodes. For each internal 
node $x$, let $\theta(x)$ be the smallest leaf label in the right subtree of the subtree of 
$\Upsilon$ rooted at $x$; for each leaf $x$, let $\theta(x)$ be the leaf label of $x$ (See Figure 
\ref{fig:examplebijectionphi}). Let $x_1,\dots,x_{2n-1}$ be the listing of all the nodes of 
$\Upsilon$ (internal and leaves) in postorder and  let 
$\theta(\Upsilon): = \theta(x_1)\theta(x_2)\dots \theta(x_{2n-1}).$

\begin{proposition}\label{proposition:seconddescriptiongamma}
For all $\Upsilon \in \nor_n$,
 \begin{align*}
  \tilde \gamma (\Upsilon)=\theta(\Upsilon)\theta(x_1)
 \end{align*}
 where $x_1$ is the leftmost leaf of $\Upsilon$.
\end{proposition}
\begin{proof}
If $\Upsilon=(\bullet,m)$ is a single node then  $\tilde \gamma 
(\Upsilon)=mm=\theta(\Upsilon)\theta(x_1)$.
If $\Upsilon$ has internal nodes, it can be expressed as
\begin{align*}
\Upsilon =(\dots (((x_1,v(x_1)) \wedge \Upsilon_1) \wedge \Upsilon_2) \wedge \dots \wedge 
\Upsilon_j)
\end{align*}
(like the one in step (2) of the definition of $\tilde \gamma$).

Let $y_i$ denote the parent of the root of $\Upsilon_i$ for each $i$.  As a consequence 
of the definition of $\theta$, we have that $\theta(y_i) = \theta(z_i)$, where $z_i$ is the 
smallest leaf of $\Upsilon_i$. 
By induction, using the definition of $\tilde \gamma $, 
\begin{align*}
\tilde \gamma(\Upsilon) &= v(x_1) \tilde\gamma(\Upsilon_1) \dots 
\tilde\gamma(\Upsilon_j) v(x_1) \\
&=\theta(x_1) \theta(\Upsilon_1) \theta(y_1) \dots 
\theta(\Upsilon_j)\theta(y_j) \theta(x_1)\\
&=\theta(\Upsilon) \theta(x_1).
\end{align*}
The last step holds since the postorder traversal of $\Upsilon$ lists first $x_1$, followed by 
postorder traversal of $\Upsilon _1$ followed by $y_1$, followed by postorder traversal of 
$\Upsilon_2$ followed by $y_2$, and so on.
\end{proof}

To remove the unnecessary leading and trailing ones in $\tilde \gamma(\Upsilon)$, we consider 
instead the 
map $\gamma:\nor_n \rightarrow \Q_{n-1}$ defined by $\gamma(\Upsilon):=\redu(\tilde 
\gamma(\Upsilon))$ 
for each $\Upsilon \in \nor_n$.

We invite the reader to recall the definition of comb type $\comblambda(\Upsilon)$ of a normalized 
tree $\Upsilon$ given in Section 
\ref{section:introduction} before Theorem \ref{theorem:dimensionscombtype} and the definition of 
Lyndon type $\lyndonlambda(\Upsilon)$ given in Section \ref{section:lyndontype}.
Recall also the definition of ascending adjacent and terminally nested pairs of a Stirling 
permutation $\theta \in \Q_n$, and the associated types $\aalambda(\theta)$ and 
$\tnlambda(\theta)$, given in the first part of this section. 
We give an 
equivalent characterization of these pairs. An ascending adjacent pair of $\theta \in \Q_n$ is a 
pair $(a,b)$ such that $a<b$ and in $\theta$ the second occurrence of $a$ is the immediate 
predecessor of the first occurrence of $b$. A 
terminally nested pair of $\theta \in \Q_n$ is a pair $(a,b)$ such that 
$a<b$ and in $\theta$ the second occurrence of $a$ is 
the immediate successor of the second occurrence of $b$.

For any node (internal or leaf) $x$ of $\Upsilon$ we define the (\emph{reduced 
valency}) $\vr(x):=v(x)-1$.\\
\begin{proposition}\label{proposition:propgamma}The map $\gamma: \nor_n \rightarrow 
\Q_{n-1}$
is a well-defined bijection that satisfies for each $\Upsilon \in \nor_n$ and internal 
node $x$ of  $\Upsilon$,
 \begin{enumerate}
  \item $x$ is a non-Lyndon node of $\Upsilon$ if and only if  $(\vr(R(L(x))),\vr(R(x)))$ 
is an ascending adjacent pair in $\gamma(\Upsilon)$. Moreover, every ascending pair of 
$\gamma(\Upsilon)$ is of the form $(\vr(R(L(x))),\vr(R(x)))$ for some internal node $x$ of 
$\Upsilon$;
    \item $x$ has a right child $R(x)$ that is also an internal node 
if and only if 
 $(\vr(R(x)),\vr(R(R(x))))$ is a terminally nested pair in 
$\gamma(\Upsilon)$. Moreover, every terminally nested pair of $\gamma(\Upsilon)$ is of the form 
$(\vr(R(x)),\vr(R(R(x))))$ for some internal node $x$ of 
$\Upsilon$;
\item $\aalambda(\gamma(\Upsilon))=\lyndonlambda(\Upsilon)$;
\item $\tnlambda(\gamma(\Upsilon))=\comblambda(\Upsilon)$.
 \end{enumerate}
\end{proposition}
\begin{proof}Let $\Upsilon \in \nor_n$ and let $x_i$ be the $i$th node of $\Upsilon$ listed in 
postorder. 
We use the alternative characterization of $\tilde \gamma$ given in Proposition 
\ref{proposition:seconddescriptiongamma}.
We claim that:

\noindent{\bf Claim 1:} The pair $(\theta(x_{i}),\theta(x_{i+1}))$ is an ascending adjacent pair 
of $\tilde \gamma(\Upsilon)$ if 
and only if $x_i$ is a left child that is not a leaf and its parent $\p(x_i)$ satisfies 
$\theta(\p(x_i))>\theta(x_i)$. (The latter condition is equivalent 
to $\p(x_i)$ being a non-Lyndon node.)
 
\noindent{\bf Claim 2:} $(\theta(x_{i+1}),\theta(x_{i}))$ is a terminally nested pair of $\tilde 
\gamma(\Upsilon)$ if and only 
if $x_i$ is a right child that is not a leaf.

We say that $\theta \in Q_n$ has a \emph{first occurrence} of the letter $\theta_i$ at position $i$
if $\theta_j \ne \theta_i$ for all $j<i$. We say that $\theta \in Q_n$ has a \emph{second 
occurrence} of the letter $\theta_i$ at position $i$ if there is a $j <i$ such that $\theta_j=  
\theta_i$.
Before proving these claims we first observe that  
in the word $\tilde \gamma(\Upsilon)=\theta(\Upsilon)\theta(x_1)$ (Proposition 
\ref{proposition:seconddescriptiongamma}), there is a first occurrence of a letter at position $i$ 
if $x_i$ is a leaf and a second occurrence of a letter if $x_i$ is an internal node.
The proof of the two claims follows from the following four cases that in turn are consequences of 
this observation.

\noindent{\bf Case 1:} Let $x_i$ be a left child that is not a leaf. Then 
$x_{i+1}$ is the leftmost leaf of the right subtree of the subtree 
of $\Upsilon$ rooted at $\p(x_i)$. By the observation above, the position $i$ of 
$\tilde\gamma(\Upsilon)$ contains the second occurrence of 
a letter while the position $i+1$ contains the first occurrence of a letter.  
Note that $\theta(\p(x_i))=\theta(x_{i+1})$. 

\noindent{\bf Case 2:} Let $x_i$ be a left child that is a leaf. Then 
$x_{i+1}$ is the smallest leaf of the right subtree of the subtree 
of $\Upsilon$ rooted at $\p(x_i)$.
Hence, positions $i$ and $i+1$ contain first occurrences of letters in 
$\tilde\gamma(\Upsilon)$ and $\theta(x_i) < \theta(x_{i+1})=\theta(\p(x_i))$.

\noindent{\bf Case 3:} Let $x_i$ be a right child that is not a leaf. Then by postorder 
$x_{i+1}=\p(x_i)$ and  positions $i$ and $i+1$ contain second 
occurrences of letters in $\tilde\gamma(\Upsilon)$. 
Note that $\theta(x_i) > \theta(x_{i+1})$.

\noindent{\bf Case 4:} Let $x_i$ be a right child that is a leaf. Then by postorder 
$x_{i+1}=\p(x_i)$ and by the definition of $\theta$ we have that $\theta(x_i)= \theta(\p(x_i)) = 
\theta(x_{i+1})$.

It is not difficult to see that the two claims imply (1) and (2) after applying the definitions of 
$\redu$ and $\vr$.
Parts (3) and (4) are immediate consequences of parts (1) and (2), respectively.
\end{proof}

We have now four different combinatorial interpretations of the  coefficients of the symmetric 
function $L_{n}(\xx):=\sum_{\mu \in \wcomp_{n}}\dim \lie(\mu)\,\xx^{\mu}$ in the elementary 
symmetric function basis. 
Theorem \ref{theorem:dimensionsofliekstirling} below includes Theorem 
\ref{theorem:dimensionscombtype}.

\begin{theorem}\label{theorem:dimensionsofliekstirling}
For all $n$,
 \begin{align*}
\sum_{\mu \in \wcomp_{n-1}}\dim \lie(\mu)\,\xx^{\mu} &= \sum_{\Upsilon \in \nor_n} 
e_{\lyndonlambda(\Upsilon)}(\xx) \\
 &= \sum_{\theta \in \Q_{n-1}} 
e_{\aalambda(\theta)}(\xx)\nonumber\\
&= \sum_{\theta \in \Q_{n-1}} e_{\tnlambda(\theta)}(\xx)\nonumber\\
&= \sum_{\Upsilon \in \nor_n} e_{\comblambda(\Upsilon)}(\xx)\nonumber.
 \end{align*}

\end{theorem}

\begin{proof}The first equality comes from Theorem \ref{theorem:dimensionslyndontype}, the third 
equality is a consequence of Proposition \ref{proposition:propxi}, and the second 
and fourth equality are consequences of Proposition 
\ref{proposition:propgamma}. 
\end{proof}

For a permutation $\theta \in \Q_n$ we define the \emph{initial permutation} $\init(\theta)\in
\sym_n$ to be the subword of $\theta$ formed by the first occurrence of each of the letters in 
$\theta$.  For example, $\init(233772499468861551)=237496815$.

\begin{proposition}\label{proposition:init} For any $\theta \in \Q_n$ and $\Upsilon=(T,\sigma) \in 
\nor_n$, 
\begin{enumerate}
\item $\init(\xi(\theta))=\init(\theta)$
\item $\sigma=\init(\tilde \gamma(\Upsilon))$.
\end{enumerate}
\end{proposition}
\begin{proof}
In the definition of $\xi$, the relative order of the initial occurrence of the letters is not 
changed; which proves (1). We consider the alternative characterization of $\tilde \gamma$ of 
Proposition \ref{proposition:seconddescriptiongamma}. Recall that $\theta(x_i)$ is a first 
occurrence of a letter in $\theta(\Upsilon)$ if and only if $x_i$ is a leaf of $\Upsilon$.  Hence, 
part (2) follows from the fact that postorder of the nodes of $\Upsilon$ restricted to the leaves 
is 
just left to right reading of the leaves.
\end{proof}

We have the following diagram of bijections:

   \begin{tikzpicture}[scale=0.7]

\tikzstyle{every node}=[inner sep=2pt,scale=0.9]
  
    \draw [circle] (0,0)  node (i1){$\nor_n$};
    \draw [circle] (5,0)  node (i2){$\Q_{n-1}$};
    \draw [circle] (10,0)  node (i3){$\Q_{n-1}$};
     \draw [circle] (15,0)  node (i3){$\nor_n$};

 \draw[color=black,->] (1.2,0.1) -- node[above] {$\gamma$ }(3.5,0.1);
 \draw[color=black,->] (3.5,-0.1) -- node[below] {$\gamma^{-1}$ }(1.2,-0.1);
  \draw[color=black,->] (6.5,0.1) -- node[above] {$\xi$ }(8.5,0.1);
 \draw[color=black,->] (8.5,-0.1) -- node[below] {$\xi^{-1}$ }(6.5,-0.1);
 \draw[color=black,->] (11.5,0.1) -- node[above] {$\gamma^{-1}$ }(14,0.1);
 \draw[color=black,->] (14,-0.1) -- node[below] {$\gamma$ }(11.5,-0.1);

\end{tikzpicture}

 The following theorem is a generalization of the classical bijections between Lyndon 
trees, combs and permutations in $\sym_{n-1}$. See Figure \ref{fig:exampleofbijections} for a 
complete example of the bijections.
  
\begin{corollary}\label{corollary:bijectiontypes}
The map $\gamma^{-1}\xi\gamma$ is a bijection on $\nor_n$ that translates between the 
Lyndon type and comb type. Moreover, the bijection preserves the permutation of leaf labels for 
each tree.
\end{corollary}
\begin{proof}
By Propositions \ref{proposition:propxi} and \ref{proposition:propgamma}, 
\begin{align*}
 \comblambda(\gamma^{-1}\xi\gamma(\Upsilon))&=\tnlambda(\xi\gamma(\Upsilon))\\
 &=\aalambda(\gamma(\Upsilon))\\
 &=\lyndonlambda(\Upsilon).
\end{align*}
Proposition \ref{proposition:init} implies that the order of the leaf labels is preserved. 
\end{proof}

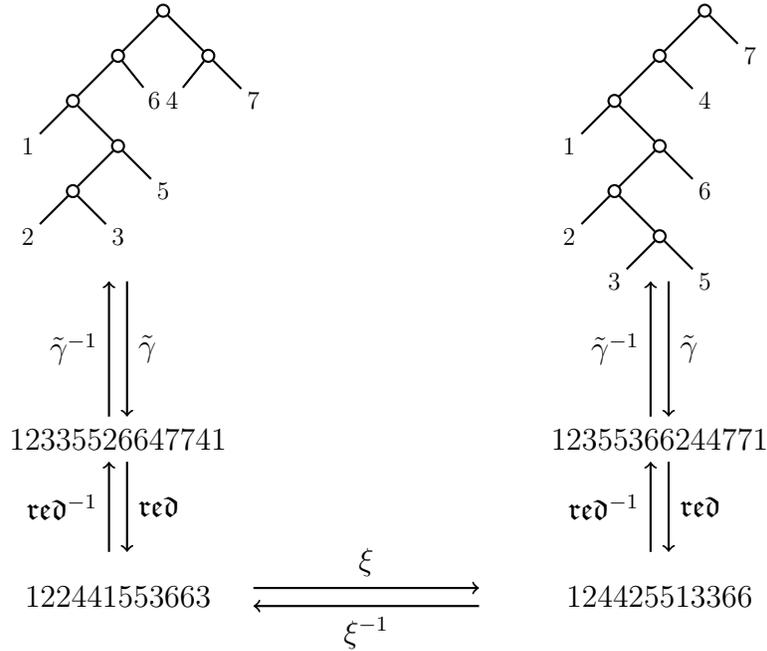
\begin{figure}[ht]
\centering
   \begin{tikzpicture}[thick,scale=0.6]

\begin{scope}[xshift=0cm,yshift=0cm]
\tikzstyle{every node}=[inner sep=2pt,scale=0.8]
  
    \draw [circle] (0,0)  node (l1){$1$};
    \draw [circle] (0,-2)  node (l2){$2$};
    \draw [circle] (2,-2)  node (l3){$3$};
    \draw [circle] (3,-1)  node (l5){$5$};

 \draw [circle] (2.8,1)  node (l6){$6$};
 \draw [circle] (3.2,1)  node (l4){$4$};
 \draw [circle] (5,1)  node (l7){$7$};

\tikzstyle{every node}=[draw,inner sep=2pt,scale=0.8]

    \draw [circle] (1,1)  node (n1){};
    \draw [circle] (2,0)  node (n3){};
     \draw [circle] (1,-1)  node (n2){};
     \draw [circle] (2,2)  node (n6){};
     \draw [circle] (3,3)  node (n4){};
     \draw [circle] (4,2)  node (n7){};

     \draw (l1) --  (n1) ;
     \draw (n1) --  (n3) ;
     \draw (n2) --  (n3) ;
     \draw (n3) --  (l5) ;
     \draw (n2) --  (l2) ;
     \draw (n2) --  (l3) ;

\draw (n6) --  (l6) ;
     \draw (n6) --  (n1) ;
     \draw (n4) --  (n7) ;
     \draw (n4) --  (n6) ;
 \draw (n7) --  (l4) ;
 \draw (n7) --  (l7) ;
\end{scope}

\begin{scope}[xshift=12cm,yshift=0cm]

\tikzstyle{every node}=[inner sep=2pt,scale=0.8]

    \draw [circle] (0,0)  node (l1){$1$};
    \draw [circle] (0,-2)  node (l2){$2$};
    \draw [circle] (1,-3)  node (l3){$3$};
    \draw [circle] (3,-3)  node (l5){$5$};

 \draw [circle] (3,-1)  node (l6){$6$};
 \draw [circle] (3,1)  node (l4){$4$};
 \draw [circle] (4,2)  node (l7){$7$};

\tikzstyle{every node}=[draw,inner sep=2pt,scale=0.8]

    \draw [circle] (1,1)  node (n1){};
    \draw [circle] (2,0)  node (n3){};
    \draw [circle] (1,-1)  node (n2){};

     \draw [circle] (2,2)  node (n6){};
     \draw [circle] (3,3)  node (n4){};
   \draw [circle] (2,-2)  node (n5){};

     \draw (l1) --  (n1) ;
     \draw (n2) --  (l2) ;
\draw (n2) --  (n3) ;
\draw (n2) --  (n5) ;

\draw (n1) --  (n3) ;
     \draw (n3) --  (l6) ;
     \draw (n4) --  (l7) ;
\draw (n6) --  (l4) ;

     \draw (n6) --  (n1) ;
     \draw (n4) --  (n6) ;
     \draw (n5) --  (l3) ;
     \draw (n5) --  (l5) ;

\end{scope}

\begin{scope}[xshift=12cm,yshift=-6.5cm]
\tikzstyle{every node}=[inner sep=2pt,scale=1]

   \draw [circle] (2,0)  node (i0){$12355366244771$};
\end{scope}

\begin{scope}[xshift=0cm,yshift=-6.5cm]

\tikzstyle{every node}=[inner sep=2pt,scale=1]

   \draw [circle] (2,0)  node (i0){$12335526647741$};

\end{scope}

\begin{scope}[xshift=12cm,yshift=-10cm]
\tikzstyle{every node}=[inner sep=2pt,scale=1]

   \draw [circle] (2,0)  node (i0){$124425513366$};
\end{scope}

\begin{scope}[xshift=0cm,yshift=-10cm]

\tikzstyle{every node}=[inner sep=2pt,scale=1]

   \draw [circle] (2,0)  node (i0){$122441553663$};

\end{scope}

 \draw[color=black,->] (2.2,-3) -- node[right] {$\tilde \gamma$}(2.2,-6);
 \draw[color=black,->] (1.8,-6) -- node[left] {$\tilde \gamma^{-1}$}(1.8,-3);

\draw[color=black,->] (14.2,-3) -- node[right] {$\tilde \gamma$}(14.2,-6);
 \draw[color=black,->] (13.8,-6) -- node[left] {$\tilde \gamma^{-1}$}(13.8,-3);

 \draw[color=black,->] (2.2,-7) -- node[right] {$\redu$}(2.2,-9);
 \draw[color=black,->] (1.8,-9) -- node[left] {$\redu^{-1}$}(1.8,-7);

\draw[color=black,->] (14.2,-7) -- node[right] {$\redu$}(14.2,-9);
 \draw[color=black,->] (13.8,-9) -- node[left] {$\redu^{-1}$}(13.8,-7);

  \draw[color=black,->] (5,-9.8) -- node[above] {$\xi$}(10,-9.8);
  \draw[color=black,->] (10,-10.2) -- node[below] {$\xi^{-1}$}(5,-10.2);

\end{tikzpicture}
 
 \caption{Example of the bijections $\tilde \gamma$, $\redu$ and $\xi$}
 \label{fig:exampleofbijections}
 \end{figure}

We can combine Theorem \ref{theorem:compositionalinverse} with Theorem 
\ref{theorem:dimensionsofliekstirling} and conclude the following 
$e$-positivity result.

\begin{theorem}\label{theorem:compositionalinversefinal}
We have
  \begin{align*}
 \left [ \sum_{n\ge1}(-1)^{n-1} 
h_{n-1}(\xx)\frac{y^n}{n!} \right ] ^{<-1>} &= \sum_{n\ge1}\sum_{\Upsilon \in \nor_n} 
e_{\lambda(\Upsilon)}(\xx)\frac{y^n}{n!}\\
&=  \sum_{n\ge1}\sum_{\theta \in \Q_{n-1}} 
e_{\lambda(\theta)}(\xx)\frac{y^n}{n!}\nonumber,
  \end{align*}
  where $\lambda(\Upsilon)$  is either the Lyndon type or the comb type of the normalized tree 
$\Upsilon$ and $\lambda(\theta)$ is either the AA type or the TN type of the Stirling permutation 
$\theta$.
 \end{theorem}
 
 In \cite{Dleon2013b} the author gives another proof of Theorem 
\ref{theorem:compositionalinversefinal} that 
does not involve poset topology and instead involves a nice interpretation of the compositional 
inverse of exponential generating functions given by B. Drake in \cite{Drake2008}.

\subsection{A remark about colored Stirling permutations}
We can also define colored Stirling permutations in analogy with the case of colored 
normalized binary trees. An \emph{AA colored Stirling permutation} $\Theta=(\theta, c)$ is a 
Stirling permutation $\theta \in \Q_n$ together with a map $c:[n] \rightarrow \PP$ such that for 
every occurrence of an ascending adjacent pair $(a,b)$ in $\theta$, $c$ satisfies the condition 
$c(a)>c(b)$. 
\begin{example}
 If $\theta=233772499468861551$, the map $c:[9]\rightarrow \PP$ defined by the pairs
 $(i,c(i))$:
\[\{(1,{\color{blue}1}),(2,{\color{brown}3}),(3,{\color{brown}3}),(4,{\color{red}2}),(5,{\color{red}
2 } ) ,(6,{\color{ blue}1}),(7, {\color{blue}1}),(8, {\color{red}2}),(9,{\color{blue}1})\}\]
is an AA coloring, but
\[\{(1,{\color{blue}1}),(2,{\color{red}2}),(3,{\color{brown}3}),(4,{\color{brown}3}),(5,{\color{red}
2 } ) ,(6,{\color{ blue}1}),(7, {\color{blue}1}),(8, {\color{red}2}),(9,{\color{blue}1})\}\]
is not since $24$ is an adjacent ascending pair but
$c(2)={\color{red}2}<{\color{brown}3}=c(4)$.  
\end{example}
In the same manner we define a \emph{TN colored Stirling permutation} to the pair $\Theta=(\theta, 
c)$, where $c$ satisfies $c(a)>c(b)$ whenever $(a,b)$ is a terminally nested pair. For $\mu \in 
\wcomp_n$, we say that a colored Stirling permutation $(\theta,c)$ is $\mu$-colored if $\mu(i) = 
|c^{-1}(i)|$ for all $i$.
We denote by 
$\Q^{AA}_{\mu}$ the set of AA $\mu$-colored Stirling permutations of $[n]$ and 
$\Q^{TN}_{\mu}$ the set of TN $\mu$-colored Stirling permutations of $[n]$.

\begin{corollary}[of Corollary \ref{corollary:muintervalslyn} ]\label{corollary:muintervalsperm} 
For 
all $n \ge 1$ and $\mu \in \wcomp_{n-1}$
\[ \bar \mu_{\Pi_{n}^{k}}(\hat{0},[n]^{\mu})=(-1)^{n-1}|\Q^{AA}_{\mu}|=(-1)^{n-1}|\Q^{TN}_{\mu}|.\]
Consequently,
\[
 \dim \widetilde H^{n-3}((\hat 0, [n]^{\mu}))=|\Q^{AA}_{\mu}|=|\Q^{TN}_{\mu}|.
\]
\end{corollary}
\begin{proof}
Note that the bijection $\gamma:\nor_n \rightarrow \Q_{n-1}$ extends naturally to a bijection 
$\lyn_{\mu} \cong \Q^{AA}_{\mu}$ and the bijection $\xi:\Q_{n-1}\rightarrow \Q_{n-1}$ extends 
naturally to a bijection  $ \Q^{AA}_{\mu}\cong \Q^{TN}_{\mu}$. Thus the result is a corollary 
of Corollary \ref{corollary:muintervalslyn}.
\end{proof}

By Theorem \ref{theorem:liehomisomorphism},
\begin{corollary}
For 
all $n \ge 1$ and $\mu \in \wcomp_{n-1}$

\[
 \dim \lie(\mu)=|\Q^{AA}_{\mu}|=|\Q^{TN}_{\mu}|.
\]
\end{corollary}

\section{Combinatorial bases}\label{section:combinatorialbases}
In this section we discuss  various  bases for $\lie(\mu)$,
$\widetilde H^{n-3}((\hat 0,[n]^{\mu}))$ and $\widetilde H^{n-2}(\Pi_n^k\setminus \{\hat{0}\})$.

\subsection{Colored Lyndon basis}\label{section:lyndonbasis}
Recall from Theorem \ref{theorem:elth} that the ascent-free maximal chains of the EL-labeling of  
$[\hat{0},[n]^{\mu}]$ yield a basis for the cohomology $\widetilde H^{n-3}((\hat{0},[n]^{\mu}))$.
Hence, Theorem \ref{thm:ascfreeEL} gives a description of
this basis in terms of colored Lyndon trees. 
The following result gives a basis closely related (equal up to signs).
By applying 
the isomorphism of
Theorem~\ref{theorem:liehomisomorphism}, one gets a corresponding basis for $\lie(\mu)$, which 
reduces to the classical Lyndon basis for $\lie(n)$ when $\mu$ has a single nonzero component. 

\begin{theorem}\label{theorem:lyndonbasis}
The set $\{\bar{c}(T,\sigma)\mid\, (T,\sigma) \in \lyn_{\mu}\}$ is a basis of
$\widetilde H^{n-3}((\hat{0},[n]^{\mu}))$ and the set
$\{[T,\sigma] \mid \, (T,\sigma) \in \lyn_{\mu}\}$ is a basis for $ \lie(\mu)$.
\end{theorem}
\begin{proof}
By Theorem  \ref{thm:ascfreeEL} and Theorem  \ref{theorem:elth}, the set  
$\{\bar{c}(T,\sigma,\tau_{T,\sigma})\mid\, (T,\sigma) \in \lyn_{\mu}\}$ is a basis of
$\widetilde H^{n-3}((\hat{0},[n]^{\mu}))$. Lemma \ref{lemma:52} 
implies that we can replace $\tau_{T,\sigma}$ by any other linear 
extension and still obtain a basis for $\widetilde H^{n-3}((\hat{0},[n]^{\mu}))$. In particular, we 
can 
replace 
it by postorder.
\end{proof}

Theorem \ref{theorem:lyndonbasis} already implies that the set of maximal chains coming 
from colored
Lyndon trees spans $\widetilde H^{n-3}((\hat 0, [n]^{\mu}))$; however, in order to complete the 
proof 
of
Theorem \ref{theorem:binarybasishomology}, and conclude that the relations in the theorem
generate all the cohomology relations, we will show that we can represent any  
$\bar{c}
\in \widetilde H^{n-3}((\hat 0, [n]^{\mu}))$ as a linear combination of chains in
$\{\bar{c}(T,\sigma)\mid\, (T,\sigma) \in \lyn_{\mu}\}$ using only the relations in Proposition
\ref{theorem:binarybasishomology}.

\begin{proposition}\label{proposition:lyndononlyrelations}
 The relations in Proposition \ref{theorem:binarybasishomology} generate all the cohomology
relations in $\widetilde H^{n-3}((\hat 0, [n]^{\mu}))$.
\begin{proof}
We use a ``straightening'' strategy using the relations of Theorem 
\ref{theorem:binarybasishomology} in order to prove the result. Recall that for an internal node 
$x$ of a normalized binary tree $\Upsilon \in \nor_n$, we define the valency $v(x)$ to be the 
smallest of the labels in the subtree of $\Upsilon$ rooted at $x$. We define a \emph{valency 
inversion} in $\Upsilon \in \nor_n$ to be a pair of internal 
nodes $(x,y)$ such that:
\begin{itemize}
 \item $x$ is in the subtree rooted at the left child of $y$,
 \item $v(R(x))<v(R(y))$.
\end{itemize}
Let $\valinv(\Upsilon)$ denote the number of valency inversions in $\Upsilon$. 
Note for example that a Lyndon tree is a normalized binary tree such that $\valinv(\Upsilon)=0$.

A \emph{coloring inversion} is a pair of internal nodes $(x,y)$ in $\Upsilon$ such that
\begin{itemize}
\item $v(x)=v(y)$,
\item $x$ is in the subtree rooted at the left child of $y$,
\item $\clr(x)<\clr(y)$.
\end{itemize}
We denote by $\colinv(\Upsilon)$, the number of coloring inversions in $\Upsilon$.

Define the \emph{inversion pair} of $\Upsilon$ to be $(\valinv(\Upsilon),\colinv(\Upsilon))$.  We 
order these pairs lexicographically; that is, we say  
\begin{align*}
(\valinv(\Upsilon),\colinv(\Upsilon))< 
(\valinv(\Upsilon^\prime),\colinv(\Upsilon^\prime)),
\end{align*} 
if either $\valinv(\Upsilon) < 
\valinv(\Upsilon^\prime)$ or $\valinv(\Upsilon)=\valinv(\Upsilon^\prime)$ and $\colinv(\Upsilon) < 
\colinv(\Upsilon^\prime)$. Note that
if the inversion pair of $\Upsilon$ is $(0,0)$ then $\Upsilon$ is a colored Lyndon tree since
in particular its underlying uncolored tree is a Lyndon tree.

Now let $\Upsilon \in \BT_{\mu}$ be a colored normalized  binary tree that is not a colored 
Lyndon tree. Then $\Upsilon$ must have a subtree of the form: 
$(\Upsilon_1 \substack{i \\ \wedge} \Upsilon_2 )\substack{j \\ \wedge} \Upsilon_3$, with 
$v(\Upsilon_2)<v(\Upsilon_3)$ and $i \le j$. We will show  that $\bar c(\Upsilon)$ can be expressed 
as a linear combination of chains associated with colored normalized binary trees with smaller
inversion pair.

\noindent{\bf Case $i=j$:} 
Using  relation (\ref{relation:3h}) (and  relation(\ref{relation:1h})) we have that
\begin{align*}
\bar c(\alpha((\Upsilon_1 \substack{i \\ \wedge} \Upsilon_2) \substack{i \\ \wedge}
\Upsilon_3)\beta)
=\pm\bar c( \alpha(\Upsilon_1 \substack{i \\ \wedge} (\Upsilon_2 \substack{i \\ 
\wedge}\Upsilon_3))\beta )
\pm \bar c(\alpha((\Upsilon_1 \substack{i \\ \wedge} \Upsilon_3)\substack{i \\ \wedge}
\Upsilon_2)\beta).
\end{align*}
(The  signs in the relations of Theorem~\ref{theorem:binarybasishomology} are not relevant 
here and have therefore been suppressed.)

Let $\p(\Upsilon_j)$ denote the parent of the root of the subtree $\Upsilon_j$ in $\Upsilon$. We 
then have that
\begin{align*}
 \valinv(\alpha((\Upsilon_1 \substack{i \\ \wedge} \Upsilon_2) \substack{i \\ \wedge}
\Upsilon_3)\beta)
-\valinv( \alpha(\Upsilon_1 \substack{i \\ \wedge} (\Upsilon_2 \substack{i \\ 
\wedge}\Upsilon_3))\beta)&\ge 1,
\end{align*}
since the pair $(\p(\Upsilon_2),\p(\Upsilon_3))$ and any other valency inversion between an 
internal node of $\Upsilon_1$ and $\p(\Upsilon_3)$ are valency inversions in the former tree but 
not in the later and no other change occurs to the set of valency inversions. We also have that
\begin{align*}
\valinv(\alpha((\Upsilon_1 \substack{i \\ \wedge} \Upsilon_2) \substack{i \\ \wedge}
\Upsilon_3)\beta)
-\valinv(\alpha((\Upsilon_1 \substack{i \\ \wedge} \Upsilon_3)\substack{i \\ \wedge}
\Upsilon_2)\beta)& \ge 1,
\end{align*}
since the pair $(\p(\Upsilon_2),\p(\Upsilon_3))$ and any other valency inversion between an 
internal node of $\Upsilon_2$ and $\p(\Upsilon_3)$ are valency inversions in the former tree but 
not in the later and no other change occurs to the set of valency inversions.

\noindent{\bf Case $i < j$:}
Using  relation (\ref{relation:5h}) (and  relation (\ref{relation:1h})) we have that

\begin{align*}\bar c(\alpha((\Upsilon_1 \substack{i \\ \wedge} \Upsilon_2)\substack{j \\ 
\wedge}
\Upsilon_3)\beta)=&
\pm \bar c(\alpha((\Upsilon_1 \substack{j \\ \wedge} \Upsilon_2 )\substack{i \\ \wedge}
\Upsilon_3)\beta) \\ & 
\pm \bar c(\alpha(\Upsilon_1 \substack{j \\ \wedge} (\Upsilon_2\substack{i \\ \wedge}
\Upsilon_3))\beta) \\ & 
\pm  \bar c(\alpha( \Upsilon_1 \substack{i \\ \wedge} (\Upsilon_2 \substack{j\\
\wedge} \Upsilon_3) )\beta) \\ & 
\pm \bar c(\alpha((\Upsilon_1 \substack{i \\ \wedge} \Upsilon_3 )\substack{j \\ \wedge}
\Upsilon_2)\beta) \\ & 
\pm \bar c(\alpha((\Upsilon_1 \substack{j \\ \wedge} \Upsilon_3 )\substack{i \\ \wedge}
\Upsilon_2)\beta) .
\end{align*}

Just as in the previous case, all the labeled colored trees on the right hand side of the 
equation, except for the first, have a smaller number of valency inversions than that of 
the tree in the left hand side.  The first labeled colored tree $\bar c(\alpha((\Upsilon_1 
\substack{j \\ \wedge} \Upsilon_2 )\substack{i \\ \wedge}
\Upsilon_3)\beta)$ has the same number of valency inversions as that of $c(\alpha((\Upsilon_1 
\substack{i \\ \wedge} \Upsilon_2)\substack{j \\ 
\wedge} \Upsilon_3)\beta)$.  However the coloring inversion number is reduced by one and so the 
inversion pair is reduced.

From the two cases above we conclude that  if $\Upsilon \in \BT_{\mu}$ is a colored normalized  
binary tree then $\bar c(\Upsilon)$ can be expressed as a linear 
combination of chains, associated to colored normalized  binary 
trees, of smaller inversion pair. Hence by induction on the inversion 
pair, $\bar 
c(\Upsilon)$ can be expressed as a linear combination of chains of the form $\bar c 
(\Upsilon^\prime)$ where $\Upsilon^\prime \in \lyn_{\mu}$.
Also since by relation (\ref{relation:1h}) any $\Upsilon \in \BT_{\mu}$ is of the form $\pm \bar 
c(\Upsilon^\prime)$, where $\Upsilon^\prime$ is a colored normalized binary tree, the same is true 
for any $\Upsilon \in \BT_{\mu}$.

Since the set $\{\bar c(\Upsilon) \mid \Upsilon \in \BT_{\mu}\}$ is a spanning set for $\widetilde 
H^{n-3}((\hat 0, [n]^{\mu}))$,  we have shown using only the relations in  Theorem 
\ref{theorem:binarybasishomology} that  $\{\bar 
c(\Upsilon) \mid \Upsilon \in \lyn_{\mu}\}$ is also a spanning set for $\widetilde H^{n-3}((\hat 
0, 
[n]^{\mu}))$. The fact that $\{\bar 
c(\Upsilon) \mid \Upsilon \in \lyn_{\mu}\}$ is a basis (Theorem \ref{theorem:lyndonbasis}), proves 
the result.
\end{proof}

\end{proposition}

\subsection{Colored comb basis}\label{section:coloredcombs}

A \emph{colored comb} is a  normalized colored binary tree that satisfies the following
coloring restriction:  for each internal node $x$ whose right child $R(x)$ is not a leaf, 
\begin{align}\label{equation:combcondition}
 \clr(x)>\clr(R(x)).
\end{align}

Let $\comb_n$ be the set of colored combs in $\BT_n$ and $\comb_{\mu}$
be the set of the $\mu$-colored ones. Note that in a monochromatic comb every right child has to be 
a 
leaf and hence they are the classical left combs that yield a basis for $\lie(n)$ (see 
\cite[Proposition~2.3]{Wachs1998}). The colored combs generalize also the bicolored combs that 
yield a basis of $\lie(n,i)$ in \cite{DleonWachs2013a}.
Figure \ref{fig:combsn3} illustrates the bicolored combs for $n=3$.

\begin{figure}[ht]
  \begin{tikzpicture}[thick,scale=0.6]
 \draw [circle,color=red] (0.5,2.5)  node (red){\small 2};
    \draw [color=blue] (0.5,1.7)  node (blue){ \small 1};
\tikzstyle{every node}=[fill, draw,inner sep=4pt, minimum width=1pt,scale=0.5]
    \draw [circle,color=red] (0,2.5)  node (r){};
    \draw [color=blue] (0,1.7)  node (b){};

\begin{scope}[xshift=0,yshift=-1cm]

\tikzstyle{every node}=[fill, draw,inner sep=2pt,scale=0.8]
    \draw [color=blue] (1,1)  node (i1){};
    \draw [color=blue] (2,2)  node (i2){};

\tikzstyle{every node}=[inner sep=1pt, minimum width=14pt,scale=0.7]

    \draw (0,0)  node (m){$1$};
    \draw (2,0)  node (l1){$2$};
    \draw (3,1)  node (l2){$3$};

    \draw (m) --  (i1) ;
    \draw (i1) --  (l1) ;
    \draw (i1) --  (i2) ;
    \draw (i2) --  (l2) ;
\end{scope}

\begin{scope}[xshift=3.5cm,yshift=0]
\tikzstyle{every node}=[fill, draw,inner sep=2pt,scale=0.8]
    \draw [circle,color=red] (1,1)  node (i1){};
    \draw [color=blue] (2,2)  node (i2){};

\tikzstyle{every node}=[inner sep=1pt, minimum width=14pt,scale=0.7]

    \draw (0,0)  node (m){$1$};
    \draw (2,0)  node (l1){$2$};
    \draw (3,1)  node (l2){$3$};

    \draw (m) --  (i1) ;
    \draw (i1) --  (l1) ;
    \draw (i1) --  (i2) ;
    \draw (i2) --  (l2) ;
\end{scope}

\begin{scope}[xshift=7cm,yshift=0]
\tikzstyle{every node}=[fill, draw,inner sep=2pt,scale=0.8]
    \draw [color=blue] (1,1)  node (i1){};
    \draw [circle,color=red] (2,2)  node (i2){};

\tikzstyle{every node}=[inner sep=1pt, minimum width=14pt,scale=0.7]

    \draw (0,0)  node (m){$1$};
    \draw (2,0)  node (l1){$2$};
    \draw (3,1)  node (l2){$3$};

    \draw (m) --  (i1) ;
    \draw (i1) --  (l1) ;
    \draw (i1) --  (i2) ;
    \draw (i2) --  (l2) ;
\end{scope}
\begin{scope}[xshift=10.5cm,yshift=-1cm]
\tikzstyle{every node}=[fill, draw,inner sep=2pt,scale=0.8]
    \draw [circle,color=red] (1,1)  node (i1){};
    \draw [circle,color=red] (2,2)  node (i2){};

\tikzstyle{every node}=[inner sep=1pt, minimum width=14pt,scale=0.7]

    \draw (0,0)  node (m){$1$};
    \draw (2,0)  node (l1){$2$};
    \draw (3,1)  node (l2){$3$};

    \draw (m) --  (i1) ;
    \draw (i1) --  (l1) ;
    \draw (i1) --  (i2) ;
    \draw (i2) --  (l2) ;
\end{scope}

\begin{scope}[xshift=0,yshift=-5cm]

\tikzstyle{every node}=[fill, draw,inner sep=2pt,scale=0.8]
    \draw [color=blue] (1,1)  node (i1){};
    \draw [color=blue] (2,2)  node (i2){};

\tikzstyle{every node}=[inner sep=1pt, minimum width=14pt,scale=0.7]

    \draw (0,0)  node (m){$1$};
    \draw (2,0)  node (l1){$3$};
    \draw (3,1)  node (l2){$2$};

    \draw (m) --  (i1) ;
    \draw (i1) --  (l1) ;
    \draw (i1) --  (i2) ;
    \draw (i2) --  (l2) ;
\end{scope}

\begin{scope}[xshift=3.5cm,yshift=-6cm]
\tikzstyle{every node}=[fill, draw,inner sep=2pt,scale=0.8]
    \draw [circle,color=red] (1,1)  node (i1){};
    \draw [color=blue] (2,2)  node (i2){};

\tikzstyle{every node}=[inner sep=1pt, minimum width=14pt,scale=0.7]

    \draw (0,0)  node (m){$1$};
    \draw (2,0)  node (l1){$3$};
    \draw (3,1)  node (l2){$2$};

    \draw (m) --  (i1) ;
    \draw (i1) --  (l1) ;
    \draw (i1) --  (i2) ;
    \draw (i2) --  (l2) ;
\end{scope}

\begin{scope}[xshift=7cm,yshift=-6cm]
\tikzstyle{every node}=[fill, draw,inner sep=2pt,scale=0.8]
    \draw [color=blue] (1,1)  node (i1){};
    \draw [circle,color=red] (2,2)  node (i2){};

\tikzstyle{every node}=[inner sep=1pt, minimum width=14pt,scale=0.7]

    \draw (0,0)  node (m){$1$};
    \draw (2,0)  node (l1){$3$};
    \draw (3,1)  node (l2){$2$};

    \draw (m) --  (i1) ;
    \draw (i1) --  (l1) ;
    \draw (i1) --  (i2) ;
    \draw (i2) --  (l2) ;
\end{scope}
\begin{scope}[xshift=10.5cm,yshift=-5cm]
\tikzstyle{every node}=[fill, draw,inner sep=2pt,scale=0.8]
    \draw [circle,color=red] (1,1)  node (i1){};
    \draw [circle,color=red] (2,2)  node (i2){};

\tikzstyle{every node}=[inner sep=1pt, minimum width=14pt,scale=0.7]

    \draw (0,0)  node (m){$1$};
    \draw (2,0)  node (l1){$3$};
    \draw (3,1)  node (l2){$2$};
    
    \draw (m) --  (i1) ;
    \draw (i1) --  (l1) ;
    \draw (i1) --  (i2) ;
    \draw (i2) --  (l2) ;
\end{scope}

\begin{scope}[xshift=4cm,yshift=-3cm]
\tikzstyle{every node}=[fill, draw,inner sep=2pt,scale=0.8]
    \draw [color=blue] (3,1)  node (i1){};
    \draw [circle,color=red] (2,2)  node (i2){};

\tikzstyle{every node}=[inner sep=1pt, minimum width=14pt,scale=0.7]

    \draw (2,0)  node (m){$2$};
    \draw (4,0)  node (l1){$3$};
    \draw (1,1)  node (l2){$1$};
    
    \draw (m) --  (i1) ;
    \draw (i1) --  (l1) ;
    \draw (i1) --  (i2) ;
    \draw (i2) --  (l2) ;
\end{scope}
\end{tikzpicture}
 \caption{Set of bicolored combs for $n=3$}
 \label{fig:combsn3}
\end{figure}
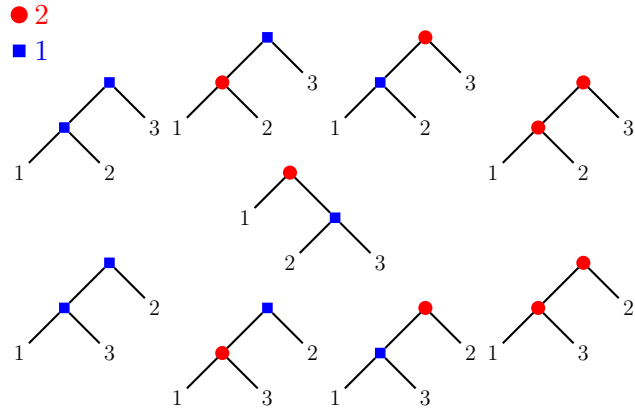	

\begin{remark}
Note that the coloring condition (\ref{equation:combcondition}) is closely related to the 
comb type of a normalized tree defined in Section \ref{section:introduction} before Theorem 
\ref{theorem:dimensionscombtype}. The coloring condition 
implies that in a colored comb $\Upsilon$ there are no repeated colors in each block $B$ of 
the partition $\pi^{\comb}(\Upsilon)$ associated to $\Upsilon$.
So after choosing 
$|B|$ 
different colors for the 
internal nodes of $\Upsilon$ in $B$,  there is a unique way to assign the colors such 
that $\Upsilon$ is a colored comb (the colors must decrease towards the 
right in each block of $\pi^{\comb}(\Upsilon)$). 
In Figure \ref{fig:combtype} this relation is illustrated.
\end{remark}

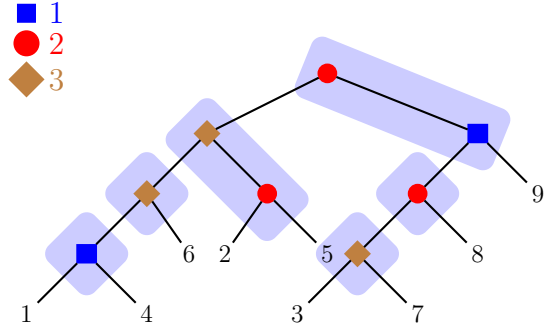
\begin{figure}[ht]
        \centering
        \usetikzlibrary{shapes,snakes}

\begin{tikzpicture}[thick,scale=0.8]

\begin{scope}[xshift=0cm,yshift=1cm]
 \draw [color=blue] (1.5,4)  node (blue){$1$};
\draw [circle,color=red] (1.5,3.5)  node (red){$2$};
 \draw [color=brown] (1.5,2.9)  node (brown){$3$};
\tikzstyle{every node}=[fill, draw,inner sep=4pt, minimum width=1pt,scale=0.8]
    
   \draw [color=blue] (1,4)  node (b){};
\draw [circle,color=red] (1,3.5)  node (r){};
    \draw [diamond,color=brown] (1,2.9)  node (g){};
\end{scope}

\node[fill=blue!20,blue!20,draw,rectangle,rounded corners,rotate=-45, minimum width=70pt,minimum 
height=30pt,scale=0.8] at  (4.5,2.4) {};
\node[fill=blue!20,blue!20,draw,rectangle,rounded corners,rotate=-22, minimum width=100pt,minimum 
height=30pt,scale=0.8] at  (7.25,3.5) {};
height=30pt,scale=0.8] at  (7.5,2) {};
\node[fill=blue!20,blue!20,draw,rectangle,rounded corners,rotate=45, minimum width=30pt,minimum 
height=30pt,scale=0.8] at  (2,1) {};
\node[fill=blue!20,blue!20,draw,rectangle,rounded corners,rotate=45, minimum width=30pt,minimum 
height=30pt,scale=0.8] at  (3,2) {};
\node[fill=blue!20,blue!20,draw,rectangle,rounded corners,rotate=45, minimum width=30pt,minimum 
height=30pt,scale=0.8] at  (6.5,1) {};
\node[fill=blue!20,blue!20,draw,rectangle,rounded corners,rotate=45, minimum width=30pt,minimum 
height=30pt,scale=0.8] at  (7.5,2) {};
\tikzstyle{every node}=[fill,draw,inner sep=0pt, minimum width=1 pt, scale=0.8]

    \draw [circle,color=red] (6,4)  node (i1){n};
    \draw [color=blue]  (8.5,3)  node (i2){N};
    \draw [circle,color=red]  (7.5,2)  node (i3){n};
    \draw [diamond,color=brown]  (6.5,1)  node (i4){n};
    \draw[diamond, color=brown]  (4,3)  node (i5){n};
    \draw [circle,color=red]  (5,2)  node (i6){n};
    \draw [diamond,color=brown] (3,2)  node (i7){n};
    \draw[color=blue]  (2,1)  node (i8){N};
\tikzstyle{every node}=[inner sep=1pt, minimum width=14pt,scale=0.8]

    \draw (4.3,1)  node (l1){2};
    \draw (5.5,0)  node (l2){3};
    \draw (1,0)  node (l3){1};
    \draw (3,0)  node (l4){4};
    \draw (6,1)  node (l5){5};
    \draw (3.7,1)  node (l6){6};
    \draw (7.5,0)  node (l7){7};
    \draw (9.5,2)  node (l8){9};
    \draw (8.5,1)  node (l9){8};

    \draw (i1) --  (i2) ;
    \draw (i1) --  (i5) ;
    \draw (i2) --  (i3) ;
    \draw (i2) --  (l8) ;
    \draw (i3) --  (i4) ;
    \draw (i3) --  (l9) ;
    \draw (i4) --  (l7) ;
    \draw (i4) --  (l2) ;
    \draw (i5) --  (i6) ;
    
    \draw (i5) --  (i7) ;
    \draw (i6) --  (l5) ;
    \draw (i6) --  (l1) ;
    \draw (i7) --  (l6) ;
    \draw (i7) --  (i8) ;
    \draw (i8) --  (l3) ;
    \draw (i8) --  (l4) ;
\end{tikzpicture}
 \caption{Example of a colored comb of comb type $(2,2,1,1,1,1)$}
\label{fig:combtype}
  \end{figure}
\begin{theorem}\label{theorem:combbasiscardinality}
There is a bijection
\[
\lyn_{\mu} \cong \comb_{\mu}.
\]
\end{theorem}

\begin{proof}
This is a consequence of Theorem \ref{corollary:bijectiontypes}. Indeed, the 
bijection $\gamma^{-1}\xi\gamma$ that translates between the 
Lyndon type and comb type on $\nor_n$ extends naturally to a bijection $\lyn_{\mu} 
\cong \comb_{\mu}$.
\end{proof}

We obtain the following corollary from Corollary \ref{corollary:muintervalslyn}.   
\begin{corollary}\label{corollary:muintervalscomb} For all $n \ge 1$ and $\mu \in \wcomp_{n-1}$,
\[ \bar \mu_{\Pi_{n}^{k}}(\hat{0},[n]^{\mu})=(-1)^{n-1}|\comb_{\mu}|.\]
Hence,
\[
 \dim \widetilde H^{n-3}((\hat 0, [n]^{\mu}))=|\comb_{\mu}|.
\]
\end{corollary}

By Theorem \ref{theorem:dimliemu} or Theorem \ref{theorem:liehomisomorphism} we have,
\begin{corollary}
For 
all $n \ge 1$ and $\mu \in \wcomp_{n-1}$

\[
 \dim \lie(\mu)=|\comb_{\mu}|.
\]
\end{corollary}

\begin{theorem}\label{theorem:combsarebasis}
$\{\bar c(T,\sigma)\mid (T,\sigma) \in \comb_{\mu} \}$ is a basis for $\widetilde H^{n-3}((\hat 0,
[n]^{\mu}))$ and $\{[T,\sigma]\mid (T,\sigma) \in \comb_{\mu} \}$ is a basis for $\lie(\mu)$.
\end{theorem}
\begin{proof}
The ``spanning" part of the proof for $\widetilde H^{n-3}((\hat 0,
[n]^{\mu}))$ follows the same idea as in the proof of Proposition 
\ref{proposition:lyndononlyrelations}. Instead of using $\valinv(T)$ we use $\sum_{x \in I(T)} 
r(x)$, 
where $I(T)$ is the set of internal nodes  of  $T$ and $r(x)$ is the number of internal nodes in 
the 
right subtree of $x$. And instead of using $\colinv(T)$ we use  the number 
of pairs $(x,y)$ of internal nodes of $T$ such that  $y$ is a descendant of $x$ that can be reached 
from $x$ along a path of right edges and $\clr(x)<\clr(y)$ (see \cite[Theorem 5.1]{DleonWachs2013a} 
for the special case in which $\supp(\mu) \subseteq [2]$).

Using Corollary 
\ref{corollary:muintervalscomb} we conclude that the set $\{\bar c(T,\sigma)\mid (T,\sigma) \in 
\comb_{\mu} \}$ is a basis for $\widetilde H^{n-3}((\hat 0,
[n]^{\mu}))$ and using Theorem \ref{theorem:liehomisomorphism} that $\{[T,\sigma]\mid 
(T,\sigma) \in \comb_{\mu} \}$ is a basis for $\lie(\mu)$.
\end{proof}

\subsection{Bases for cohomology of the full weighted partition poset}

In \cite{DleonWachs2013a} bicolored combs and bicolored Lyndon trees are used to 
construct bases for $\widetilde H^{n-2}(\Pi_n^w \setminus \{\hat 0\})$.

Denote the root of a colored binary tree $T$ by $\root(T)$, and 
define
\begin{align*}
\BT_n^k&:= \bigcup_{\substack{ \mu \in \wcomp_{n-1}\\\supp(\mu) 
\subseteq [k]}}\BT_{\mu},\\
\comb_n^k&:= \bigcup_{\substack{ \mu \in \wcomp_{n-1}\\\supp(\mu) \subseteq 
[k]}}\comb_{\mu},\\
\lyn_n^k&:= \bigcup_{\substack{ \mu \in \wcomp_{n-1}\\\supp(\mu) \subseteq 
[k]}}\lyn_{\mu}.
\end{align*}
For a chain $c$ in $\Pi_n^k$, let $$\breve c:= c \setminus \{\hat 0\}.$$
\begin{theorem}\label{theorem:lyndonbasiscohomologyhat}
 The set
$$\{\breve c(T,\sigma) \mid \,(T,\sigma) \in \lyn_n^k, \, \clr(\root(T)) \neq 1 \}$$
is a basis for 
 $\widetilde H^{n-2}(\Pi_n^k \setminus \{\hat 0\})$. 
\end{theorem}
\begin{proof}
 From the EL-labeling of
Theorem
\ref{theorem:ellabelingposet} we have that all the maximal chains of $\widehat{\Pi_n^k}$ have last 
label
$(1,n+1)^1$. Then for a maximal chain to be ascent-free it must have a second to last label of
the form $(1,a)^j$ for $a \in [n]$ and $j \in [k]\setminus \{1\}$. By
Theorem~\ref{thm:ascfreeEL}, we see that the ascent-free chains correspond to
colored Lyndon trees such that the color of the root is different from $1$.  It therefore follows 
from Theorem~\ref{theorem:elth} and Lemma~\ref{lemma:52} (with 
$\bar c$ 
replaced by $\breve c$) that the  set is a basis for $\tilde H^{n-2}(\Pi^k_n \setminus{\hat 
0})$.
\end{proof}

Using an extended version of the straightening algorithm of \cite[Theorem 
5.15]{DleonWachs2013a} with an additional cohomology relation that is more general than the one in 
\cite[Theorem 5.15]{DleonWachs2013a}, 
the reader can verify the following theorem (or see \cite[Proposition 6.3.2]{Dleon2014} for the 
proof).

\begin{proposition}\label{proposition:combbasiscohomologyhat}
 The set $$\{\breve c(T,\sigma) \mid \,(T,\sigma) \in \comb_n^k, \, \clr(\root(T)) \neq k \}$$
spans
 $\widetilde H^{n-2}(\Pi_n^k \setminus \{\hat 0\})$. 
\end{proposition}

\begin{theorem}[{\cite[Theorem 5.15]{DleonWachs2013a}}]
 The set $$\{\breve c(T,\sigma) \mid \,(T,\sigma) \in \comb_n^2, \, \clr(\root(T)) \neq 2 \}$$
is a basis of
 $\widetilde H^{n-2}(\Pi_n^w \setminus \{\hat 0\})$.
\end{theorem}

We propose the following conjecture.

\begin{conjecture}
 The set $$\{\breve c(T,\sigma) \mid \,(T,\sigma) \in \comb_n^k, \, \clr(\root(T)) \neq k \}$$
is a basis of
 $\widetilde H^{n-2}(\Pi_n^k \setminus \{\hat 0\})$. 
\end{conjecture}

We can combine Theorem \ref{theorem:lyndonbasiscohomologyhat} and the exact same idea of the 
proof of Proposition \ref{proposition:lyndononlyrelations} to show that 
the set $B=\{\breve c(T,\sigma) \mid \,(T,\sigma) \in \lyn_n^k, \, \clr(\root(T)) \neq k \}$  
spans $\tilde
H^{n-2}(\Pi_n^k\setminus \{\hat 0 \})$ by using only the relations of  
Theorem~\ref{theorem:binarybasishomology} and the additional relation
\begin{align} \label{eq:type4}
 \breve c(\Upsilon_1 \substack{1 \\ \wedge} \Upsilon_2) +
\breve c(\Upsilon_1 \substack{2 \\ \wedge}\Upsilon_2)+\cdots+
\breve c(\Upsilon_1 \substack{k \\ \wedge} \Upsilon_2) =0.\end{align}
  for all $A \subseteq [n]$ and for all $\Upsilon_1 \in \BT_A$ and $\Upsilon_2 \in 
\BT_{[n]\setminus A}$.
We conclude 
that 
these are the 
only relations in a presentation of $\tilde
H^{n-3}(\Pi_n^k\setminus \{\hat 0 \})$ since $B$ is a basis.  We summarize with the following 
result.

\begin{theorem} 
 The set $\{ \breve{c}(\Upsilon) : \Upsilon \in \BT_n^k \}$ is a generating set 
for $\tilde
H^{n-3}(\Pi_n^k\setminus \{\hat 0 \})$,
subject only to the relations of Theorem~\ref{theorem:binarybasishomology} (with $\bar c$ 
replaced by 
$\breve c$) and relation (\ref{eq:type4}).
\end{theorem}

\section{Whitney numbers and Whitney (co)homology}\label{section:whitneynumbers}
In this section we discuss weighted Whitney numbers and Whitney (co)homology of $\Pi_n^k$.

\subsection{Whitney numbers and weighted uniformity}

Let $P$ 
denote a pure poset 
with a minimum element $\hat 0$. Denote by $Int(P)$ the
set of closed intervals $[x,y]$ in the poset $P$. For some unitary commutative ring $R$ (for 
example $\kk[x]$ or $\kk[x_1,\dots,x_k]$) we say that a \emph{weight function} 
$\varpi_{P}:Int(P)\rightarrow R$ is
\emph{$P$-compatible} if
\begin{itemize}
 \item for any $\alpha\in P$, $\varpi_{P}(\alpha,\alpha)=1$ and,
 \item $\theta \le \alpha \le \beta$ in $P$ implies
$\varpi_{P}(\theta,\beta)=\varpi_{P}(\theta,\alpha)\varpi_{P}(\alpha,\beta)$.
\end{itemize}
Equivalently,
let $\kk[Int(P)]$ be the unitary commutative algebra over $\kk$ generated by intervals
$[\alpha,\beta] \in Int(P)$ subject to the relations:
\begin{itemize}
 \item $[\alpha,\alpha]=1$ for any $\alpha\in P$, and
 \item $[\theta,\beta]=[\theta,\alpha][\alpha,\beta]$ for all $\theta \le \alpha \le \beta$ in $P$.
\end{itemize}
Then a $P$-compatible weight function is just an algebra homomorphism 
$\varpi_{P}:\kk[Int(P)]\rightarrow
R$. The poset $\Pi_n^k$ has a natural $\Pi_n^k$-compatible weight function $\varpi_{\Pi_n^k}$. 
Indeed, we define the map $\varpi_{\Pi_n^k}:\kk[Int(P)]\rightarrow \kk[x_1,\dots,x_k]$ by letting
$\varpi_{\Pi_n^k}(\hat{0},\hat{0})=1$ and $\varpi_{\Pi_n^k}(\hat{0},\alpha)=x_1^{w(1)}\cdots
x_k^{w(k)}$ for any $\alpha = \{A_1^{\mu_1},\dots,A_s^{\mu_s}\} \in \Pi_n^k$, with 
$w=\mu(\alpha)=\sum_{i=1}^s \mu_i$. This extends to any interval 
$[\alpha,\beta]$, by 
setting 
$\varpi_{\Pi_n^k}(\alpha,\beta)=\dfrac{\varpi_{\Pi_n^k}(\hat{0},\beta)}{\varpi_{\Pi_n^k}(\hat{0},
\alpha)}$ 
(clearly a monomial in $\kk[x_1,\dots,x_k]$), and to $\kk[Int(P)]$ by linearity.

The \emph{weighted Whitney numbers} $w_j(P,\varpi_{P})$ and $W_j(P,\varpi_{P})$ of the first and 
second kind
are defined as:
\begin{align*}
 w_j(P,\varpi_{P})&=\sum_{\substack{\alpha \in P\\
\rho(\alpha)=j}}\overline{\mu}_{P}(\hat{0},\alpha)\varpi_{P}(\hat{0},\alpha),\\
 W_j(P,\varpi_{P})&=\sum_{\substack{\alpha \in P\\ \rho(\alpha)=j}}\varpi_{P}(\hat{0},\alpha).
\end{align*}

Note that if $\varpi_{P}$ is the trivial $P$-compatible function defined by 
$\varpi_{P}(\alpha,\alpha^{\prime})=1$ for all $\alpha \le \alpha^{\prime} \in P$, then 
$w_j(P):=w_j(P,\varpi_{P})$ and $W_j(P):=W_j(P,\varpi_{P})$ are the classical Whitney numbers of 
the first and second kind respectively.

Recall that for each 
$\alpha \in \Pi_n^k$, we have
$\rho(\alpha)=n-|\alpha|$. For a partition $\lambda \vdash n$, with $\ell(\lambda)=r$
and where a part of size $i$ occurs
$m_i(\lambda)$ times, let $\lambda\smallsetminus(1^{r})$ denote the partition obtained from 
$\lambda$ by decreasing each of its parts by 1. Recall the symmetric function
\begin{align}\label{equation:defln}
 L_{n}(\xx):=\sum_{\mu \in \wcomp_{n}}\dim \lie(\mu)\,\xx^{\mu},
\end{align}
and for a partition $\lambda=(\lambda_1,\dots,\lambda_r)$, define
\begin{align*}
 L_{\lambda}(\xx):=L_{\lambda_1}(\xx)\cdots L_{\lambda_r}(\xx).
\end{align*}
Note that $L_{\lambda}(\xx)$ is a homogeneous symmetric function of 
degree $|\lambda|$. Define $m(\lambda)!:=\prod_{s=1}^{n}m_s(\lambda)!$.

\begin{proposition}  \label{proposition:whitneynumbers} For all $n \ge 1$,  the weighted Whitney 
numbers are given by 
\begin{align}
w_r(\Pi_n^k,\varpi_{\Pi_n^k}) &= (-1)^{r}\sum_{\substack{\lambda\vdash n\\
\ell(\lambda)=n-r}}\binom{n}{\lambda}\frac{1}{m(\lambda)!}L_{\lambda\smallsetminus(1^{n-r})}(x_1,
\dots,
x_k),\label{equation:whitneyfirst}\\ 
 W_r(\Pi_n^k,\varpi_{\Pi_n^k})&=\sum_{\substack{\lambda \vdash n\\
\ell(\lambda)=n-r}}\binom{n}{\lambda}\frac{1}{m(\lambda)!}h_{\lambda\smallsetminus(1^{n-r})}(x_1,
\dots,
x_k),\label{equation:whitneysecond}
\end{align}
where $h_{\lambda}$ denotes the complete homogeneous symmetric function 
associated with the partition $\lambda$. 
\end{proposition}
\begin{proof}
We want to construct a weighted partition $\alpha$ that has underlying (unweighted) set 
partition $\pi \in \Pi_n$. For a block of size $s$ in $\pi$, 
any monomial $x_1^{\mu(1)}\cdots x_k^{\mu(k)}$ with $|\mu|=s-1$ is a valid weight, so the 
contribution of 
this block corresponds to the complete homogeneous symmetric polynomial 
$h_{s-1}(x_1,\dots,x_k)$.  Then $\pi$ has a contribution of $h_{\lambda(\pi) \smallsetminus 
(1^{|\pi|})}$, where  $\lambda(\pi)$ denotes the integer partition whose parts are equal to the 
block sizes of  $\pi$, proving
equation (\ref{equation:whitneysecond}).

By Proposition \ref{proposition:upperlowerideals}, intervals of the form $[\hat{0}, \alpha]$ are 
isomorphic to products of maximal intervals of smaller copies of $\Pi_n^k$. Following a
similar argument as in (\ref{equation:whitneysecond}), using the fact that the M\"obius 
function is multiplicative and Theorem \ref{theorem:liehomisomorphism},
equation (\ref{equation:whitneyfirst}) follows.
\end{proof}

\begin{definition}
A pure   poset $P$ of length $\ell$ with minimum element $\hat 0$ and with rank function $\rho$, is
said to be {\it uniform}  if
there is a family of posets $\{P_i \mid 0 \le i \le \ell\}$ such that for all 
$x\in P$, the  upper order ideal  $I_x:=\{y \in P \mid x \le y\}$ is isomorphic to $P_i$, where $i 
= 
\ell-
\rho(x)$.
\end{definition}
We refer to
$(P_0,\dots,P_{\ell})$  as the associated {\em uniform sequence}.
It follows from Proposition~\ref{proposition:upperlowerideals} that $P=\Pi_n^k$ is uniform
with $P_i =\Pi_{i+1}^k$ for $i=0,\dots, n-1$.  

Note that a $P$-compatible weight function 
$\varpi_P$ induces, 
for any $x\in P$, an $I_x$-compatible weight function $\varpi_{I_x}$, the restriction of 
$\varpi_P$ to $\kk[Int(I_x)]$. For a uniform poset $P$, we say that a $P$-compatible weight 
function $\varpi_P$ is \emph{uniform} if for any two elements $x,y \in P$ such that 
$\rho(x)=\rho(y)$ there is a poset isomorphism $f:I_x \rightarrow I_y$, such that the induced 
weight functions $\varpi_{I_x}$ and $\varpi_{I_y}$ satisfy 
$\varpi_{I_x}(z,z')=\varpi_{I_y}(f(z),f(z'))$ for all $z \le z' \in I_x$. For example, 
the $\Pi_n^k$-compatible weight function $\varpi_{\Pi_n^k}$ defined before is uniform. 
It is clear that for a  
uniform poset $P$ with associated uniform sequence $(P_0,\dots,P_{\ell})$ and uniform  
$P$-compatible weight function $\varpi_P$ there is a well-defined induced $P_i$-compatible weight 
function $\varpi_{P_i}$ for each $i$. The following proposition is a weighted version of a variant 
of \cite[Exercise 3.130 (a)]{Stanley2012}.

\begin{proposition} \label{uniprop}Let $P$ be a uniform
poset of length $\ell$, with associated uniform sequence $(P_0,\dots,P_{\ell})$ and  a uniform 
$P$-compatible weight function $\varpi_{P}$. 
Then the matrices  $[w_{i-j}(P_i,\varpi_{P_i})]_{0\le i,j \le \ell}$ and 
$[W_{i-j}(P_i,\varpi_{P_i})]_{0\le i,j 
\le \ell}$ are
inverses of each other.
\end{proposition}
\begin{proof}
 For a fixed $\alpha \in P$ with $\rho(\alpha)=\ell-i$ we have by the recursive 
definition of the
M\"obius function and the uniformity of $P$
\begin{align*}
 \delta_{i,j}&=\sum_{\substack{\beta \in
P\\ \rho(\beta)=\ell-j}}\varpi_{P}(\alpha,\beta)\sum_{x \in
[\alpha,\beta]}\overline{\mu}_{P}(\alpha,x)\\
&=\sum_{s=0}^{\ell}\sum_{\substack{x \in P\\
\rho(x)=\ell-s}}\overline{\mu}_{P}(\alpha,x)\varpi_{P}(\alpha,x)
\sum_{\substack{\beta \ge x\\
\rho(\beta)=\ell-j}}\varpi_{P}(x,\beta)\\
&=\sum_{s=0}^{\ell}\sum_{\substack{\tilde{x} \in P_i\\
\rho(\tilde{x})=i-s}}\overline{\mu}_{P_i}(\hat{0},\tilde{x})\varpi_{P_i}(\hat{0},\tilde{x})
\sum_{\substack{\tilde{\beta} \in P_s\\
\rho(\beta)=s-j}}\varpi_{P_s}(\hat{0},\tilde{\beta})\\
&=\sum_{s=0}^{\ell}w_{i-s}(P_i,\varpi_{P_i})W_{s-j}(P_s,\varpi_{P_s}).\qedhere
\end{align*}
\end{proof}

From the uniformity of the pair $(\Pi_n^k,\varpi_{\Pi_n^k})$ and Proposition 
\ref{proposition:whitneynumbers},  we have the
following consequence of Proposition~\ref{uniprop}.
\begin{corollary}\label{corollary:matrices} The matrices $ A=\Biggl [ 
(-1)^{i-j}\sum_{\substack{\lambda\vdash i\\
\ell(\lambda)=j}}\binom{i}{\lambda } 
\frac{1}{m(\lambda)!}L_{\lambda\smallsetminus(1^{j})}(\xx)\Biggr ]_{0\le i,j \le n-1}$ and 
$B=\Biggl [\sum_{\substack{\lambda \vdash i\\
\ell(\lambda)=j}}\binom{i}{\lambda}\frac{1}{m(\lambda)!}h_{\lambda\smallsetminus(1^{j})}(\xx) 
\Biggr ]_{0\le i,j \le n-1}$ are 
inverses of each other.  
\end{corollary}

When $x_1=x_2=1$ and $x_i=0$ for $i\ge 3$, these matrices have the simpler form given in the 
following result.

\begin{theorem}[{\cite[Corollary 2.11]{DleonWachs2013a}}] The matrices $ A=[(-1)^{i-j} \binom 
{i-1} {j-1}i^{i-j} ]_{1\le i,j \le n}$ and 
$B=[\binom i {j} j^{i-j}]_{1\le i,j \le n}$ are inverses of each other.  
\end{theorem}
It can be shown that when  $x_1=1$ and $x_i=0$ for $i\ge 2$, Corollary  \ref{corollary:matrices} 
reduces to the following classical result since $\Pi_n^1=\Pi_n$.
\begin{theorem}[{see \cite{Stanley2012}}] Let $\mathbf{s}(i,j)$ and 
$\mathbf{S}(i,j)$ denote respectively, the Stirling 
numbers of the first and of the second kind. The matrices $ A=[\mathbf{s}(i,j) ]_{1\le i,j \le n}$ 
and 
$B=[\mathbf{S}(i,j)]_{1\le i,j \le n}$ are inverses of each other.
\end{theorem}

\subsection{Whitney (co)homology}\label{section:whitneycohomology}

Whitney homology was introduced by Baclawski in \cite{Baclawski1975}  
giving an affirmative answer to a
question of Rota about the existence of a homology theory on the category of posets where the 
Betti 
numbers for geometric lattices are given by the Whitney numbers of the first kind. Whitney 
homology was later used to compute group representations on the homology of Cohen-Macaulay 
posets by Sundaram \cite{Sundaram1994} and generalized 
to the non-pure case by Wachs \cite{Wachs1999} (see also \cite{Wachs2007}). 

\emph{Whitney cohomology} (over the field ${\bf k}$) of a poset $P$ with a minimum element $\hat 0$ 
can be
defined for each integer $r$ as follows:
\begin{align*}
 WH^r(P):= \bigoplus_{x \in P} \widetilde H^{r-2}((\hat 0, x);{\bf k}).
\end{align*}
In the case of a Cohen-Macaulay poset this formula becomes
\begin{align}\label{equation:defwhitneyhomology}
 WH^r(P):= \bigoplus_{\substack{x \in P\\ \rho(x)=r}} \widetilde H^{r-2}((\hat 0, x);{\bf k}).
\end{align}
Note that
\begin{align}\label{equation:dimensionWH}
 \dim WH^{r}(P)=|w_r(P)|.
\end{align}
where $w_r(P)$ is the classical $r$th Whitney number of the first kind.

Define $\land^r \lie_k(n)$ to be the multilinear component of the $r$th exterior power of the free 
Lie algebra on $[n]$ with $k$ compatible brackets. From the definition of $\land^r 
\lie_k(n)$ and equation (\ref{equation:defln}) we can derive the following proposition.
\begin{proposition}\label{proposition:dimexteriorlie}
 For $0 \le r \le n-1$ and $k\ge 1$, 
\begin{align*}\dim \land^{r} \lie_k(n) = 
\sum_{\substack{\lambda\vdash n\\
\ell(\lambda)=r}}\binom{n}{\lambda}\frac{1}{m(\lambda)!}L_{\lambda\smallsetminus(1^{r})}(1^k). 
\end{align*}
Consequently, if $\land \lie_k(n)$ is the multilinear component of the exterior algebra of the free 
Lie
algebra with $k$ compatible brackets on $n$ generators then
\begin{align*}
\dim \land \lie_k(n) = \sum_{\lambda\vdash 
n}\binom{n}{\lambda}\frac{1}{m(\lambda)!}L_{\lambda\smallsetminus(1^{\ell(\lambda)})}(1^k).
\end{align*}
Equivalently,
\begin{align*}
 \dim \land \lie_k(n)=n![x^n]\exp\left(\sum_{i \ge 1} L_{i-1}(1^k)\dfrac{x^i}{i!}\right),
\end{align*}
where $[x^n]F(x)$ denotes the coefficient of $x^n$ in the formal power series $F(x)$ and
$\exp(x)=\sum_{n\ge 1}\dfrac{x^n}{n!}$ (see \cite[Theorem 5.1.4]{Stanley1999}).
\end{proposition}

Note that since, by equation (\ref{equation:dimensionWH}), $\dim WH^{r}(\Pi_n^k)$ equals the 
signless $r$th Whitney number of the first
kind $|w_r(\Pi_n^k)|$, Propositions 
\ref{proposition:whitneynumbers} and \ref{proposition:dimexteriorlie} imply that the dimensions of 
$\dim WH^{n-r}(\Pi_n^k)$ and $\land^r \lie_k(n)$ are equal.

\begin{corollary}\label{corollary:whitdim} For $0 \le r \le n-1$ and $k\ge 1$, 
\begin{align*}\dim \land^{r} \lie_k(n) = \dim WH^{n-r}(\Pi_n^k). 
\end{align*}
\begin{align*}
\dim \land \lie_k(n) = \dim WH(\Pi_n^k).
\end{align*}
where $WH(\Pi_n^k) := \oplus_{r \ge 0} WH^r(\Pi_n^k)$.
\end{corollary}

If a group $G$ of automorphisms acts on a poset $P$, this action induces a representation of $G$  
on $WH^r(P)$ for every $r$. Thus the action of $\sym_n$ on $\Pi_n^k$ turns  $WH^r(\Pi_n^k)$ into an 
$\sym_n$-module for each $r$. Moreover, the symmetric group $\sym_n$ acts naturally on 
$\land^r \lie_k(n)$ giving it the structure of an $\sym_n$-module.
We will present an equivariant verion of Corollary \ref{corollary:whitdim} below. 

In 
\cite{BarceloBergeron1990} Barcelo and Bergeron proved the  following $\sym_n$-module 
isomorphism 
for the poset of partitions:
\begin{align*}
 WH^{n-r}(\Pi_n) \simeq_{\sym_n} \land^r\lie_1(n) \otimes \sgn_n.
\end{align*}
In \cite{Wachs1998} Wachs shows that an extension of her correspondence between generating sets of
$\widetilde H^{n-3}(\overline{\Pi}_n)$ and $\lie(n) \otimes \sgn_n$ can be used to prove this 
result.
The same technique in \cite{DleonWachs2013a} proves:
\begin{align*}
 WH^{n-r}(\Pi_n^w) \simeq_{\sym_n} \land^r\lie_2(n) \otimes \sgn_n.
\end{align*}

We use the same technique to prove that in general
\begin{align*}
 WH^{n-r}(\Pi_n^k) \simeq_{\sym_n} \land^r\lie_k(n) \otimes \sgn_n.
\end{align*}

A {\em colored binary forest} is a sequence of colored
binary trees.  Given a colored binary forest $F$ with $n$ leaves and  $\sigma \in \sym_n$, let
$(F,\sigma)$ denote the {\it labeled}  colored binary forest  whose $i$th leaf from left to right
has label $\sigma(i)$.  Let $\mathcal {BF}_{n,r}$ be the set of labeled  colored binary forests
with $n$ leaves and $r$ trees.   If the $j$th labeled colored binary tree of 
$(F,\sigma)$ is $(T_j,\sigma_j)$ for each $j=1,\dots r$ then define
$$[F,\sigma] := [T_1,\sigma_1] \land \dots \land [T_r,\sigma_r],$$ where
now $\land$ denotes the wedge product operation in the exterior algebra.  The set $\{[F,\sigma] :
(F,\sigma) \in \mathcal {BF}_{n,r}\}$ is a generating set for $\land^r \lie_k(n)$.

The set $\mathcal {BF}_{n,r}$ also provides a natural generating set for $WH^{n-r}(\Pi_n^k)$.  For
$(F,\sigma) \in \mathcal {BF}_{n,r}$, let $ c(F,\sigma)$ be the unrefinable chain of $\Pi_n^k$ whose
rank $i$ partition is obtained from its rank $i-1$ partition by $\col_i$-merging the blocks $L_i$
and $R_i$, where $\col_i$ is the color of the $i$th postorder internal node $v_i$ of $F$, and $L_i$ 
and $R_i$ are  the respective sets of leaf labels in the left and right subtrees of $v_i$.

We have the following generalization of Theorem~\ref{theorem:liehomisomorphism} and \cite[Theorem
7.2]{Wachs1998}.  The proof is similar to that of Theorem~\ref{theorem:liehomisomorphism} and is
left to the reader.
\begin{theorem}\label{theorem:whitney} For each $r$, there is an $\sym_n$-module isomorphism
$$\phi:\land^r \lie_k(n) \to WH^{n-r}(\Pi_n^k) \otimes \sgn_n$$ determined by
$$\phi([F,\sigma]) = \sgn(\sigma) \sgn(F) \bar c(F,\sigma), \qquad (F,\sigma) \in \mathcal
{BF}_{n,r},$$
where if $F$ is the sequence $T_1,\dots,T_r$ of colored binary trees  then 
$$\sgn(F) :=  (-1)^{I(T_2)+I(T_4) + \dots + I(T_{2\lfloor r/2 \rfloor})} \sgn(T_1) \sgn(T_2) \dots
\sgn(T_r).$$
\end{theorem}

\begin{remark}
When $k=1$ it is proved in \cite{BarceloBergeron1990} that
\begin{align*}
  \dim \land \lie(n) = \dim WH(\Pi_n) = n!,
 \end{align*}
 and when $k=2$ it is proved in \cite{DleonWachs2013a} that , 
 \begin{align*}
  \dim \land \lie_2(n) = \dim WH(\Pi_n^w) = (n+1)^{n-1}.
 \end{align*}
\end{remark}

\section{The Frobenius characteristic of $\lie(\mu)$} \label{section:frobeniuscharacteristic}

In this section  we prove Theorem \ref{theorem:lierepresentation}. We use a technique developed 
by Sundaram \cite{Sundaram1994}, and further developed by Wachs \cite{Wachs1999}, to compute 
group representations on the (co)homology of Cohen-Macaulay posets using Whitney (co)homology.
We introduce and develop first the concepts and results necessary to  prove Theorem 
\ref{theorem:lierepresentation} in Sections \ref{section:wreathproducandplethysm} and 
\ref{section:weightednumberpartitions}; we give a proof 
of the theorem in Section \ref{section:sundaramtechnique}. For information not presented here
about symmetric functions, plethysm  and the representation theory of the symmetric group see 
\cite{Macdonald1995}, \cite{Sagan2001}, \cite{JamesKerber1981} and \cite[Chapter 7]{Stanley1999}.

\subsection{Wreath product modules and plethysm}\label{section:wreathproducandplethysm}

In the following we follow closely the exposition and the results in \cite{Wachs1999}. 

Let $R$ be a commutative ring containing $\QQ$ and let $\Lambda_{R}$ denote the ring of 
symmetric functions with 
coefficients in $R$ with variables $(y_1,y_2,\dots)$. The \emph{power-sum symmetric functions} 
$p_k$ are defined by $p_0=1$ and
\begin{align*}
p_{k}&=y_1^k+y_2^k+\cdots \text{ for }k \in \PP\nonumber.
\end{align*}
For a partition $\lambda\vdash n$, $p_{\lambda}$ denotes the power-sum symmetric 
function associated to $\lambda$, i.e., $p_{\lambda}=p_{\lambda_1}p_{\lambda_2}\cdots 
p_{\lambda_{\ell(\lambda)}}$, where  $\ell(\lambda)$ is the number of nonzero parts of $\lambda$.
It is well-known that the set $\{p_{\lambda} \mid \lambda \vdash n\}$ is a basis for the component 
$\Lambda^n_{R}$ of 
$\Lambda_{R}$ consisting of homogeneous symmetric functions of degree $n$.

Let $\QQ[[z_1,z_2,\dots]]$ be the ring of formal power series in variables $(z_1,z_2,\dots)$. If $g 
\in \QQ[[z_1,z_2,\dots]]$ then  
\emph{plethysm} 
$p_k[g]$ of $p_k$ and $g$ is defined as:
\begin{align}\label{equation:definitionplethysm}
 p_k[g]=g(z_1^k,z_2^k,\dots).
\end{align}
The definition of plethysm is then extended to $p_{\lambda}$ multiplicatively and then to all of 
$\Lambda_{R}$ linearly with respect to $R$. 

It follows from (\ref{equation:definitionplethysm}), that 
if $f \in \Lambda^n_{R}$ and $g\in \QQ[[z_1,z_2,\dots]]$, the following identity holds:
\begin{align}\label{equation:plethysmnegative2}
f[-g]=(-1)^n\omega(f)[g]. 
\end{align}
where $\omega(\cdot)$ is the involution in $\Lambda_{R}$ that maps $p_{i}(\yy)$ to 
$(-1)^{i-1}p_{i}(\yy)$.

For (perhaps empty) integer partitions $\nu$ and $\lambda$ such that $\nu\subseteq \lambda$ 
(that is $\nu(i) \le \lambda(i)$ for all $i$), let
$S^{\lambda / \nu}$ denote the Specht module of shape $\lambda / \nu$ and $s_{\lambda / \nu}$ the 
Schur function of shape  $\lambda / \nu$. Recall that $s_{\lambda / \nu}$ is the image in the ring 
of symmetric functions of the specht module $S^{\lambda / \nu}$  under the Frobenius 
characteristic map $\ch$, i.e., $\ch S^{\lambda / \nu}=s_{\lambda / \nu}$.

We will use the following standard results in the theory of symmetric 
functions and the representation theory of the symmetric group, respectively.
\begin{proposition}[cf. \cite{Macdonald1995} and 
\cite{Wachs1999}]\label{proposition:macdonaldplethysm}
Let $\nu$ be a non empty integer partition and let $\{f_i\}_{i\ge 1}$ be a sequence of formal power 
series $f_i \in \ZZ[z_1,z_2,\dots]$ such that the sum $\sum_i f_i$ exists as a formal power 
series. Then
\begin{align*}
 s_{\lambda}\left [ \sum_{i\ge 1} f_i\right ]=\sum_{\substack{\emptyset= 
 \nu_0 \subseteq \nu_1 \subseteq \cdots \subseteq \nu_{j-1} \subsetneq \nu_j=\lambda\\j\ge 1}} 
\prod_{i=1}^j 
s_{\nu_i/\nu_{i-1}}[f_i].
\end{align*}
\end{proposition}

\begin{proposition}[cf. \cite{JamesKerber1981} and 
\cite{Wachs1999}]\label{proposition:specht}
 Let $\lambda \vdash \ell$ and let $(m_1,m_2,\dots,m_t)$ be a sequence of nonnegative integers 
whose 
sum is $\ell$. Then the restriction of the $\sym_{\ell}$-module $S^{\lambda}$ to the Young subgroup 
$\times_{i=1}^t\sym_{m_i}$ decomposes into a direct sum of outer tensor products of 
$\sym_{m_i}$-modules as follows,
\begin{align*}
 S^{\lambda} \downarrow_{\times \sym_{m_i}}^{\sym_{\ell}} = \bigoplus_{\substack{\emptyset= 
 \nu_0 \subseteq \nu_1 \subseteq \cdots \subseteq \nu_t=\lambda\\|\nu_i|-|\nu_{i-1}|=m_i}}
 \bigotimes_{i=1}^t S^{\nu_i/\nu_{i-1}}.
\end{align*}
\end{proposition}

Recall that the 
\emph{wreath product} of the symmetric groups $\sym_m$ and $\sym_n$, denoted $\sym_m[\sym_n]$, is 
the normalizer of the Young subgroup 
$\stackrel{m \text{ times}}{\overbrace{\sym_n\times \cdots \times  \sym_n}}$ of $\sym_{mn}$. 
Each element of $\sym_m[\sym_n]$ can be represented as an $(m+1)$-tuple 
$(\sigma_1,\dots,\sigma_{m};\tau)$ with $\tau \in \sym_m$ and $\sigma_i \in \sym_n$ for all $i \in 
[m]$.

From an $\sym_n$-module $W$ we can construct a representation $\widetilde{W^{\otimes m}}$ of 
$\sym_m[\sym_n]$ on the vector space $W^{\otimes m}:=\stackrel{m \text{ 
times}}{\overbrace{W\otimes \cdots \otimes W}}$ with action given 
by
\begin{align*} 
(\sigma_1,\dots,\sigma_{m};\tau)(w_1\times\cdots\times w_m):=\sigma_1 w_{\tau^{-1}(1)}
\times\cdots\times \sigma_m w_{\tau^{-1}(m)},
\end{align*}
and from an $\sym_m$-module $V$ we can construct a representation $\widehat V$ of $\sym_m[\sym_n]$ 
with 
action given by
\begin{align*}
 (\sigma_1,\dots,\sigma_{m};\tau)(v):=\tau v,
\end{align*}
called the \emph{pullback} of $V$ from $\sym_m$ to $\sym_m[\sym_n]$.
The \emph{wreath product module} $V[W]$ of the $\sym_m$-module $V$ and the $\sym_n$-module $W$ is 
the $\sym_m[\sym_n]$-module
\begin{align}\label{definition:wreathproductmodule}
 V[W]:=\widetilde{W^{\otimes m}}\otimes \widehat{V},
\end{align}
where $\otimes$ denotes inner tensor product.

\begin{proposition}[\cite{Macdonald1995}]\label{proposition:tensorandwreathch}
Let $V$ be an $\sym_m$-module and $W$ an $\sym_n$-module. Then
\begin{align*}
 \ch \left ( (V \otimes W) \uparrow_{\sym_m \times \sym_n}^{\sym_{m+n}} \right )&= \ch V\ch W,\\
 \ch \left ( V[W]\uparrow_{\sym_m [\sym_n]}^{\sym_{mn}}\right )&=\ch V [\ch W],
\end{align*}
where $\uparrow_*^*$ denotes induction.
\end{proposition}

\subsection{Weighted integer partitions} \label{section:weightednumberpartitions}
Now let $\Phi$ be a finitary totally ordered set and let $||\cdot||: \Phi \rightarrow 
\PP$ be a map. We call a finite multiset $\tlambda=(\tlambda_1 \ge_{\Phi} 
\tlambda_2\ge_{\Phi}\cdots \ge_{\Phi} \tlambda_j)$ of $\Phi$ a \emph{$\Phi$-partition}  of 
\emph{length} $\ell(\tlambda):=j$.
We also define $|\tlambda|:=\sum_{j} ||\tlambda_j||$ and say that $\tlambda$ is a $\Phi$-partition 
of $n$ if  $|\tlambda|=n$. Denote the 
set of $\Phi$-partitions by $\Par(\Phi)$ and the set of $\Phi$-partitions of length $\ell$ by 
$\Par_{\ell}(\Phi)$. For $\phi \in \Phi$, we denote by $m_{\phi}(\tlambda)$, the number of 
times  
$\phi$ appears in $\tlambda$.

Let $V$ be an $\sym_{\ell}$-module, $W_{\phi}$ be an $\sym_{||\phi||}$-module for each 
$\phi \in \Phi$ and $\tlambda$ a $\Phi$-partition with $\ell$ parts.  Note that
$\times_{\phi 
\in \Phi}\sym_{m_{\phi}(\tlambda)}[\sym_ { ||\phi||}]$
 is a finite product since $\tlambda$ is a 
finite multiset.
The module
\begin{align*}
 \bigotimes_{\phi \in \Phi} \widetilde{W_{\phi}^{\otimes m_{\phi}(\tlambda)}} \otimes 
\widehat{V}^{\tlambda},
\end{align*}
is the inner tensor product of two 
$\times_{\phi  \in \Phi}\sym_{m_{\phi}(\tlambda)}[\sym_{||\phi||}]$-modules. The first module is 
the 
outer tensor product 
$ \bigotimes_{\phi \in \Phi} \widetilde{W_{\phi}^{\otimes m_{\phi}(\tlambda)}}$ of the 
$\sym_{m_{\phi}(\tlambda)}[\sym_{||\phi||}]$-modules 
$\widetilde{W_{\phi}^{\otimes 
m_{\phi}(\tlambda)}}$ (cf. Section \ref{section:wreathproducandplethysm}) 
and the second module is the pullback 
$\widehat{V}^{\tlambda}$ of the restricted representation $V\downarrow^{\sym_{\ell}}_{\times_{\phi 
\in \Phi}
\sym_{m_{\phi}(\tlambda)}}$ to 
$\times_{\phi \in \Phi}\sym_{m_{\phi}(\tlambda)}[\sym_{||\phi||}]$ 
through 
the product of the natural homomorphisms 
$\sym_{m_{\phi}(\tlambda)}[\sym_{||\phi||}]\rightarrow 
\sym_{m_{\phi}(\tlambda)}$ given by 
$(\sigma_1,\dots,\sigma_{m_{\phi}(\tlambda)};\tau)\mapsto 
\tau$.

The following theorem generalizes \cite[Theorem 5.5]{Wachs1999} and the proof follows 
the same idea.

\begin{theorem}\label{theorem:wachsplethysm}
 Let $V$ be an $\sym_{\ell}$-module and $W_{\phi}$ be an $\sym_{||\phi||}$-module for each 
$\phi \in \Phi$. Then
 \begin{align*}
  \sum_{\tlambda \in \Par_{\ell}(\Phi)} \ch \left ( \left ( \bigotimes_{\phi \in \Phi} 
\widetilde{W_{\phi}^{\otimes 
m_{\phi}(\tlambda)}} 
\otimes 
\widehat{V}^{\tlambda} \right ) \Biggl 
\uparrow^{\sym_{|\tlambda|}}_{\times_{\phi \in \Phi}\sym_{m_{\phi}(\tlambda)}[\sym_{||\phi||}]} 
\right ) 
z_{\tlambda}=\ch(V)\left [\sum_{\phi \in \Phi}\ch(W_{\phi})z_{\phi}\right ],
\end{align*}
where  $z_{\phi}$ are 
indeterminates with $z_{\tlambda}:=z_{\tlambda_1}\cdots z_{\tlambda_{\ell}}$ .
\end{theorem}

\begin{proof}
Note that restriction, induction, pullback, $\ch$ and plethysm in the outer 
component are all linear and inner tensor product is bilinear. Thus it is enough to 
prove the theorem for $V$ equal to an irreducible $\sym_{\ell}$-module $S^{\eta}$ (the Specht 
module associated to 
the partition $\eta\vdash \ell$). 
Since the set $\Phi$ is a finitary totally ordered set, we can denote
by $\phi^{i}$, the $i$th element in the total order of $\Phi$.
Consider $\tlambda \in \Par_{\ell}(\Phi)$ and let  $t:=\max\{i \mid \phi^i \in \tlambda\}$.
Now using Proposition \ref{proposition:specht} and the definition of a wreath product module in 
equation (\ref{definition:wreathproductmodule}) yields
\begin{align*}
 \bigotimes_{i=1}^{t} \widetilde{W_{\phi^{i}}^{\otimes m_{\phi^{i}}(\tlambda)}} \otimes 
\widehat{S^{\eta}}^{\tlambda}&=\bigotimes_{i=1}^{t} \widetilde{W_{\phi^{i}}^{\otimes 
m_{\phi^{i}}(\tlambda)}} \otimes \left (\bigoplus_{\substack{\emptyset= 
 \nu_0 \subseteq \nu_1 \subseteq \cdots \subseteq \nu_t=\eta 
\\|\nu_i|-|\nu_{i-1}|=m_{\phi^{i}}(\tlambda)}}
 \bigotimes_{i=1}^t \widehat{S^{\nu_i/\nu_{i-1}}}\right )\nonumber\\
 &=\bigoplus_{\substack{\emptyset= 
 \nu_0 \subseteq \nu_1 \subseteq \cdots \subseteq \nu_t=\eta 
\\|\nu_i|-|\nu_{i-1}|=m_{\phi^{i}}(\tlambda)}}
\left ( \bigotimes_{i=1}^{t} \widetilde{W_{\phi^{i}}^{\otimes 
m_{\phi^{i}}(\tlambda)}} \otimes 
 \bigotimes_{i=1}^t \widehat{S^{\nu_i/\nu_{i-1}}}\right ) \nonumber\\
 &=\bigoplus_{\substack{\emptyset= 
 \nu_0 \subseteq \nu_1 \subseteq \cdots \subseteq \nu_t=\eta 
\\|\nu_i|-|\nu_{i-1}|=m_{\phi^{i}}(\tlambda)}}
\bigotimes_{i=1}^{t} \left ( \widetilde{W_{\phi^{i}}^{\otimes 
m_{\phi^{i}}(\tlambda)}} \otimes 
\widehat{S^{\nu_i/\nu_{i-1}}}\right ) \nonumber\\
&=\bigoplus_{\substack{\emptyset= 
 \nu_0 \subseteq \nu_1 \subseteq \cdots \subseteq \nu_t=\eta 
\\|\nu_i|-|\nu_{i-1}|=m_{\phi^{i}}(\tlambda)}}
\bigotimes_{i=1}^{t} S^{\nu_i/\nu_{i-1}}[W_{\phi^{i}}].\nonumber
\end{align*}
We induce and then apply the Frobenius characteristic map $\ch$. Then using 
Proposition
\ref{proposition:tensorandwreathch} and the transitivity 
property of induction of representations, we have that
\begin{align*}
 &\ch \left ( \left ( \bigotimes_{i=1}^{t} \widetilde{W_{\phi^{i}}^{\otimes 
m_{\phi^{i}}(\tlambda)}} 
\otimes 
\widehat{S^{\eta}}^{\tlambda}\right )
\Biggl \uparrow^{\sym_{|\tlambda|}}_{\times_{i=1}^t\sym_{m_{\phi^{i}}(\tlambda) } [ 
\sym_{||\phi^{i}||}]} \right )
\hspace{2.5in} \nonumber\\
=& \ch \left ( \bigoplus_{\substack{\emptyset= 
 \nu_0 \subseteq \nu_1 \subseteq \cdots \subseteq \nu_t=\eta 
\\|\nu_i|-|\nu_{i-1}|=m_{\phi^{i}}(\tlambda)}}
\left (\bigotimes_{i=1}^{t} 
S^{\nu_i/\nu_{i-1}}[W_{\phi^{i}}]\right)
\Biggl \uparrow^{\sym_{|\tlambda|}}_{\times_{i=1}^t\sym_{m_{\phi^{i}}(\tlambda)}[
\sym_{||\phi^{i}||}]} \right ) \nonumber\\
=&  \sum_{\substack{\emptyset= 
 \nu_0 \subseteq \nu_1 \subseteq \cdots \subseteq \nu_t=\eta 
\\|\nu_i|-|\nu_{i-1}|=m_{\phi^{i}}(\tlambda)}}
\ch \left ( \left (\bigotimes_{i=1}^{t} 
S^{\nu_i/\nu_{i-1}}[W_{\phi^{i}}]\right)
\Biggl \uparrow^{\sym_{|\tlambda|}}_{\times_{i=1}^t\sym_{m_{\phi^{i}}(\tlambda)}[
\sym_{||\phi^{i}||}]} \right ) \nonumber\\
=&  \sum_{\substack{\emptyset= 
 \nu_0 \subseteq \nu_1 \subseteq \cdots \subseteq \nu_t=\eta 
\\|\nu_i|-|\nu_{i-1}|=m_{\phi^{i}}(\tlambda)}}
\ch \left (  \left ( \left( \bigotimes_{i=1}^{t} 
S^{\nu_i/\nu_{i-1}}[W_{\phi^{i}}]\right)
\Biggl 
\uparrow^{\times_{i=1}^t\sym_{m_{\phi^{i}}(\tlambda)||\phi^{i}||}}_{\times_{i=1}^t\sym_{m_{\phi^{i}}
(\tlambda)}[
\sym_{||\phi^{i}||}]} \right)
\Biggl \uparrow^{\sym_{|\tlambda|}}_{\times_{i=1}^t\sym_{m_{\phi^{i}}(\tlambda)||\phi^{i}||}} 
\right)
\nonumber\\
=&  \sum_{\substack{\emptyset= 
 \nu_0 \subseteq \nu_1 \subseteq \cdots \subseteq \nu_t=\eta 
\\|\nu_i|-|\nu_{i-1}|=m_{\phi^{i}}(\tlambda)}}
\ch \left ( \left (\bigotimes_{i=1}^{t} \left(
S^{\nu_i/\nu_{i-1}}[W_{\phi^{i}}]
\Biggl 
\uparrow^{\sym_{m_{\phi^{i}}(\tlambda)||\phi^{i}||}}_{\sym_{m_{\phi^{i}}
(\tlambda)}[
\sym_{||\phi^{i}||}]}\right) \right)
\Biggl \uparrow^{\sym_{|\tlambda|}}_{\times_{i=1}^t\sym_{m_{\phi^{i}}(\tlambda)||\phi^{i}||}} 
\right )
\nonumber\\
=&  \sum_{\substack{\emptyset= 
 \nu_0 \subseteq \nu_1 \subseteq \cdots \subseteq \nu_t=\eta 
\\|\nu_i|-|\nu_{i-1}|=m_{\phi^{i}}(\tlambda)}}
\prod_{i=1}^{t}\ch  \left(
S^{\nu_i/\nu_{i-1}}[W_{\phi^{i}}]
\Bigl 
\uparrow^{\sym_{m_{\phi^{i}}(\tlambda)||\phi^{i}||}}_{\sym_{m_{\phi^{i}}
(\tlambda)}[
\sym_{||\phi^{i}||}]}\right)
\nonumber\\
=&  \sum_{\substack{\emptyset= 
 \nu_0 \subseteq \nu_1 \subseteq \cdots \subseteq \nu_t=\eta 
\\|\nu_i|-|\nu_{i-1}|=m_{\phi^{i}}(\tlambda)}}
\prod_{i=1}^{t}
s_{\nu_i/\nu_{i-1}}[\ch W_{\phi^{i}}].
\nonumber
\end{align*}
Now note that $s_{\nu_i/\nu_{i-1}}[\ch W_{\phi}z_{\phi}]=s_{\nu_i/\nu_{i-1}}[\ch 
W_{\phi}]z_{\phi}^{|\nu_i|-|\nu_{i-1}|}$ and that $s_{\nu_i/\nu_{i-1}}=s_{\emptyset}=1$ if 
$\nu_i=\nu_{i-1}$. And using Proposition \ref{proposition:macdonaldplethysm}, we obtain
\begin{align*}
  \sum_{\tlambda \in \Par_{\ell}(\Phi)} \ch \left ( \left ( \bigotimes_{i\ge 1} 
\widetilde{W_{\phi^{i}}^{\otimes 
m_{\phi^{i}}(\tlambda)}} 
\otimes 
\widehat{S^{\eta}}^{\tlambda} \right ) \Biggl 
\uparrow^{\sym_{|\tlambda|}}_{\times_{i=1}^t\sym_{m_{\phi^{i}}(\tlambda)}[\sym_{||\phi^{i}||}]} 
\right ) 
z_{\tlambda}\hspace{2.5in}\nonumber
\end{align*}
\begin{align*}
 =&\sum_{\tlambda \in \Par_{\ell}(\Phi)} \sum_{\substack{\emptyset= 
 \nu_0 \subseteq \nu_1 \subseteq \cdots \subseteq \nu_t=\eta 
\\|\nu_i|-|\nu_{i-1}|=m_{\phi^{i}}(\tlambda)}}
\prod_{i=1}^{t}
s_{\nu_i/\nu_{i-1}}[\ch W_{\phi^{i}}z_{\phi^{i}}]\nonumber\\
 =&\sum_{\substack{\emptyset= 
 \nu_0 \subseteq \nu_1 \subseteq \cdots \subseteq \nu_{j-1} \subsetneq \nu_j=\eta\\ j\ge 1}} 
\prod_{i=1}^{j}
s_{\nu_i/\nu_{i-1}}[\ch W_{\phi^{i}}z_{\phi^{i}}]\nonumber\\
=&s_{\eta}\left[\sum_{i\ge1}\ch W_{\phi^{i}}z_{\phi^{i}}\right]\nonumber\\
=&\ch S^{\eta}\left[\sum_{\phi \in \Phi} \ch W_{\phi}z_{\phi}\right].\nonumber\qedhere
\end{align*}
\end{proof}

\subsection{Using Whitney (co)homology to compute (co)homology} \label{section:sundaramtechnique}

The technique of Sundaram \cite{Sundaram1994} to compute characters of  
$G$-representations on the (co)homology of pure $G$-posets is based on the following result: 
\begin{lemma}[\cite{Sundaram1994} Lemma 1.1]\label{lemma:sundaramsum}
Let $P$ be a bounded poset of length $\ell\ge 1$ and let $G$ be a group of automorphisms of $P$.
Then the following isomorphism of virtual $G$-modules holds
\begin{align*}
 \bigoplus_{i=0}^{\ell}(-1)^{i}WH^i(P)\cong_{G}0.
\end{align*}
\end{lemma}

Recall that if a group $G$ of automorphisms acts on the poset $P$, this action induces a 
representation of $G$  
on $WH^r(P)$ for every $r$. From equation 
(\ref{equation:defwhitneyhomology}), when $P$ is Cohen-Macaulay, $WH^r(P)$ breaks into the direct 
sum of 
$G$-modules
\begin{align}\label{equation:defwhitneyhomologyinduced}
 WH^r(P)\cong_{P} \bigoplus_{\substack{x \in P/\sim\\ \rho(x)=r}} \widetilde H^{r-2}((\hat 0, 
x);{\bf 
k}) \bigl \uparrow _{G_x}^G,
\end{align}
where  $P/\sim$ is a set of orbit representatives and $G_x$ the stabilizer of $x$.

Let $\mu \in {\wcomp_{n-1}}$. We want to apply Lemma \ref{lemma:sundaramsum} to the 
dual poset $[\hat{0},[n]^{\mu}]^{*}$ of the maximal interval $[\hat{0},[n]^{\mu}]$, which by 
Theorem \ref{theorem:elk} is Cohen-Macaulay. In order to compute $WH^r([\hat{0},[n]^{\mu}]^{*})$, 
by equation 
(\ref{equation:defwhitneyhomologyinduced}), we need to 
specify a set of orbit representatives for the action of $\sym_n$ on $[\hat{0},[n]^{\mu}]^{*}$.
For this we consider the set  
\begin{align*}
\Phi=\{\phi \in \wcomp \mid \supp(\phi) \subseteq [k]\} 
\end{align*}
 and the map 
$||\phi||:=|\phi|+1$ 
for $\phi \in \wcomp$ (cf. Section \ref{section:weightednumberpartitions}). We fix any 
finitary total order on $\Phi$. For any  
$\Phi$-partition 
$\tlambda$ of $n$ of length $\ell$ we denote by $\alpha_{\tlambda}$, the weighted partition 
$\{A_1^{\tlambda_1},\dots, A_{\ell}^{\tlambda_{\ell}}\}$ of $[n]$ whose  blocks are of the form 
\begin{align*}
A_i=\left [\sum_{j=1}^i||\tlambda_j||\right]\setminus 
\left[\sum_{j=1}^{i-1}||\tlambda_j||\right]. 
\end{align*}

Recall that for $\nu,\mu
\in \wcomp$, we say that $\mu\le \nu$ if $\mu(i)\le \nu(i)$
for every $i$ and we denote by $\nu + \mu$, the weak composition 
defined by 
$(\nu + \mu)(i):=\nu(i) + \mu(i)$.
Let 
\begin{align*}
\Par^{\mu}(\Phi):=\{\tlambda \in \Par(\Phi) \mid |\tlambda|=||\mu||\,,\,\sum_{i}\tlambda_i\le 
\mu\}. 
\end{align*}
It is not difficult to see that 
$\{\alpha_{\tlambda} \mid \tlambda \in \Par^{\mu}(\Phi)\}$ is a set of orbit representatives for 
the action of $\sym_n$ on $[\hat{0},[n]^{\mu}]^{*}$. Indeed, any weighted partition $\beta \in 
[\hat{0},[n]^{\mu}]^{*}$ can be obtained as $\beta=\sigma \alpha_{\tlambda}$ for suitable $\tlambda 
\in \Par^{\mu}(\Phi)$ and $\sigma \in \sym_n$. 
It is also clear that $\alpha_{\tlambda} \ne \sigma 
\alpha_{\tlambda^{\prime}}$ for $\tlambda 
\ne \tlambda^{\prime} \in \Par^{\mu}(\Phi)$ and  
for every $\sigma \in \sym_n$.
Observe that the partition $\alpha_{\tlambda}$ has  stabilizer 
$\times_{\phi}\sym_{m_{\phi}(\tlambda)}[\sym_{||\phi||}]$. By equation
(\ref{equation:defwhitneyhomologyinduced}) applied to $[\hat{0},[n]^{\mu}]^{*}$,

\begin{align}\label{equation:whitneydualinterval}
  WH^r([\hat{0},[n]^{\mu}]^{*})\cong_{\sym_n}\bigoplus_{\substack{\tlambda \in 
\Par^{\mu}(\Phi)\\ \ell(\tlambda)=r}} \widetilde 
{H}^{r-3}((\alpha_{\tlambda},[n]^{\mu}))\bigl 
\uparrow_{\times_{\phi \in \Phi}\sym_{m_{\phi}(\tlambda)}[\sym_{||\phi||}]}^{\sym_{n}}.
\end{align}

Note that if  $r=2$ then the open interval $(\alpha_{\tlambda},[n]^{\mu})$  is the empty poset.  
Hence $\widetilde {H}^{r-3}((\alpha_{\tlambda},[n]^{\mu}))$ is isomorphic to the trivial 
representation of $\times_{\phi \in \Phi}\sym_{m_{\phi}(\tlambda)}[\sym_{||\phi||}]$.  If $r=1$ 
then 
$\alpha_{\tlambda}=[n]^\mu$.  In this case we use the convention that $\widetilde 
{H}^{r-3}((\alpha_{\tlambda},[n]^{\mu}))$ is isomorphic to the trivial representation of 
$\times_{\phi \in \Phi}\sym_{m_{\phi}(\tlambda)}[\sym_{||\phi||}]$.

We apply Lemma \ref{lemma:sundaramsum} together with equation 
(\ref{equation:whitneydualinterval}) to obtain the following result.

\begin{lemma} \label{lemma:equivariantrecursivity}For $n\ge 1$ and $\mu \in \wcomp_{n-1}$ we have 
the 
following $\sym_n$-module isomorphism

\begin{align}\label{equation:equivariantrecursivity}
  \mathbf{1}_{\sym_n}\delta_{n,1} \cong_{\sym_n}\bigoplus_{\substack{\tlambda \in 
\Par^{\mu}(\Phi)\\|\tlambda|=n}} (-1)^{\ell(\tlambda)-1}\widetilde 
{H}^{\ell(\tlambda)-3}((\alpha_{\tlambda},[n]^{\mu}))\bigl 
\uparrow_{\times_{\phi \in \Phi}\sym_{m_{\phi}(\tlambda)}[\sym_{||\phi||}]}^{\sym_{n}},
\end{align}
where $\mathbf{1}_{\sym_n}$ denotes the trivial representation of $\sym_n$.

\end{lemma}

\begin{lemma}\label{lemma:isomodules}For all $\tlambda \in \Par(\Phi)$ with $|\tlambda|=n$ and $\nu 
\in \wcomp_{\ell(\tlambda)-1}$, the following 
$\times_{\phi \in \Phi}\sym_{m_{\phi}(\tlambda)}[\sym_{||\phi||}]$-module isomorphism 
holds:
\begin{align*}
\widetilde{H}^{\ell(\tlambda)-3}((\alpha_{\tlambda},[n]^{\nu+\sum \tlambda_j}))
\cong
\left ( \bigotimes_{\phi \in \Phi} \widetilde{(\mathbf{1}_{\sym_{||\phi||}})^{\otimes 
m_{\phi}(\tlambda)}} \right )
\otimes \widehat{\widetilde{H}^{\ell(\tlambda)-3}((\hat{0},[\ell(\tlambda)]^{\nu}))^{\tlambda}}.
\end{align*}
\end{lemma}
\begin{proof}The poset $[\hat{0},[\ell(\tlambda)]^{\nu}]$ is a 
$\times_{\phi \in \Phi}\sym_{m_{\phi}(\tlambda)}[\sym_{||\phi||}]$-poset with the action given by 
the 
pullback through the 
product of the natural homomorphisms $\sym_{m_{\phi}(\tlambda)}[\sym_{||\phi||}]\rightarrow 
\sym_{m_{\phi}(\tlambda)}$.
There is a natural poset isomorphism between $[\alpha_{\tlambda},[n]^{\nu+\sum \tlambda_j}]$ 
and $[\hat{0},[\ell(\tlambda)]^{\nu}]$. Indeed, for a weighted partition  
$\{B_1^{\mu_1},\dots, 
B_t^{\mu_t}\} \ge\alpha_{\tlambda}=\{A_1^{\tlambda_1},\dots, A_{\ell}^{\tlambda_{\ell}}\}$, each 
weighted block $B_j^{\mu_j}$ is of the form $B_j=A_{i_1}\cup A_{i_2}\cup ... \cup A_{i_s}$ 
and 
 $\mu_j=u_j+\sum_k \tlambda_{i_k}$, where $|u_j|=s-1$ and $\sum_j u_j \le \nu$. 
 Let \begin{align*}
      \Gamma:[\alpha_{\tlambda},[n]^{\nu+\sum \tlambda_j}]\rightarrow 
      [\hat{0},[\ell(\tlambda)]^{\nu}]
     \end{align*}
     be the map such that $\Gamma(\{B_1^{\mu_1},\dots, 
B_t^{\mu_t}\})$
is the weighted partition in which each weighted block $B_j^{\mu_j}$ is replaced by  
$\{i_1,i_2,\dots,i_s\}^{u_j}$.
The map $\Gamma$ is an isomorphism of posets that commutes with the action 
of $\times_{\phi \in \Phi}\sym_{m_{\phi}(\tlambda)}[\sym_{||\phi||}]$. The 
isomorphism 
of 
$\times_{\phi \in \Phi}\sym_{m_{\phi}(\tlambda)}[\sym_{||\phi||}]$-posets induces an isomorphism of 
the 
$\times_{\phi \in \Phi}\sym_{m_{\phi}(\tlambda)}[\sym_{||\phi||}]$-modules 
$\widetilde{H}^{\ell(\tlambda)-3}((\alpha_{\tlambda},[n]^{\nu+\sum \tlambda_j}))$ and 
$\widehat{\widetilde{H}^{\ell(\tlambda)-3}((\hat{0},[\ell(\tlambda)]^{\nu}))^{\tlambda}}$.
The result follows since 
$\bigotimes_{\phi \in \Phi} \widetilde{(\mathbf{1}_{\sym_{||\phi||}})^{\otimes 
m_{\phi}(\tlambda)}}$ 
is the 
trivial 
representation of $\times_{\phi \in \Phi}\sym_{m_{\phi}(\tlambda)}[\sym_{||\phi||}]$.
\end{proof}

Let $R$ be the ring of symmetric functions $\Lambda_{\QQ}$ in variables $\xx=(x_1,x_2,\dots)$. 
There is a natural inner product in $\Lambda_R$ defined for arbitrary 
partitions $\lambda$ and $\nu$, by 
$$\langle p_{\lambda},p_{\nu} \rangle=\delta_{\lambda,\nu},$$  and then
extended linearly to $\Lambda_R$. This inner product defines a notion of 
convergence. Indeed, 
for a sequence of symmetric functions $f_n \in \Lambda_R$, $n \ge 1$, we say that 
$\{f_n\}_{n\ge1}$ converges 
if for every partition $\nu$ there is a number $N$ such that 
$\langle f_n,p_{\nu}\rangle=\langle f_m,p_{\nu}\rangle$ whenever 
$n,m \ge N$. We use $\widehat{\Lambda_R}$ to denote the completion of the ring of $\Lambda_R$ with 
respect to 
this topology.
It is not difficult to verify that $\widehat{\Lambda_R}$ consists of the class of formal 
power series in two sets of variables, 
$\xx=(x_1,x_2,\dots)$ and $\yy=(y_1,y_2,\dots)$, 
that can be expressed as $\sum_{\lambda} c_{\lambda}(\xx) 
p_{\lambda}(\yy)$, where $c_{\lambda}(\xx) \in \Lambda_{\QQ}$.  
Given a formal power series $F(\yy)=\sum_{\lambda} c_{\lambda}(\xx) 
p_{\lambda}(\yy)$ in $\widehat{\Lambda_R}$ and a formal power series 
$g \in \QQ[[z_1,z_2,\dots]]$, we can extend the definition of plethysm from 
symmetric 
functions to formal power series in $\widehat{\Lambda_R}$ by 
\begin{align*}
 F[g]:=\sum_{\lambda} c_{\lambda}(\xx) p_{\lambda}[g].
\end{align*}
The reader can check that $\widehat{\Lambda_R}$, together with 
plethysm and the plethystic unit $p_1(\yy)$, has the structure of a monoid.

 Let $G(\yy)$ and $F(\yy)$ be in $\widehat{\Lambda_R}$. 
 The 
power series $G(\yy)$ is said to be a \emph{plethystic 
inverse} of $F(\yy)$ with respect to $\yy$, if $F(\yy)[G(\yy)] = p_1(\yy)$. It is 
straightforward to show that 
if 
this is the case, then $F(\yy)$ is unique and also $G(\yy)[F(\yy)] = p_1(\yy)$. Thus $G(\yy)$ and 
$F(\yy)$ are 
said to be \emph{plethystic inverses} of each other with respect to $\yy$, and we write 
$G(\yy)=F^{[-1]}(\yy)$. Note that 
$$\sum_{\mu \in \wcomp_{n-1}}\ch  
\widetilde{H}^{n-3}((\hat{0},[n]^{\mu}))\,\xx^{\mu}=\sum_{\lambda 
\vdash {n-1}}\ch  
\widetilde{H}^{n-3}((\hat{0},[n]^{\lambda}))\,m_{\lambda}(\xx) \in \Lambda_R,$$
where $m_{\lambda}$ is the monomial symmetric function associated to the partition $\lambda$.
Hence the left hand side of equation (\ref{equation:cohomologyrep}) below is in  
$\widehat{\Lambda_R}$.

\begin{theorem}\label{theorem:cohomologyrepresentation}
We have 
\begin{align}\label{equation:cohomologyrep}
 \sum_{n\ge 1}(-1)^{n-1} \sum_{\mu \in \wcomp_{n-1}}\ch  
\widetilde{H}^{n-3}((\hat{0},[n]^{\mu}))\,\xx^{\mu}=\Bigl (\sum_{n\ge
1}h_{n-1}(\xx)h_{n}(\yy)
\Bigr)^{[-1]}
 \end{align}
 
\end{theorem}
\begin{proof}
 For presentation purposes let us use temporarily the notation 
$G_{\tlambda}:=\times_{\phi\in\Phi}\sym_{m_{\phi}(\tlambda)}[\sym_{||\phi||}]$. We also use the 
convention that $$\widetilde 
{H}^{\ell(\tlambda)-3}((\alpha_{\tlambda},[n]^{\mu}))=0$$ whenever
$\alpha_{\tlambda} \nleq [n]^{\mu}$.
Applying the Frobenius characteristic map $\ch$ (in $\yy$ variables) to both sides of equation 
\ref{equation:equivariantrecursivity}, multiplying by 
$\xx^{\mu}$ and summing over all $\mu \in \wcomp_{n-1}$ with  
$\supp(\mu)\subseteq [k]$ yields
\begin{align*}
&h_1(\yy)\delta_{n,1}=\sum_{\substack{\mu \in \wcomp_{n-1}\\\supp(\mu)\subseteq 
[k]}}\mathbf{x}^{\mu}
\ch \left( \bigoplus_{\substack{\tlambda \in 
\Par^{\mu}(\Phi)\\|\tlambda|=n}}(-1)^{\ell(\tlambda)-1}
 \widetilde 
{H}^{\ell(\tlambda)-3}((\alpha_{\tlambda},[n]^{\mu}))\bigl 
\uparrow_{G_{\tlambda}}^{\sym_{n}}\right )\\
=&\sum_{\substack{\mu \in \wcomp_{n-1}\\\supp(\mu)\subseteq [k]}}\mathbf{x}^{\mu}
\sum_{\substack{\tlambda \in 
\Par^{\mu}(\Phi)\\|\tlambda|=n}}(-1)^{\ell(\tlambda)-1}
\ch \left( \widetilde 
{H}^{\ell(\tlambda)-3}((\alpha_{\tlambda},[n]^{\mu}))\bigl 
\uparrow_{G_{\tlambda}}^{\sym_{n}}\right )\\
=&\sum_{\substack{\tlambda \in \Par(\Phi)\\|\tlambda|=n}}(-1)^{\ell(\tlambda)-1}
\sum_{\substack{\nu \in \wcomp_{\ell(\tlambda)-1}\\ \supp(\nu)\subseteq [k]}}
\xx^{\nu+\sum \tlambda_r} \ch \left ( \widetilde 
{H}^{\ell(\tlambda)-3}((\alpha_{\tlambda},[n]^{\nu+ \sum \tlambda_r}))
\bigl \uparrow^{\sym_{n}}_{G_{\tlambda}} \right ).
\end{align*}
Using the shorthand notation
\begin{align*}
 H_n^{\mu}:= \widetilde{H}^{n-3}((\hat{0},[n]^{\mu})),
\end{align*}
Lemma \ref{lemma:isomodules} and summing over all $n \ge 1$ we have

\begin{align*}
&h_1(\yy)\\
&= \sum_{\tlambda \in \Par(\Phi)}(-1)^{\ell(\tlambda)-1}
		\sum_{\substack{\nu \in \wcomp_{\ell(\tlambda)-1}\\ \supp(\nu)\subseteq [k]}}
		\xx^{\nu+\sum \tlambda_r} \ch \left ( \left ( \bigotimes_{\phi\in\Phi} 
\widetilde{(\mathbf{1}_{\sym_{||\phi||}})^{\otimes m_{\phi}(\tlambda)}}
\otimes \widehat{H_{\ell(\tlambda)}^{\nu}}^{\tlambda}\right)
\Biggl \uparrow^{\sym_{|\tlambda|}}_{G_{\tlambda}}\right )\\
&= 	\sum_{\ell \ge 1}(-1)^{\ell-1} \sum_{\tlambda \in \Par_{\ell}(\Phi)}
		\sum_{\substack{\nu \in \wcomp_{\ell-1}\\ \supp(\nu)\subseteq [k]}}
		\xx^{\nu+\sum \tlambda_r} \ch \left ( \left ( \bigotimes_{\phi\in\Phi} 
\widetilde{(\mathbf{1}_{\sym_{||\phi||}})^{\otimes m_{\phi}(\tlambda)}}
\otimes \widehat{H_{\ell}^{\nu}}^{\tlambda}\right)
\Biggl \uparrow^{\sym_{|\tlambda|}}_{G_{\tlambda}} \right )\\
&= 	\sum_{\ell \ge 1}(-1)^{\ell-1} 
		\sum_{\substack{\nu \in \wcomp_{\ell-1}\\ \supp(\nu)\subseteq [k]}}
		\xx^{\nu} 
		\sum_{\tlambda \in \Par_{\ell}(\Phi)} 
\mathbf{x}^{\sum \tlambda_r}
		\ch \left ( \left ( \bigotimes_{\phi\in\Phi} 
\widetilde{(\mathbf{1}_{\sym_{||\phi||}})^{\otimes m_{\phi}(\tlambda)}}
\otimes \widehat{H_{\ell}^{\nu}}^{\tlambda}\right)
\Biggl \uparrow^{\sym_{|\tlambda|}}_{G_{\tlambda}} \right ).
\end{align*}

Now we use Theorem \ref{theorem:wachsplethysm} with $z_{\phi}=\xx^{\phi}$,
\begin{align*}
	 h_1(\yy)   & = 	\sum_{\ell\ge 1}(-1)^{\ell-1}
\sum_{\substack{\nu \in \wcomp_{\ell-1}\\ \supp(\nu)\subseteq [k]}}\xx^{\nu}
\ch(H_{\ell}^{\nu})\left [\sum_{\substack{\phi \in \wcomp\\ \supp(\phi)\subseteq 
[k]}}h_{||\phi||}(\yy)\xx^{\phi}\right ]  \\
& = \left( \sum_{\ell\ge 1}(-1)^{\ell-1}
\sum_{\substack{\nu \in \wcomp_{\ell-1}\\ \supp(\nu)\subseteq [k]}}
\ch(H_{\ell}^{\nu}) \xx^{\nu} \right ) \left [\sum_{j\ge 1}h_j(\yy)h_{j-1}(x_1,\dots,x_k)\right ] .
\end{align*}
The last step uses the definition of the complete homogeneous symmetric polynomial 
$h_{j-1}(x_1,\dots,x_k)$. To 
complete the proof we let $k$ get arbitrarily large.
\end{proof}

\begin{theorem}[Theorem \ref{theorem:lierepresentation}]
 We have that 
\begin{align*}
 \sum_{n\ge 1} \sum_{\mu \in \wcomp_{n-1}}\ch \lie(\mu)\,\xx^{\mu}=-\Bigl 
(-\sum_{n\ge
1}h_{n-1}(\xx)h_{n}(\yy)
\Bigr)^{[-1]},
\end{align*}
\end{theorem}
\begin{proof}
We have 
\begin{align*}
p_1(\yy)&=\left(\sum_{n\ge 1}(-1)^{n-1} \sum_{\mu \in \wcomp_{n-1}}\ch  
\widetilde{H}^{n-3}((\hat{0},[n]^{\mu}))\,\xx^{\mu}\right 
)\left[\sum_{n\ge1}h_{n-1}(\xx)h_{n}(\yy)\right]\\
&=\left(-\sum_{n\ge 1} \sum_{\mu \in\wcomp_{n-1}} 
\omega  \left (\ch \widetilde{H}^{n-3}((\hat{0},[n]^{\mu}))
 \right)\,\xx^{\mu}\right )\left[-\sum_{n\ge1}h_{n-1}(\xx)h_{n}(\yy)\right]\\
&=\left(-\sum_{n\ge 1} \sum_{\mu \in\wcomp_{n-1}}\ch 
\left ( \widetilde{H}^{n-3}((\hat{0},[n]^{\mu}))\otimes_{\sym_n} \sgn_n
 \right)\,\xx^{\mu}\right )
 \left[-\sum_{n\ge1}h_{n-1}(\xx)h_{n}(\yy)\right]\\
&= \left(-\sum_{n\ge 1} \sum_{\mu \in \wcomp_{n-1}}\ch \lie(\mu)\,\xx^{\mu} \right )
 \left[-\sum_{n\ge1}h_{n-1}(\xx)h_{n}(\yy)\right].
\end{align*}
The first two equalities above follow from Theorem \ref{theorem:cohomologyrepresentation} and 
from equation (\ref{equation:plethysmnegative2}), respectively. Recall that for an $\sym_n$-module 
$V$ we have that  $$\ch(V \otimes_{\sym_n} \sgn_n)=\omega(\ch 
V),$$ \\\\which proves the third equality. Finally the last equality makes use of
Theorem \ref{theorem:liehomisomorphism}.
\end{proof}

In the case $k=2$, Theorem \ref{theorem:lierepresentation} specializes to the following 
result when $x_1=t$, $x_2=1$ and $x_i=0$ for $i\ge3$.

\begin{theorem}\label{theorem:lierepresentationk2}
For $n \ge 1$,
\begin{align*}
 \sum_{n\ge 1} \sum_{i=0}^{n-1}\ch \lie_2(n,i)\,t^i=-\Bigl 
(-\sum_{n\ge
1}\dfrac{t^n-1}{t-1} h_{n}(\yy)
\Bigr)^{[-1]}.
\end{align*}
\end{theorem}
For a closely related result obtained using operad theoretic arguments, see 
\cite{DotsenkoKhoroshkin2007}. One should also be able to approach Theorem 
\ref{theorem:lierepresentation} via operad theory.
 
And we obtain a well-known classical result when $x_1=1$, $x_i=0$ for $i\ge2$.
\begin{theorem}
For $n \ge 1$,
\begin{align*}
 \sum_{n\ge 1} \ch \lie(n)=-\Bigl 
(-\sum_{n\ge
1}h_{n}(\yy)
\Bigr)^{[-1]}.
\end{align*}
\end{theorem}

We show that Theorem \ref{theorem:lierepresentation} reduces to 
Theorem \ref{theorem:compositionalinverse} after applying an appropiate specialization. Recall that
$R=\Lambda_{\QQ}$ and 
consider the map $E:\widehat{\Lambda_R} \rightarrow R[[y]]$ defined by:
\begin{align*}
 E(p_i(\yy))=y\delta_{i,1}
\end{align*}
for $i \ge 1$ and extended multiplicatively, linearly and taking the corresponding limits to all of 
$\widehat{\Lambda_R}$. It is not difficult to check that $E$ is a ring homomorphism since $E$ is 
defined on generators. Moreover, we show in the following proposition that the specialization $E$ 
maps plethysm in $\widehat{\Lambda_R}$ to 
composition in $R[[y]]$.

\begin{proposition}For all $F,G \in \widehat{\Lambda_R}$,
 \begin{align*}
  E(F[G])=E(F)(E(G)).
 \end{align*}
\end{proposition}
\begin{proof}
Using the definition of plethysm and using the convention $\xx^k=(x_1^k,x_2^k,\dots)$, 
\begin{align}
 p_{\nu}(\yy) \left [\sum_{\lambda} c_{\lambda}(\xx) p_{\lambda}(\yy) \right ]&= 
\prod_{i=1}^{\ell(\nu)}p_{\nu_i}(\yy)\left [\sum_{\lambda} c_{\lambda}(\xx) p_{\lambda}(\yy)\right 
]\nonumber\\
&=\prod_{i=1}^{\ell(\nu)}\sum_{\lambda} c_{\lambda}(\xx^{\nu_i}) 
p_{\lambda}(\yy^{\nu_i})\nonumber\\
&=\prod_{i=1}^{\ell(\nu)}\sum_{\lambda} c_{\lambda}(\xx^{\nu_i}) \prod_{j=1}^{\ell(\lambda)}  
p_{\lambda_j}(\yy^{\nu_i})\nonumber\\
&=\prod_{i=1}^{\ell(\nu)}\sum_{\lambda} c_{\lambda}(\xx^{\nu_i}) \prod_{j=1}^{\ell(\lambda)}  
p_{\lambda_j\nu_i}(\yy).\label{equation:plethysmexpanded}
\end{align}
 Note that 
$E(p_{\lambda_j\nu_i}(\yy))=y\delta_{\lambda_j\nu_i,1}=y\delta_{\lambda_j,1}\delta_{\nu_i,1}$. Then 
if
$\nu$ has at least one part $\nu_i$ of size greater than $1$, equation 
(\ref{equation:plethysmexpanded}) implies
\begin{align*}
E \left (p_{\nu}(\yy) \left [\sum_{\lambda} c_{\lambda}(\xx) p_{\lambda}(\yy) \right 
]\right)=0=E(p_{\nu}(\yy)) 
\left (E \left (\sum_{\lambda} c_{\lambda}(\xx) p_{\lambda}(\yy) \right )\right ),
\end{align*}
since $E(p_{\nu}(\yy))=0$. If $\nu=(1^{m})$, then 

\begin{align*}
E \left (p_{\nu}(\yy) \left [\sum_{\lambda} c_{\lambda}(\xx) p_{\lambda}(\yy) \right 
]\right)&=
 E \left( \prod_{i=1}^{\ell(\nu)}\sum_{\lambda} c_{\lambda}(\xx^{\nu_i}) \prod_{j=1}^{\ell(\lambda)}
p_{\lambda_j\nu_i}(\yy) \right) \\
&= E \left( \prod_{i=1}^{m}\sum_{\lambda} c_{\lambda}(\xx) 
\prod_{j=1}^{\ell(\lambda)}  
p_{\lambda_j}(\yy) \right)\\
&= \left (E \left(\sum_{\lambda} c_{\lambda}(\xx) 
\prod_{j=1}^{\ell(\lambda)}  
p_{\lambda_j}(\yy) \right) \right )^{m}\\
&= E(p_{(1^m)}(\yy))\left (E \left(\sum_{\lambda} c_{\lambda}(\xx) 
\prod_{j=1}^{\ell(\lambda)}  
p_{\lambda_j}(\yy) \right) \right ).
\end{align*}
We just proved that $E \left (p_{\nu}(\yy) \left [G\right ]\right)=E \left (p_{\nu}(\yy) \right) 
\left [ E(G) \right ]$ for any $G \in \widehat{\Lambda_R}$.
The proof of the proposition follows by extending this result to all of $\widehat{\Lambda_R}$ by 
linearity and taking limits.
\end{proof}
Since $E(p_1(\yy))=y$, we conclude that $E$ is a monoid homomorphism  
\begin{align*}
 (\widehat{\Lambda_R},\text{plethysm},p_1)\rightarrow (R[[y]],\text{composition},y).
\end{align*}
The specialization $E$  can be better understood under the definition of the Frobenius 
characteristic map. Let $V$ be a representation of $\sym_n$ and $\chi^V$ its character, then
 $$\ch(V)=\dfrac{1}{n!}\sum_{\sigma \in \sym_n} \chi^V(\sigma)p_{\lambda(\sigma)}(\yy),$$
where $\lambda(\sigma)$ is the cycle type of the permutation $\sigma \in \sym_n$. 

We have that
\begin{align*}
 E(\ch V) = \dfrac{1}{n!}E \left(\sum_{\sigma \in \sym_n} \chi^V(\sigma)p_{\lambda(\sigma)}(\yy) 
\right) =  \chi_V(id)\dfrac{y^n}{n!}
= \dim V \dfrac{y^n}{n!}.
\end{align*}
In particular since $h_n(\yy)=\ch(\mathbf{1}_{\sym_{n}})$, the Frobenius characteristic of 
the trivial representation of $\sym_n$, we have that
$E(h_n(\yy))=\dfrac{y^n}{n!}$.
Therefore Theorem 
\ref{theorem:lierepresentation} reduces to 
Theorem \ref{theorem:compositionalinverse} after we apply the specialization $E$. 

Theorem \ref{theorem:lierepresentation} gives an 
implicit description of the character for the representation of $\sym_n$ on 
$\lie(\mu)$; Theorem \ref{theorem:lierepresentationk2} gives a description of the character of 
$\lie_2(n,i)$.
Dotsenko and Khoroshkin in \cite{DotsenkoKhoroshkin2007} computed an explicit product formula for 
the $\mathit{SL}_2 \times \sym_n$-character of $\lie_2(n)$. From this one can get the coefficients 
(as polynomials in $t$) of $p_{\lambda}$ in the symmetric function 
$\sum_{i=0}^{n-1}\ch \lie_2(n,i)\,t^i$.
\begin{question}
 Can we find explicit character formulas for  the representation of $\sym_n$ on $\lie(\mu)$ for 
general $\mu \in \wcomp_{n-1}$? What are 
the multiplicities of the irreducibles?
\end{question}

Since $\sum_{\mu \in \wcomp_{n-1}}\ch \lie(\mu)\,\xx^{\mu}$ is a symmetric function in $\xx$ with 
coefficients that are symmetric functions in $\yy$, we can 
write 

\begin{align*}
\sum_{\mu \in \wcomp_{n-1}}\ch \lie(\mu)\,\xx^{\mu}=\sum_{\lambda \vdash n-1}
C_{\lambda}(\yy)e_{\lambda}(\xx),
\end{align*}
 where $C_{\lambda}(\yy)$ is a homogeneous symmetric function of degree $n$ with coefficients in 
$\ZZ$.

By Theorem \ref{theorem:dimensionsofliekstirling}, $E(C_\lambda(\yy))$ equals the number 
$c_{n,\lambda}$ of trees 
$\Upsilon \in Nor_n$ of comb type (or Lyndon type) $\lambda(\Upsilon)=\lambda$.
We propose the following 
conjecture.
\begin{conjecture}\label{conjecture:schurpositive}
 The coefficients $C_{\lambda}(\yy)$ are Schur positive.
 \end{conjecture}
The conjecture basically asserts that $C_{\lambda}(\yy)$ is the Frobenius characteristic of a
representation of dimension $c_{n,\lambda}$. 
An approach to proving the conjecture is to find such a representation.

\section{Related work}
In a subsequent paper we generalize the results presented here further. We consider a more 
general family of weighted partition posets with different restrictions on the sizes of the 
blocks. Then the corresponding generalization of the isomorphism of Theorem \ref{theorem:liekiso} 
allows us to study generalizations of the free multibracketed Lie algebras. These 
generalizations 
include for example multibracketed versions of free Lie $k$-algebras closely related to the ones 
studied by Hanlon and Wachs in \cite{HanlonWachs1995}.

\bibliographystyle{abbrv}
\bibliography{kweighted_partitions}

\end{document}